\documentclass[11pt,a4paper]{article}
\usepackage[margin=2.5cm]{geometry}
\usepackage[utf8]{inputenc}
\usepackage[T1]{fontenc}
\usepackage{amsfonts}
\usepackage{amsthm}
\usepackage{amsfonts,amsmath,amssymb,mathabx}
\usepackage{xcolor}

\usepackage[scaled=0.90]{helvet} 
\usepackage[english]{babel}
\usepackage[T1]{fontenc}
\usepackage[utf8]{inputenc}
\usepackage{amsmath}
\usepackage{graphicx}
\usepackage{geometry}
\usepackage{amsfonts}
\usepackage{amssymb}
\usepackage{amsthm}
\usepackage{xcolor}
\usepackage{mathrsfs}
\usepackage{titling}
\usepackage{abstract}
\usepackage{mathrsfs}
\usepackage{mathtools}
\usepackage{pythontex} 
\usepackage{csquotes}
\usepackage[toc]{appendix}
\usepackage{cite}

\newcommand{\Lad}{\mathop{\text{Lad}}}
\newcommand{\uab}{\mathop{u_{\rho}(a,b)}}

\newcommand{\Dr}{\mathop{\Delta_{\rho}}}

\newcommand{\abs}{|\cdot|}

\newcommand*{\N}{\mathbb{N}}
\newcommand*{\Z}{\mathbb{Z}}

\newcommand*{\R}{\mathbb{R}}
\newcommand*{\C}{\mathbb{C}}

\newcommand{\ceq}{\coloneqq}

\numberwithin{equation}{section}

\usepackage{thmtools}

\usepackage{thmtools, thm-restate}

\numberwithin{equation}{section}

\newtheorem{thm}{Theorem}[section]
\newtheorem{lemma}[thm]{Lemma}
\newtheorem{prop}[thm]{Proposition}
\newtheorem{cor}[thm]{Corollary}
\theoremstyle{definition}
\newtheorem{de}[thm]{Definition}
\newtheorem{rmk}[thm]{Remark}
\newtheorem{con}[thm]{Conjecture}

\usepackage{xpatch}
\usepackage{blindtext}
\makeatletter
\xpatchcmd{\titlepage}{\@restonecolfalse\newpage}{\@restonecolfalse}{}{}
\xpatchcmd{\endtitlepage}{\if@restonecol\twocolumn \else \newpage \fi}{\if@restonecol\twocolumn \else  \fi}{\typeout{success}}{\typeout{fail}}
\makeatother

\def\Z{\mathbb{Z}}

\title{\large Irreducibility of the parabolic induction of essentially Speh representations and a representation of Arthur type over a $p$-adic field}
\date{}

\usepackage{hyperref}
\hypersetup{
    colorlinks,
    citecolor=black,
    filecolor=black,
    linkcolor=black,
    urlcolor=black
}
\hypersetup{linktocpage}

\begin{document}

\numberwithin{thm}{section}
\numberwithin{de}{section}
\numberwithin{lemma}{section}
\numberwithin{cor}{section}
\numberwithin{ex}{section}
\numberwithin{prop}{section}
\numberwithin{rmk}{section}

\author{Barbara Bošnjak and Alexander Stadler}

\maketitle
\begin{abstract}
    Let $F$ be a $p$-adic field. In this article, we consider representations of split special orthogonal groups $\mathrm{SO}_{2n+1}(F)$ and symplectic groups $\mathrm{Sp}_{2n}(F)$ of rank $n$. We denote by $\pi_1 \times \ldots \times \pi_r \rtimes \pi$ the normalized parabolically induced representation of either. Now let $u_i$ be essentially Speh representations and $\pi$ a representation of Arthur type. We prove that the parabolic induction $u_1 \times \ldots \times u_r \rtimes \pi$ is irreducible if and only if $u_i \times u_j$, $u_i \times u_j^\vee$ and $u_i \rtimes \pi$ are irreducible for all choices of $i\neq j$. If $u_i$ are Speh representations, we determine the composition series of the above parabolically induced representation. Through this, we are able to produce a new collection of unitary representations.
\end{abstract}

\tableofcontents

\pagenumbering{roman}

\pagestyle{headings}
\pagenumbering{arabic}

\section{Introduction}
A central problem of representation theory is the classification of the unitary dual (i.e. the class of irreducible unitary representations) of classical $p$-adic groups. It is still an open question and the main motivation for the problems solved in this article.
The groups of our interest are split special odd orthogonal groups $\mathrm{SO}_{2n+1}(F)$ or symplectic groups $\mathrm{Sp}_{2n}(F)$ of rank $n$, which we both denote by $G_n$. 
Here $F$ stands for a $p$-adic field.
The unitary dual of these groups $G_n$ is explored in various directions by representation theorists. To mention some of them, we note that, for Levi subgroups of corank at most 3, Tadić described the irreducible unitary subquotients of parabolically induced representations of $G_n$ in \cite{T3}.
The \textit{generic dual} of $G_n$ has been determined by Lapid, Muić and Tadić in \cite{LMT}. 
The unramified unitary dual of $G_n$ was determined by Muić and Tadić in \cite{TadMu}.
In \cite{Matic1}, Matić constructed a class of unitarizabile non-tempered representations using the Aubert involution. 

The relation between the question of unitarizability of representations of classical $p$-adic groups $G_n$ and Arthur packets is studied by Tadić in \cite{T6}, and further explored in \cite{big} by Hazeltine, Jiang, Liu, Lo and Zhang.
Considering a natural topology on unitary duals, the aforementioned authors state conjectures linking Arthur type representations with isolated representations and the representations appearing at the critical points, in this way indicating that the representations of Arthur type should play an important role in the classification of the unitary dual of the group $G_n$.
The notion of a representation $\pi$ of \textit{Arthur type} was introduced by Arthur in \cite[Chapter 8]{Arthur} and says that $\pi$ is the local factor of a square-integrable automorphic representation and in particular, unitary.
Closely related to Arthur type representations are the \textit{essentially Speh representations}.
They are the main building blocks in the classification of unitary irreducible representations of $\mathrm{GL}_n(F)$.
This classification was a major milestone, achieved by Tadić in 1986 \cite{T2}.

Parabolic induction is an important tool in the smooth complex representation theory of the classical groups $G_n$, which is a method to construct representations of a group from representations of its Levi subgroups. 
A natural step towards understanding the unitary dual of the groups $G_n$ is to consider the \textit{parabolically induced} representations of the tensor product of  essentially Speh and Arthur type representations. 
For representations $\pi_1, \ldots, \pi_r$ and $\pi_0$ of the groups $\mathrm{GL}_{d_1}(F), \ldots, \mathrm{GL}_{d_r}(F)$ and $G_{n_0}$ respectively, we denote by $\pi_1 \times \cdots \times \pi_r \rtimes \pi_0$ the normalized parabolically induced representation.
In this article, we are interested in an irreducibility criterion for this representation. The natural conjecture, which is well-known by the experts (however, we haven't seen it written down before), is the following:

\begin{con} \label{conjecture}
    Let $\pi_1,\ldots,\pi_r$ be irreducible representations of $\mathrm{GL}_{d_1}(F),...,\mathrm{GL}_{d_r}(F)$ and let $\pi$ be an irreducible representation of $G_{n_0}$. Then the parabolically induced representation $\pi_1\times\ldots\times \pi_r\rtimes\pi$ is irreducible if and only if 
\begin{gather*}
    \begin{cases}
    \pi_i\times \pi_j  \\
    \pi_i\times \pi_j^\vee \\
    \pi_i\rtimes\pi 
    \end{cases} 
    \textit{are irreducible for}\ 1 \leq i,j \leq r,\ i\neq j.
\end{gather*}
\end{con}
Here $\pi_j^\vee$ denotes the \textit{contragredient representation} of $\pi_j$. A similar conjectural criterion exists for the parabolically induced representation $\pi_1 \times \ldots \times \pi_r$ of $\mathrm{GL}_n(F)$ and it is known for the case where all but at most two of the representations $\pi_i$ are square irreducible (\cite[Corollary 4.3]{Cyclic}). We will prove Conjecture \ref{conjecture} when $\pi_i=u_i$ are \textit{essentially Speh} and when $\pi$ is of \textit{Arthur type} (defined in Section \ref{a-packs}).

Indeed, the next theorem is our main result:

\begin{restatable}{thm}{IRR}
\label{thm:IRR}
Let $u_1,\ldots,u_r$ be essentially Speh representations of $\mathrm{GL}_{d_1}(F),...,\mathrm{GL}_{d_r}(F)$ and let $\pi$ be an irreducible representation of $G_{n_0}$ of Arthur type. Then the parabolically induced representation $u_1\times\ldots\times u_r\rtimes\pi$ is irreducible if and only if 
\begin{gather} \label{irr} \tag{Irr}
    \begin{cases}
    u_i\times u_j  \\
    u_i\times u_j^\vee \\
    u_i\rtimes\pi 
    \end{cases}
    \text{are irreducible for}\ 1 \leq i,j \leq r,\ i\neq j. 
\end{gather}
\end{restatable}

In case that the $u_i$'s are actually Speh representations, one can say even more. Besides the irreducibility result in this special case (which is also the first step in the proof of the Theorem \ref{thm:IRR} and carried out in Chapter \ref{Arthur}), we can also determine the composition series of the representation $u_1\times\ldots\times u_r\rtimes\pi $ (see Corollary \ref{Nindex}). The key to understanding representations of Arthur type and Arthur packets is the combinatorial construction by Mœglin \cite{Mœglin}, which has been refined by Xu \cite{Xu1}, \cite{Xu2} and Atobe \cite{A1}. We rely heavily on Atobe's work and simplify some notation, to obtain the desired decomposition. 

In the general case, where $u_i$ are essentially Speh representations, we extensively use two articles, which are due to Tadić \cite{T1} and Atobe \cite{A2}.
Namely, there are explicit criteria to check the left-hand side of (\ref{irr}). 
Tadić proved in \cite[Theorem 1.1]{T1} a criterion for the reducibility of the induced representation from two essentially Speh representations. 
In \cite{A2}, Atobe described the socle, i.e. the  maximal semisimple subrepresentation, of a representation of $G_n$ of the form $u|\cdot|^s\rtimes\pi$, where $u$ is a Speh representation. We generalize these methods to several essentially Speh representations. Our main tool are the Jacquet modules, mostly in the form of derivatives (see Subsection \ref{derivatives}). 

Let us state a corollary of Theorem \ref{thm:IRR} in a form that emphasizes its role in the classification of the unitary dual of the groups $G_n$.
We recall that Tadić's classification states that every irreducible unitary representation of a general linear group is of the form $u_1\times\cdots\times u_r$, for some essentially Speh representations $u_1,\ldots,u_r$. For classical groups $G_n$, one gets the following unitarity result due to Deformation (\cite[Chapter 2]{LMT}):

\begin{cor}
    For $1\leq i \leq r$ let $\rho_i$ be self-dual supercuspidal representations. Let $u_i \coloneqq u_{\rho_i}(a_i,b_i)$ be Speh representations of $\mathrm{GL}_{n_1}(F),\ldots ,\mathrm{GL}_{n_r}(F)$ respectively and let $\pi$ be an irreducible representation of $G_{n_0}$ of Arthur type. Let $s_i:[0,1] \rightarrow \R$ with $s_i(0)=0$ be continuous for $1\leq i\leq r$. We define 
    \begin{align*}
        \pi_t \ceq u_1|\cdot|^{s_1(t)}\times\ldots\times u_r|\cdot|^{s_r(t)}\rtimes\pi
    \end{align*}
    as a representation of $G_{n_0+n_1+\ldots +n_r}$. If for $0\leq t <1$ it holds that
    \begin{gather}
    \begin{cases}
    u_i|\cdot|^{s_i(t)} \times u_j|\cdot|^{ s_j(t)}  \\
    u_i|\cdot|^{s_i(t)} \times u_j^\vee|\cdot|^{- s_j(t)}  \\
    u_i|\cdot|^{s_i(t)} \rtimes\pi 
    \end{cases}
    \text{are irreducible for} \ 1 \leq i,j \leq r,\ i\neq j, 
\end{gather}
    then $\pi_t$ is unitarizable for all $t\in [0,1)$ and so is every subquotient of $\pi_1$.
\end{cor}
Note that in \cite[Chapter 2]{LMT} the representation $\pi_t$ is also required to be hermitian, but this follows from irreducibility in this case, because 
\begin{align*}
    \pi_t^+ \cong u_1^\vee|\cdot|^{-s_1(t)}\times\ldots\times u_r^\vee|\cdot|^{-s_r(t)}\rtimes\pi \cong \pi_t.
\end{align*}

For a long time now, the endoscopic classification of Arthur \cite{Arthur} relied on results that were expected to hold, but not yet proven. Recently, this gap has been closed in \cite{intertwine} and the results presented here are only dependent on the validity of the twisted weighted fundamental Lemma (\cite[Section II.4.4]{twisted}).
One would expect similar results to hold for other classical groups than the ones treated here ($\mathrm{SO}_{2n+1}(F)$ and $\mathrm{Sp}_{2n}(F)$). However, the necessary methods are so far only available in these cases, due to Atobe's work.

Let us describe the organization of this article into chapters:
In Chapter \ref{second} we define notations used throughout the article and state the main theorems in the surrounding theory of the problem.
In Chapter \ref{Arthur}, we introduce the notion of extended multi-segments and explain how parabolic induction of Speh representations and Arthur type representations behaves in terms of extended multi-segments.
The results behind the proof of Theorem \ref{thm:IRR} are developed in Chapter \ref{main}. 
It is organized by different types of essentially Speh representations appearing in the induction.
Finally in the Appendix, we provide some proofs omitted in Chapter \ref{Arthur} and give a formula for the Aubert dual of representations of Arthur type.

This work is supported (in part) by the Croatian Science Foundation under the project number HRZZ-IP-2022-10-4615.
The authors would also like to thank Alberto Mínguez for his guidance and support throughout this project. Moreover, they are grateful to Hiraku Atobe, Johannes Droschl and Alexander Hazeltine for their helpful correspondence.

\section{Notation and Preliminaries}
\label{second}

\subsection{Notations}
Let $F$ be a $p$-adic field and let $G$ be the $F$-points of a reductive group. Let $P$ denote a parabolic subgroup of $G$. We denote by $\mathrm{Ind}_P^{G}$ the normalized parabolic induction along $P$ and by $\mathrm{Jac}_P$ its left adjoint functor, the Jacquet functor. A representation of $G$ is called \textit{supercuspidal}, if it is not a composition factor of any representation of the form $\mathrm{Ind}_P^{G}(\tau)$, where $P$ is a proper parabolic subgroup of $G$ and $\tau$ is a representation of the corresponding Levi subgroup $M$. We let $\mathrm{Rep}(G)$, $\mathrm{Irr}(G)$, $\mathrm{Irr}_{unit}(G)$, $\mathrm{Cusp}^\perp(G)$ and $\mathrm{Cusp}_{unit}(G)$ denote the set of equivalence classes of smooth $\C$-representations of $G$, and its subsets of irreducible, unitary (irreducible), self-dual supercuspidal and unitary supercuspidal representations respectively. Furthermore, let $\mathscr{R}(G)$ denote the Grothendieck group of smooth representations of finite length of $G$.
Let $\mathrm{GL}_n(F)$ denote the general linear group of rank $n$ and let $G_n$ either denote the split special orthogonal group $\mathrm{SO}_{2n+1}(F)$ or the symplectic group $\mathrm{Sp}_{2n}(F)$ of rank $n$ over $F$. We define $\mathscr{R}(\mathrm{GL})=\oplus_{n\ge0}\mathscr{R}(\mathrm{GL}_n(F))$ and $\mathscr{R}(G)=\oplus_{n\ge0}\mathscr{R}(G_n)$. We let $\text{Irr}(GL) \ceq \bigcup_{n\ge0} \text{Irr}(\mathrm{GL}_n(F))$ and similarly $\text{Irr}(G) \ceq \bigcup_{n\ge0} \text{Irr}(G_n)$.
For $\pi\in \mathrm{Rep}(GL_n(F))$ or $\pi\in \mathrm{Rep}(G_n)$, we denote by $[\pi]$ its semisimplification, i.e. the direct sum of its composition factors.
If $\pi_1\in\mathrm{Irr}(G)$ is a subquotient of $\pi$, we write $\pi_1\le \pi$. 
By $\mathrm{soc}(\pi)$ we denote the socle of $\pi$, which is the maximal semisimple subrepresentation of $\pi$ and we denote by $\pi^\vee$ the contragradient (or dual) representation of $\pi$. By $\overline{\pi}$, we denote the complex conjugate representation of $\pi$, which acts on the vector space $(\overline{V_\pi},\overline{\cdot})$, where $\overline{V_\pi}=V_\pi$, $\alpha \overline{\cdot} v = \overline{\alpha} v$ for $\alpha \in \C$ and $v\in V_\pi$ and $\overline{\pi}(g)(v)=\pi(g)(v)$. Furthermore, we denote by $\pi^+ = \overline{\pi^\vee} \cong \overline{\pi}^\vee$ the hermitian contragredient of $\pi$. The irreducible representation $\pi$ is \textit{hermitian} (i.e. there exists a non-degenerate $G$-invariant hermitian form on $\pi$) if and only if $\pi^+\cong \pi$. The representation $\pi$ is \textit{unitary}, if it is hermitian and the hermitian form is definite.

\subsection{General Linear Groups}
\label{glchapter}

Let us first consider the general linear group $\mathrm{GL}_n(F)$. Let $|\cdot|$ be the norm on $F$. We use the same notation for the map
$$
|\cdot|: \mathrm{GL}_n(F) \rightarrow \R_{\geq 0},
$$
which is given by $|X|=|\det X|$ for $X\in \mathrm{GL}_n(F)$.
Fix the $F$-rational Borel subgroup consisting of upper-triangular matrices of $\mathrm{GL}_n(F)$.
Then the Levi subgroup $M$ of a standard parabolic subgroup $P$ of $\mathrm{GL}_n(F)$ has the form $M\cong\mathrm{GL}_{d_1}(F) \times \cdots \times \mathrm{GL}_{d_r}(F)$, for some integers $d_1,\ldots,d_r$ with $d_1 + \ldots +d_r =n$. For smooth representations $\pi_1, \ldots, \pi_r$ of the groups $\mathrm{GL}_{d_1}(F)$, $\ldots$, $\mathrm{GL}_{d_r}(F)$ respectively, we denote the parabolic induction of their tensor product by
\[
\pi_1 \times \cdots \times \pi_r \ceq \operatorname{Ind}_{P}^{\mathrm{GL}_n(F)} (\pi_1 \boxtimes \cdots \boxtimes \pi_r),
\]
where $\pi_i$ is a smooth representation of $\mathrm{GL}_{d_i}(F)$. We will also use the notation 
\begin{align*}
    \pi^k \ceq \pi \times \overset{k\text{ times}}{\cdots} \times \pi.
\end{align*}
We write $m :\mathscr{R}(\mathrm{GL})\otimes \mathscr{R}(\mathrm{GL})\to \mathscr{R}(\mathrm{GL})$ for the map that performs this multiplication, i.e. $m: \pi_1 \otimes \pi_2 \mapsto \pi_1 \times \pi_2$. 
Let $P_{(k_1,\ldots, k_j)}$ denote the maximal parabolic subgroup of $\mathrm{GL}_n(F)$ with Levi subgroup isomorphic to $\mathrm{GL}_{k_1}(F)\times \ldots \times  \mathrm{GL}_{k_j}(F)$ and denote by $\mathrm{Jac}_{(k_1,\ldots,k_j)}=\mathrm{Jac}_{P_{(k_1,\ldots,k_j)}}^{\mathrm{GL}_n(F)}$ the corresponding Jacquet functor. We define a map $m^*:\mathscr{R}(\mathrm{GL})\to \mathscr{R}(\mathrm{GL})\otimes \mathscr{R}(\mathrm{GL})$ by 
$$m^*(\pi)=\sum_{n\geq 0} \sum_{n_1+n_2=n}[\text{Jac}_{(n_1,n_2)}(\pi)]$$
for $\pi \in \mathscr{R}(GL_n(F))$.
The Geometric Lemma (see \cite{BernZel1}) implies that the map $ m^* $ is multiplicative, i.e. for $\pi_1,\pi_2\in \mathscr{R}(GL)$ we have $m^*(\pi_1\times\pi_2)=m^*(\pi_1)\times m^*(\pi_2)$. 
The multiplication on the right hand side is determined by $(\tau_1\otimes\tau_2)\times(\tau_1'\otimes\tau_2')=(\tau_1\times\tau_1')\otimes(\tau_2\times\tau_2')$.

A \textit{segment} $[x, y]_\rho$ (we sometimes drop the index $\rho$ if it is clear from the context) is a set of supercuspidal representations of the form
\[
[x, y]_\rho \ceq \{\rho | \cdot |^x, \rho | \cdot |^{x-1}, \ldots, \rho | \cdot |^y\},
\]
where $\rho \in \mathrm{Cusp}_\mathrm{unit}(\mathrm{GL}_n(F))$, $\rho | \cdot |^z$ denotes the twist of the representation $\rho$ by the character $|\cdot|^z$ and $x, y \in \mathbb{R}$ such that $x - y \in \mathbb{Z}$ and $x \geq y$. Note that for every supercuspidal representation $\rho$, there exists a $z\in \R$ such that $\chi_\rho|\cdot|^z$ is unitary, where $\chi_\rho$ is the central character of $\rho$. But this implies that $\rho|\cdot|^z$ itself is unitary. Hence we can restrict the definition of segments to unitary $\rho$. We call the set 
\begin{align*}
    \Z_{\rho}\ceq \{\rho|\cdot|^z\ |z\in \Z\}
\end{align*}
the \text{line} of $\rho$. For the segment $\Delta \ceq [x,y]_\rho$, we denote
\begin{align*}
    \overset{\leftarrow}{\Delta} &\ceq [x-1,y-1]_\rho \\
    \min(\Delta) &\ceq y \\
    \max(\Delta)&\ceq x.
\end{align*}
We define the \textit{Steinberg representation} $\Delta_\rho[x, y]$ as the unique irreducible subrepresentation of
\[
\rho | \cdot |^x \times \cdots \times \rho | \cdot |^y.
\]
This is an essentially discrete series representation of $\mathrm{GL}_{d(x-y+1)}(F)$. Moreover, we define $Z_\rho[y, x]$ as the unique irreducible quotient of the same induced representation. By convention, we set $\Delta_\rho[x, x+1]$ and $Z_\rho[x+1, x]$ to be the trivial representation of the trivial group $\mathrm{GL}_0(F)$.

The Langlands classification (\cite[Section 2]{A2}) for $\mathrm{GL}_n(F)$ says that every $\pi \in \mathrm{Irr}(\mathrm{GL}_n(F))$ is the unique irreducible subrepresentation of $\Delta_{\rho_1}[x_1, y_1] \times \cdots \times \Delta_{\rho_r}[x_r, y_r]$, for some $\rho_i \in \mathrm{Cusp}_{unit}(\mathrm{GL}_{d_i}(F))$ for $i = 1, \ldots, r$ and segments $[x_1,y_1]_{\rho_1},\ldots ,[x_r,y_r]_{\rho_r}$ such that $x_1 + y_1 \leq \cdots \leq x_r + y_r$ and $\sum_{i=1}^r d_i(x_i-y_i+1)=n$. In this case, we write
\[
\pi = L(\Delta_{\rho_1}[x_1, y_1], \ldots, \Delta_{\rho_r}[x_r, y_r]).
\]
We call the segments $[x_1, y_1]_{\rho_1}, \ldots, [x_r, y_r]_{\rho_r}$ the \textit{Langlands data} of $\pi$. 
We will also use the following notation: We call the formal sum $\mathfrak{m}=\Delta_1 + \ldots +\Delta_r$ of segments $\Delta_1,\ldots,\Delta_r$, where $\Delta_i=[x_i, y_i]_{\rho_i}$ as above, a \textit{multi-segment}. Then we also denote $\pi$ with $L(\mathfrak{m})$.
\begin{de}
    Let $\Delta=[x,y]_{\rho}$ and $\Delta'=[x',y']_{\rho'}$ be two segments. We say that $\Delta$ and $\Delta'$ are linked if $\Delta\cup \Delta'$ forms a segment, but neither $\Delta\subset \Delta'$ nor $\Delta'\subset \Delta$. If $\Delta$ and $\Delta'$ are linked and $\rho'|\cdot|^{y'}=\rho|\cdot|^{y+i}$ for $i>0$, then we say that $\Delta$ precedes $\Delta'$ and write $\Delta\prec \Delta'$.
\end{de}
Now fix some $\rho \in \mathrm{Cusp}_{unit}(\mathrm{GL}_n(F))$. For positive integers $a$ and $b$, we define $\uab$ to be the unique irreducible subrepresentation of 
$$ \Dr[\tfrac{a-1}{2},-\tfrac{a-1}{2}]\abs^{-\frac{b-1}{2}}\times\ldots\times\Dr[\tfrac{a-1}{2},-\tfrac{a-1}{2}]\abs^{\frac{b-1}{2}}. $$
In terms of the Langlands classification, 
\begin{align} \label{speh}
    \uab = L(\Dr[\tfrac{a-b}{2},-\tfrac{a+b}{2}+1],\ldots,\Dr[\tfrac{a+b}{2}-1,\tfrac{b-a}{2}]).
\end{align}
We call $\uab$ a \textit{Speh representation}. It is an irreducible unitary representation. We often automatically set $A \ceq \frac{a+b}{2} - 1$ and $B \ceq \frac{a-b}{2}$.
An \textit{essentially Speh representation} is a representation of the form
$$
u_\rho(a,b)|\cdot|^s,
$$
for a real number $s$.
We see that if $\pi=L(\Delta_{\rho}[x_1, y_1], \ldots, \Delta_{\rho}[x_n, y_n])$ is an essentially Speh representation, it is determined by the numbers $x_1,y_1,x_n$ and $y_n$.
Therefore, in the following definition, we introduce another notation for essentially Speh representations, used by Tadić in \cite{T1}:
\begin{de}
An essentially Speh representation $L(\Dr[x_1,y_1],\ldots,\Dr[x_n,y_n])$, where $y_n=y_1+n-1$ and $x_n=x_1+n-1$, is denoted 
$$u_{ess}\begin{pmatrix}
y_1 \hspace{3mm} & x_1 \hspace{3mm} \\
\hspace{3mm} y_n & \hspace{3mm} x_n
\end{pmatrix}^{(\rho)}.$$
\end{de}
Using this notation, we determine the Langlands parameters for the essentially Speh representation $\uab|\cdot|^s$ and two closely related representations $u_{\rho}(2s,b)|\cdot|^{\frac{a}{2}}$ and $u_{\rho}(a-2s,b)$, that will often appear in Chapter \ref{main}:
\begin{itemize}
    \item $\uab|\cdot|^s = u_{ess}\begin{pmatrix}
-A+s \hspace{3mm} &  B+s \hspace{3mm} \\
\hspace{3mm} -B+s & \hspace{3mm} A+s
\end{pmatrix}^{(\rho)}$
    \item $u_{\rho}(2s,b)|\cdot|^{\frac{a}{2}} = u_{ess}\begin{pmatrix}
B-s+1 \hspace{3mm} &  B+s \hspace{3mm} \\
\hspace{3mm} A-s+1 & \hspace{3mm} A+s
\end{pmatrix}^{(\rho)}$
    \item $u_{\rho}(a-2s,b) = u_{ess}\begin{pmatrix}
-A+s \hspace{3mm} &  B-s \hspace{3mm} \\
\hspace{3mm} -B+s & \hspace{3mm} A-s
\end{pmatrix}^{(\rho)}$
\end{itemize}
For segments $[x_i,y_i]_{\rho}$ for $i=1,\ldots,k$, we call $\pi=L(\Dr[x_1,y_1],\ldots,\Dr[x_k,y_k])$ a \textit{ladder representation}, if $x_1<\ldots<x_k$ and $y_1<\ldots<y_k$.
The class of ladder representations, which contains essentially Speh representations, was defined in \cite{LapidMinguez}. 
A description of the Jacquet modules of ladder representations is obtained in \cite{KretLapid} and we will vastly use it.
Recall the map $m^*:\mathscr{R}(\mathrm{GL})\to \mathscr{R}(\mathrm{GL})\otimes \mathscr{R}(\mathrm{GL})$ from above.
From \cite[Theorem 2.1]{KretLapid} we get
\begin{gather}
    \label{lad}
    m^*(\pi)= 
    \displaystyle \sum_{\Lad(\pi)} L(\Dr[x_1,c_1+1],\ldots,\Dr[x_k,c_k+1])\otimes  \\ \otimes L(\Dr[c_1,y_1],\ldots,\Dr[c_k,y_k]), \nonumber
\end{gather}
where $\Lad(\pi)$ denotes the set of all $k$-tuples $(c_1,\ldots,c_k)$ of real numbers for which the following conditions hold:
\begin{itemize}
    \item $c_1<\ldots<c_k$,
    \item $y_i-1 \le c_i \le x_i$ for $i=1,\ldots,k$,
    \item $c_i-x_i\in\Z$ for $i=1,\ldots,k$.
\end{itemize}
We define
$$ \begin{pmatrix}
    u_1 \hspace{3mm} & v_1 \hspace{3mm} \\
    \hspace{3mm} x_1 & \hspace{3mm} y_1
    \end{pmatrix} <_{\text{strong}} \begin{pmatrix}
    u_2 \hspace{3mm} & v_2 \hspace{3mm} \\
    \hspace{3mm} x_2 & \hspace{3mm} y_2
    \end{pmatrix} \Longleftrightarrow u_1<u_2,~ v_1<v_2,~ x_1<x_2,~ y_1<y_2. $$
We will vastly use the following theorem on the reducibility of a representation induced from two essentially Speh representations.
\begin{thm} (\cite[Theorem 1.1]{T1}) \label{essred} \\
Let $n, n_1$ and $n_2$ be positive integers such that $n=n_1+n_2$.
    Let $\pi_1$ and $\pi_2$ be essentially Speh representations of $GL_{n_1}(F)$ and $GL_{n_2}(F)$. Let $\rho_1,\rho_2\in\mathrm{Cusp}_{unit}(GL_{n}(F))$ and $u_i, v_i, x_i, y_i \in \mathbb{Z}$, such that 
    $$ \pi_i=u_{ess}\begin{pmatrix}
    u_i \hspace{3mm} & v_i \hspace{3mm} \\
    \hspace{3mm} x_i & \hspace{3mm} y_i
    \end{pmatrix}^{(\rho_i)} $$
    for $i=1,2$. Then the parabolically induced representation $\pi_1 \times \pi_2$ reduces if and only if 
    \begin{itemize}
        \item[(1)] $[u_1,y_1]_{\rho_1}\cup [u_2,y_2]_{\rho_2}$ is a segment,
        \item[(2)] $$ \begin{pmatrix}
    u_1 \hspace{3mm} & v_1 \hspace{3mm} \\
    \hspace{3mm} x_1 & \hspace{3mm} y_1
    \end{pmatrix} <_{\text{strong}} \begin{pmatrix}
    u_2 \hspace{3mm} & v_2 \hspace{3mm} \\
    \hspace{3mm} x_2 & \hspace{3mm} y_2
    \end{pmatrix}$$
    or 
    $$ \begin{pmatrix}
    u_1 \hspace{3mm} & v_1 \hspace{3mm} \\
    \hspace{3mm} x_1 & \hspace{3mm} y_1
    \end{pmatrix} >_{\text{strong}} \begin{pmatrix}
    u_2 \hspace{3mm} & v_2 \hspace{3mm} \\
    \hspace{3mm} x_2 & \hspace{3mm} y_2
    \end{pmatrix}.$$
    \end{itemize}
Moreover, if $\pi_1 \times \pi_2$ is irreducible, then  $\pi_1 \times \pi_2 \cong \pi_2 \times \pi_1$. 
\end{thm}
Note that, in particular, this implies that the parabolic induction of two Speh representations is always irreducible. Also, it implies that $u_{\rho_1}(a_1,b_1)|\cdot|^{s_1} \times u_{\rho_2}(a_2,b_2)|\cdot|^{s_2}$ is irreducible if $\rho_1 \not \cong \rho_2$. 
\subsection{Classical Groups}
Next, we recall some notations for representations of the classical groups $G_n$. Fix a $F$-rational Borel subgroup consisting of upper-triangular matrices of $G_n$.
Then the Levi subgroup $M$ of a standard parabolic subgroup $P$ of $G_n$ has the form  $M \cong \mathrm{GL}_{d_1}(F) \times \cdots \times \mathrm{GL}_{d_r}(F) \times G_{n_0}$ such that $n=n_0+(d_1+\ldots+d_r)$. For representations $\pi_1, \ldots, \pi_r$ and $\pi_0$ of the groups $\mathrm{GL}_{d_1}(F), \ldots, \mathrm{GL}_{d_r}(F)$ and $G_{n_0}$ respectively, we denote by
\[
\pi_1 \times \cdots \times \pi_r \rtimes \pi_0 \ceq \operatorname{Ind}_{P}^{G_n} (\pi_1 \boxtimes \cdots \boxtimes \pi_r \boxtimes \pi_0)
\]
the normalized parabolically induced representation.
The Langlands classification (\cite[Section 2.2]{A2}) for $G_n$ says that every $\pi \in \mathrm{Irr}(G_n)$ is the unique irreducible subrepresentation of some parabolically induced representation $\Delta_{\rho_1}[x_1, y_1] \times \cdots \times \Delta_{\rho_r}[x_r, y_r] \rtimes \pi_0$, where
\begin{itemize}
    \item $\rho_i \in \mathrm{Cusp}_{unit}(\mathrm{GL}_{d_i}(F))$ for $i = 1, \ldots, r$;
    \item $x_1 + y_1 \leq \cdots \leq x_r + y_r < 0$;
    \item $\pi_0$ is an irreducible tempered representation of $G_{n_0}$; 
    \item $n=n_0+(d_1+\ldots+d_r)$.
\end{itemize}
In this case, we write
\[
\pi = L(\Delta_{\rho_1}[x_1, y_1], \ldots, \Delta_{\rho_r}[x_r, y_r]; \pi_0),
\]
and call $(\Delta_{\rho_1}[x_1, y_1], \ldots, \Delta_{\rho_r}[x_r, y_r]; \pi_0)$ the Langlands data for $\pi$. 

The Jacquet functor with respect to a parabolic subgroup $P$ is denoted $\mathrm{Jac}_P$.
Let $P_{d}$ denote the maximal parabolic subgroup of $G_n(F)$ with Levi subgroup isomorphic to $\mathrm{GL}_d(F)\times G_{n-d}(F)$. We consider a map $\mu^*:\mathscr{R}(G)\to \mathscr{R}(\mathrm{GL})\otimes \mathscr{R}(G)$ defined by 
\begin{gather} \label{mu1}
    \mu^*(\pi)=\sum_{d=0}^n[\text{Jac}_{P_{d}}(\pi)],
\end{gather}
for $\pi \in \mathscr{R}(G_n)$.
In \cite{Tad5}, M. Tadi\'{c} defined a map $M^*:\mathscr{R}(\mathrm{GL})\to \mathscr{R}(\mathrm{GL})\otimes \mathscr{R}(\mathrm{GL})$ by 
\begin{gather} \label{m1}
    M^*=(m\otimes 1)\circ (\cdot ^\vee \otimes m^*)\circ s \circ m^*
\end{gather}
where $s: \sum x_i\otimes y_i \mapsto \sum y_i\otimes x_i$.
Note that $M^*$ is a multiplicative map.
It is the map used to obtain a description of $\mu^*$ applied to a parabolically induced representation (this is known as the \textit{Geometric Lemma} or more specifically \textit{Tadi\'{c}'s Formula}):
\begin{thm}  (\cite[Theorem 5.4]{Tad5})\label{TadicStruct} \\
For $\pi\in \mathscr{R}(\mathrm{GL})$ and $\sigma\in \mathscr{R}(G)$
$$\mu^*(\pi\rtimes\sigma)=M^*(\pi)\rtimes\mu^*(\sigma), $$
where the product on the right hand side is determined by $(\pi_1\otimes\pi_2)\rtimes(\pi'\otimes\sigma')=(\pi_1\times\pi')\otimes(\pi_2\rtimes\sigma')$.
\end{thm}
For an induced representation $\pi\rtimes\sigma$, where $\pi$ is a ladder representation with Langlands data $(\Delta_{\rho}[x_1, y_1], \ldots, \Delta_{\rho}[x_k, y_k])$, we combine the result of Theorem \ref{TadicStruct} with the description of the image of the map $m^*$ applied to a ladder representation in (\ref{lad}).
We define $\Lad(\pi)'$ as the set of all pairs $\big((c_1,\ldots,c_k),(d_1,\ldots,d_k)\big)$ in $\Lad(\pi)\times\Lad(\pi)$ such that $c_i\le d_i$ for $i=1,2,\ldots,k$. Then we have
\begin{gather}
    \label{struct}
    \mu^*(\pi\rtimes\sigma)= \sum\limits_{\Lad(\pi)'} L(\Delta_{\rho^\vee}[-y_k,-c_k],\ldots,\Delta_{\rho^\vee}[-y_1,-c_1])\times \\
     L(\Dr[x_1,d_1+1],\ldots,\Dr[x_k,d_k+1]) \otimes 
    L(\Dr[d_1,c_1+1],\ldots,\Dr[d_k,c_k+1]) \rtimes \mu^*(\sigma) \nonumber
\end{gather}
up to semisimplification.

In \cite{MVW}, the covariant functor $\mathrm{MVW}:\mathrm{Rep}(G_n) \rightarrow \mathrm{Rep}(G_n)$ is defined. It has the following properties:
\begin{itemize}
    \item If $\pi \in \mathrm{Irr}(G_n)$, then $\pi^{\mathrm{MVW}}=\pi^\vee$,
    \item If $\pi \in \mathrm{Rep}(G_{n_0})$ and $\tau \in \mathrm{Rep}(\mathrm{GL}_d(F))$, then $(\tau \rtimes \pi)^{\mathrm{MVW}} \cong \tau \rtimes \pi^{\mathrm{MVW}}$.
\end{itemize}
We use it to show for example the following: If $\tau \rtimes \pi$ is irreducible for $\tau \in \mathrm{Irr}(GL_d(F))$ and $\pi\in \mathrm{Irr}(G_{n_0})$, then
\begin{align*}
    \tau \rtimes \pi & \cong (\tau^\vee \rtimes \pi^\vee)^\vee \cong (\tau^\vee \rtimes \pi^\vee)^{\mathrm{MVW}} \\
    &\cong \tau^\vee \rtimes \pi.
\end{align*}
\subsection{Derivatives} \label{derivatives}

In this section, we recall the theory of derivatives, established in \cite{AM} and in \cite[Section 3.1]{A2}. It is a computational tool, that allows one to reduce questions about representations of $G$ to representations of a certain subgroup, by utilizing the Geometric Lemma (Theorem \ref{TadicStruct}). First, we recall the notion of derivatives for representations of $\mathrm{GL}_n(F)$ (\cite[Definition 2.3]{A4} and \cite[Section A.1]{A2}):
\begin{de}
    Let $\pi \in \mathscr{R}(\mathrm{GL})$ and let $\sigma$ be either supercuspidal or equal to $Z_\rho[0,1]=L(\{\rho\},\{\rho|\cdot|^1\})$ for some supercuspidal $\rho$. For a non-negative integer $k$, we define the left and right derivatives $L^{(k)}_{\sigma}$ and $R^{(k)}_{\sigma}$ of $\pi$ to be the semisimple representations satisfying 
\begin{gather} \label{gl-deriv}
   m^*(\pi)=\sigma^k\otimes L^{(k)}_{\sigma}(\pi) + \displaystyle\sum_{i} \pi_i\otimes\tau_i,
\end{gather}
where $\pi_i\otimes\tau_i$ are irreducible representations such that $\pi_i\not\cong\sigma^k$ and
\begin{gather} \label{gl-deriv-right}
   m^*(\pi)= R^{(k)}_{\sigma}(\pi) \otimes \sigma^k + \displaystyle\sum_{i} \pi_i\otimes\tau_i,
\end{gather}
where $\pi_i\otimes\tau_i$ are irreducible representations such that $\tau_i\not\cong\sigma^k$
. For $k=1$, we define $L_{\sigma}=L_{\sigma}^{(1)}$ and $R_{\sigma}=R_{\sigma}^{(1)}$. The representation $\pi$ is called left (resp. right) $\sigma$-\textit{reduced} if $L_{\sigma}(\pi)=0$ (resp. $R_{\sigma}(\pi)=0$). When $L_\sigma^{(k+1)}(\pi)=0$ but $L_\sigma^{(k)}(\pi)\neq 0$, we call $L_\sigma^{(k)}(\pi)\neq 0$ the highest derivative and denote it by $L_\sigma^{max}(\pi)$ (analogously for the right derivative).
\end{de}
\begin{rmk}
     If $\sigma$ is supercuspidal, the highest (left or right) derivative $L_{\sigma}^{(k)}(\pi)$ of an irreducible representation $\pi$ is irreducible (\cite[Lemma 2.1.2]{Jantz}). In this case, $\pi$ is the unique irreducible subrepresentation of $\sigma^k \rtimes L_{\sigma}^{(k)}(\pi)$. If $\pi$ is irreducible and left $\rho|\cdot|^1$-reduced, then $L_{Z_\rho[0,1]}^{max}(\pi)=L_{Z_\rho[0,1]}^{(k)}(\pi)$ is irreducible and again $\pi$ is the unique irreducible subrepresentation of $Z_\rho[0,1]^k \rtimes L_{Z_\rho[0,1]}^{(k)}(\pi)$. The highest left and right derivatives are explicitly computable in terms of Langlands data (\cite[Theorem 2.4]{A4}).
\end{rmk}

\begin{de} \label{def1}
Let $\pi\in \mathscr{R}(G)$ and let $\sigma\in \mathrm{Irr}(GL)$ be either a supercuspidal representation or $\sigma =Z_\rho[0,1]$ for a supercuspidal $\rho$. For a non-negative integer $k$, we define the derivative $D^{(k)}_{\sigma}$ of $\pi$ to be the semisimple representation satisfying 
\begin{gather} \label{eq01}
   \mu^*(\pi)=\sigma^k\otimes D^{(k)}_{\sigma}(\pi) + \displaystyle\sum_{i} \pi_i\otimes\pi_i',
\end{gather}
where $\mu^*$ is the map defined in (\ref{mu1}) and $\pi_i\otimes\pi_i'$ are irreducible representations such that $\pi_i\not\cong\sigma^k$. For $k=1$, we define $D_{\sigma}=D_{\sigma}^{(1)}$.
The representation $\pi$ is called $\sigma$-\textit{reduced} if $D_{\sigma}(\pi)=0$. 
\end{de}
As we see in Theorem \ref{TadicStruct}, the map $\mu^*$ is closely related to the map $M^*$. 
This motivates the following definition, analogous to Definition \ref{def1}.
\begin{de}
Let $\pi\in \mathscr{R}(GL)$ and let $\sigma\in \mathrm{Irr}(GL)$ be either a supercuspidal representation or $\sigma =Z_\rho[0,1] $ for a supercuspidal $\rho$. For a non-negative integer $k$, we define the derivative $M^{(k)}_{\sigma}$ of $\pi$ to be the semisimple representation satisfying 
\begin{gather*}
   M^*(\pi)=\sigma^k\otimes M^{(k)}_{\sigma}(\pi) + \displaystyle\sum_{i} \pi_i\otimes\pi_i',
\end{gather*}
where $M^*$ is the map defined in (\ref{m1}) and $\pi_i\otimes\pi_i'$ are irreducible representations such that $\pi_i\not\cong\sigma^k$. For $k=1$, we define $M_{\sigma}=M_{\sigma}^{(1)}$.
\end{de}
    Note that we do \textbf{not} define the notion of being $\sigma$-reduced for $M_{\sigma}$, because this already has a meaning in terms of left and right derivatives. However, the following holds:
\begin{lemma} \label{reduced}
    Let $\pi\in \mathscr{R}(GL_n)$ and let $\rho \in \mathrm{Irr}(GL)$ be supercuspidal. If $M_\rho(\pi)=0$, then for every constituent $\tau \otimes \sigma \leq  M^*(\pi)$, the representation $\tau$ is left $\rho$-reduced.
\end{lemma}
\begin{proof}
    Every such constituent $\tau$ is a subquotient of $\tau_3^\vee \times \tau_1$, where $\tau_1 \otimes \tau_2 \otimes \tau_3 \leq [\mathrm{Jac}_{(d_1,d_2,d_3)}(\pi)]$ with $d_1+d_2+d_3=n$. If $\tau$ were not $\rho$-reduced, we would have $\rho \otimes \delta \leq [\mathrm{Jac}_{(d,d_1+d_3-d)}(\tau_3^\vee \times \tau_1)]$, which means (due to the Geometric Lemma for $\mathrm{GL}_n(F)$ \cite[Section 1.2]{LapidMinguez2}) that $\rho \otimes \delta' \leq [\mathrm{Jac}_{(d,d_3-d)}(\tau_3^\vee)]$ or $\rho \otimes \delta' \leq [\mathrm{Jac}_{(d,d_1-d)}(\tau_1)]$. In the first case, a constituent $1 \otimes \delta'' \otimes \rho^\vee$ appears in $[\mathrm{Jac}_{(0,n-d,d)}(\pi)]$, in the latter case a constituent $\rho \otimes \delta'' \otimes 1$ in $[\mathrm{Jac}_{(d,n-d,0)}(\pi)]$. Both contradict $M_\rho(\pi)=0$.
\end{proof}
\begin{de} \label{maxder}
Let $\pi_1 \in \mathscr{R}(G)$ and let $\sigma\in \mathrm{Irr}(GL)$ be either a supercuspidal representation or $\sigma =Z_\rho[0,1] $ for a self-dual supercuspidal $\rho$. 
If $D^{(k)}_{\sigma}(\pi_1)\neq 0$, but $D^{(k+1)}_{\sigma}(\pi_1)=0$, we call $D^{(k)}_{\sigma}(\pi_1)$ the maximal $\sigma$-derivative of $\pi_1$ and set $D^{\max}_{\sigma}(\pi_1)=D^{(k)}_{\sigma}(\pi_1)$.

Analogously, for $\pi_2\in \mathscr{R}(GL)$, if $M^{(k)}_{\sigma}(\pi_2)\neq 0$, but $M^{(k+1)}_{\sigma}(\pi_2)=0$, we call $M^{(k)}_{\sigma}(\pi_2)$ the maximal $\sigma$-derivative of $\pi_2$ and set $M^{\max}_{\sigma}(\pi_2)=M^{(k)}_{\sigma}(\pi_2)$. 
\end{de}
Let us recall two important facts about the maximal derivatives from Section 3 in \cite{AM}:
For a representation $\pi\in\mathrm{Irr}(G)$, a supercuspidal representation $\rho\in\mathrm{Irr}(GL)$ with $\rho \not \cong \rho^\vee$, the maximal derivative $D^{\max}_{\rho}(\pi)$ is an irreducible representation. Moreover, if $D^{\max}_{\rho}(\pi)=D^{(k)}_{\rho}(\pi)$, then $\pi$ is the unique irreducible subrepresentation of $\rho^k\rtimes D^{(k)}_{\rho}(\pi)$ (also, $k$ is then the maximal integer such that there exists an irreducible $\tau$ such that $\pi \hookrightarrow \rho^k \rtimes \tau$).

Assume for the rest of this section that $\rho$ is a self-dual supercuspidal representation. If $\pi$ is $\rho|\cdot|^1$-reduced, i.e. $D_{\rho|\cdot|^1}(\pi)=0$, and $D^{\max}_{Z_\rho[0,1]}(\pi)=D^{(k)}_{Z_\rho[0,1]}(\pi)$, then $D^{(k)}_{Z_\rho[0,1]}(\pi)$ is irreducible and $\pi$ is the unique irreducible subrepresentation of $Z_\rho[0,1]^k\rtimes D^{\max}_{Z_\rho[0,1]}(\pi)$ (\cite[Proposition 3.7]{AM}).

We often consider a certain composition of derivatives, for which we adapt the following notation of Definition 3.3 in \cite{A2}.
\begin{de} \label{def2}
For a segment $[x,y]_{\rho}$, we denote by $D^{\max}_{\rho|\cdot|^{y},\ldots,\rho|\cdot|^{x}}$ the composition of derivatives
$$
\begin{cases}
D^{\max}_{\rho|\cdot|^{x}}\circ\ldots\circ D^{\max}_{\rho|\cdot|^{y}} \text{, if }\rho\not\in[x,y]_{\rho} \\
D^{\max}_{\rho|\cdot|^{x}}\circ\ldots\circ D^{\max}_{\rho|\cdot|^{2}}\circ\left( D^{\max}_{Z_\rho[0,1]} \circ D^{\max}_{\rho|\cdot|^{1}} \right)\circ D^{\max}_{\rho|\cdot|^{-1}} \circ \ldots \circ D^{\max}_{\rho|\cdot|^{y}} \text{, if }\rho\in[x,y]_{\rho} \text{ and}\ x>0.
\end{cases}
$$
Analogously, we define $M^{\max}_{\rho|\cdot|^{y},\ldots,\rho|\cdot|^{x}}$ as
$$
\begin{cases}
M^{\max}_{\rho|\cdot|^{x}}\circ\ldots\circ M^{\max}_{\rho|\cdot|^{y}} \text{, if }\rho\not\in[x,y]_{\rho} \\
M^{\max}_{\rho|\cdot|^{x}}\circ\ldots\circ M^{\max}_{\rho|\cdot|^{2}}\circ\left( M^{\max}_{Z_\rho[0,1]} \circ M^{\max}_{\rho|\cdot|^{1}} \right)\circ M^{\max}_{\rho|\cdot|^{-1}} \circ \ldots \circ M^{\max}_{\rho|\cdot|^{y}} \text{, if }\rho\in[x,y]_{\rho} \text{ and}\ x>0,
\end{cases}
$$
and
 $L^{\max}_{\rho|\cdot|^{y},\ldots,\rho|\cdot|^{x}}$ as
$$
\begin{cases}
L^{\max}_{\rho|\cdot|^{x}}\circ\ldots\circ L^{\max}_{\rho|\cdot|^{y}} \text{, if }\rho\not\in[x,y]_{\rho} \\
L^{\max}_{\rho|\cdot|^{x}}\circ\ldots\circ L^{\max}_{\rho|\cdot|^{2}}\circ\left( L^{\max}_{Z_\rho[0,1]} \circ L^{\max}_{\rho|\cdot|^{1}} \right)\circ L^{\max}_{\rho|\cdot|^{-1}} \circ \ldots \circ L^{\max}_{\rho|\cdot|^{y}} \text{, if }\rho\in[x,y]_{\rho} \text{ and}\ x>0.
\end{cases}
$$
\end{de}
The reason for this definition for $D^{\max}_{\rho|\cdot|^{y},\ldots,\rho|\cdot|^{x}}$, when $\rho\in[x,y]_{\rho}$, is given in \cite[Proposition 3.7]{AM}. 
Namely, by defining it as we did, we have the following theorem:
\begin{thm}(\cite[Theorem 3.1, 3.2]{A2}) \label{iredder}\\
    For a segment $[x,y]_{\rho}$ (with $x>0$ if $\rho \in [x,y]_{\rho}$) and $\pi\in \mathscr{R}(G)$, if $\pi$ is irreducible, then 
    $$
    D^{\max}_{\rho|\cdot|^{y},\ldots,\rho|\cdot|^{x}}(\pi)
    $$ 
    is also irreducible.
    Moreover, the Langlands data for $D^{\max}_{\rho|\cdot|^{y},\ldots,\rho|\cdot|^{x}}(\pi)$ can be described from those for $\pi$ explicitly, and vice versa. 
\end{thm}
\begin{de}
    Let $\pi \in \mathscr{R}(GL)$. Let $\sigma$ be a supercuspidal representation or $Z_\rho[0,1]$. If $D_\sigma^{max}(\pi)=D_\sigma^{(k)}(\pi)$, we denote
    \begin{align*}
        \mathrm{mult}(D_\sigma,\pi)\ceq k.
    \end{align*}
    Analogously, we define for $\tau \in \mathscr{R}(G)$(if $k$ is the exponent of the respective maximal derivative)
    \begin{align*}
        \mathrm{mult}(M_\sigma,\pi)&\ceq k,\\
        \mathrm{mult}(L_\sigma,\pi)&\ceq k.
    \end{align*}
\end{de}
\begin{prop} \label{tad}
Let $\pi \in \mathscr{R}(GL)$ and $\sigma \in \mathscr{R}(G)$. Let $\rho\in \mathrm{Irr}(GL)$ be a supercuspidal representation. Then 
\begin{align*}
D^{\max}_{\rho|\cdot|^z}(\pi \rtimes\sigma ) &= M^{\max}_{\rho|\cdot|^z}(\pi) \rtimes D^{\max}_{\rho|\cdot|^z}(\sigma), \\
\mathrm{mult}(D_{\rho|\cdot|^z},\pi\rtimes\sigma) & =\mathrm{mult}(M_{\rho|\cdot|^z},\pi) + \mathrm{mult}(D_{\rho|\cdot|^z},\sigma).
\end{align*}
Assume now that $M_{\rho|\cdot|^1}(\pi)=0$ and $D_{\rho|\cdot|^1}(\sigma)=0$ ($\sigma$ is $\rho|\cdot|^1$-reduced). Then
\begin{align*}
D^{\max}_{Z_\rho[0,1]}(\pi\rtimes\sigma) &= M^{\max}_{Z_\rho[0,1]}(\pi) \rtimes D^{\max}_{Z_\rho[0,1]}(\sigma), \\
\mathrm{mult}(D_{Z_\rho[0,1]},\pi\rtimes\sigma) &= \mathrm{mult}(M_{Z_\rho[0,1]},\pi) + \mathrm{mult}(D_{Z_\rho[0,1]},\sigma).
\end{align*}
\end{prop}
\begin{proof}
    Write 
    \begin{align*}
        M^*(\pi) &= \sum_i \pi_1^{(i)} \otimes \pi_2^{(i)}, \\
        \mu^*(\sigma) &= \sum_j \tau_1^{(j)} \otimes \sigma_2^{(j)}. 
    \end{align*}
    By Theorem \ref{TadicStruct} we have 
    \begin{align*}
        \mu^*(\pi \rtimes \sigma) &= M^*(\pi) \rtimes \mu^*(\sigma) \\
        &= \sum_{i,j} \left[ (\pi_1^{(i)} \times \tau_1^{(j)}) \otimes (\pi_2^{(i)} \rtimes \sigma_2^{(j)} ) \right]. 
    \end{align*}
    Let us denote $k_1=\text{mult}(M_{\rho|\cdot|^z},\pi)$ and $k_2=\text{mult}(D_{\rho|\cdot|^z},\sigma)$.
    Now clearly
    \begin{align*}
        D^{\max}_{\rho|\cdot|^z}(\pi \rtimes\sigma ) 
        &= \sum_{\substack{i,j \\ \pi_1^{(i)} \cong (\rho|\cdot|^z)^{k_1} \\  \tau_1^{(j)} \cong (\rho|\cdot|^z)^{k_2} }}   \left[ \pi_2^{(i)} \rtimes \sigma_2^{(j)} \right] = M^{\max}_{\rho|\cdot|^z}(\pi) \rtimes D^{\max}_{\rho|\cdot|^z}(\sigma), 
    \end{align*}
    which also implies the equality for the multiplicities. For the second assertion, we denote $k=\text{mult}( D_{Z_\rho[0,1]},\pi\rtimes\sigma)$. Hence, we get
    \begin{align*}
       (Z_\rho[0,1])^k \leq  (\pi_1^{(i)} \times \tau_1^{(j)}). 
    \end{align*}
    Now this shows that the supercuspidal support of both $\pi_1^{(i)}$ and $\tau_1^{(j)}$ consists only of $\rho$ and $\rho|\cdot|^1$ (with multiplicities). Therefore $\pi_1^{(i)}$ is of the form
    \begin{align*}
        L(i_1\cdot [0,0]_\rho+ j_1\cdot [1,1]_\rho+ (k_1-i_1-j_1)\cdot [1,0]_\rho)
    \end{align*}
    and $\tau_1^{(j)}$ of the form
    \begin{align*}
             L(i_2\cdot [0,0]_\rho+ j_2\cdot [1,1]_\rho+ (k_2-i_2-j_2)\cdot [1,0]_\rho).
    \end{align*}
    Since $\sigma$ is $\rho|\cdot|^1$-reduced, so is $\tau_1^{(j)}$. By Lemma \ref{reduced}, we know that $\pi_1^{(i)}$ is $\rho|\cdot|^1$-reduced as well. Now this implies that $k_1=i_1+j_1$ and $k_2=i_2+j_2$ as well as $i_1\geq j_1$ and $i_2 \geq j_2$. But since $i_1+i_2=j_1+j_2=k$, we have $i_1=j_1$ and $i_2=j_2$. As in the proof of \cite[Lemma 3.6]{AM}, this
    implies that $\pi_1^{(i)} \cong (Z_\rho[0,1])^{i_1}$ and $\tau_1^{(j)} \cong (Z_\rho[0,1])^{i_2}$. The remaining argument is the same. 
\end{proof}
\begin{rmk}
Most importantly, this implies
\begin{gather} \label{a123}
       \left( D^{\max}_{Z_\rho[0,1]} \circ D^{\max}_{\rho|\cdot|^{1}} \right)(\pi\rtimes\sigma) = (M^{\max}_{Z_\rho[0,1]}\circ M^{\max}_{\rho|\cdot|^1})(\pi) \rtimes (D^{\max}_{Z_\rho[0,1]}\circ D^{\max}_{\rho|\cdot|^1})(\sigma).
    \end{gather} 
        We will only consider $D^{\max}_{Z_\rho[0,1]}$ as a part of the composition of the derivatives $D^{\max}_{\rho|\cdot|^{y},\ldots,\rho|\cdot|^{x}}$, where it is preceded by the $\rho|\cdot|^1$-derivative. In the cases of our interest (when $\pi \rtimes \sigma$ is irreducible and in a composition of derivatives $D^{\max}_{\rho|\cdot|^{y},\ldots,\rho|\cdot|^{x}}$), $M_{\rho}^{\max}(\pi)$ and $M_{Z_\rho[0,1]}^{\max}(\pi)$ are irreducible because of the irreducibility of $D_\rho^{\max}(\pi \rtimes \sigma)$ and the above proposition.
\end{rmk}
Atobe's Propositions 3.4, 3.6 and 3.7 in \cite{A2} are the cornerstones of describing the socle of the induced representation $\uab|\cdot|^s\rtimes\pi$. 
Since it is the basic case of the irreducibility criterion we study in Section \ref{main}, we introduce a notation for a special composition of derivatives:
\begin{de} \label{def3}
 Let $u_{\rho}(a,b)|\cdot|^s$ be an essentially Speh representation. We set
\begin{gather*}
    D^{\max}_{(a,b,s)}\ceq D^{\max}_{\rho|\cdot|^{B-s+1},\ldots,\rho|\cdot|^{A-s+1}} \circ \cdots \circ D^{\max}_{\rho|\cdot|^{B+s},\ldots,\rho|\cdot|^{A+s}}.
\end{gather*}
We analogously define
\begin{gather*}
    M^{\max}_{(a,b,s)}\ceq M^{\max}_{\rho|\cdot|^{B-s+1},\ldots,\rho|\cdot|^{A-s+1}} \circ \cdots \circ M^{\max}_{\rho|\cdot|^{B+s},\ldots,\rho|\cdot|^{A+s}}
\end{gather*}
and
\begin{gather*}
    L^{\max}_{(a,b,s)}\ceq L^{\max}_{\rho|\cdot|^{B-s+1},\ldots,\rho|\cdot|^{A-s+1}} \circ \cdots \circ L^{\max}_{\rho|\cdot|^{B+s},\ldots,\rho|\cdot|^{A+s}}.
\end{gather*}
\end{de}
Let us introduce some additional notation for derivatives:
\begin{de}
Let $u_{\rho_l}(a_l,b_l)|\cdot|^{s_l}$ be an essentially Speh representation. For a pair of indices $i,j$ such that $i\in\{1,\ldots,2s_l\}$ and $j\in\{1,\ldots,b_l\}$, we denote 
\begin{align} \label{sigma}
    \sigma_l^{ij} \ceq \left\{
\begin{matrix}
    \rho|\cdot|^1 & if\ B_l+s_l-i+j=0 \\
    Z_\rho[0,1] & if\ B_l+s_l-i+j=1 \text{ and } j\neq 1\\
    \rho|\cdot|^{B_l+s_l-i+j} & else.
\end{matrix}
    \right.
\end{align}
For an essentially Speh $u_\rho(a,b)|\cdot|^s$ we simply denote this by $\sigma^{ij}$.
Let $\pi\in\mathrm{Irr}(G)$. If
\begin{align*}
    D^{\max}_{(a_l,b_l,s_l)}(\pi) = 
    & \left( 
        \begin{array}{c}
        D^{(k_{2s_l,b_l})}_{\sigma_l^{2s_l b_l}} \mathrel{\circ} \dotsm \mathrel{\circ} D^{(k_{2s_l,1})}_{\sigma_l^{2s_l 1}} 
        \end{array}
    \right) \\
    & \mathrel{\circ} \dotsm \mathrel{\circ} \\
    & \left( 
        \begin{array}{c}
        D^{(k_{1,b_l})}_{\sigma_l^{1 b_l}} \mathrel{\circ} \dotsm \mathrel{\circ} D^{(k_{1,1})}_{\sigma_l^{11}} 
        \end{array}
    \right)(\pi).
\end{align*}
Then we denote $\text{mult}_{ij} (D_{(a_l,b_l,s_l)},\pi)\ceq k_{i,j}.$
Also, $D^{\max}_{(a_l,b_l,s_l),< i,j}(\pi)$ denotes the first derivatives in the above composition (up to the indices $i,j$):
\begin{align*}
  D^{\max}_{(a_l,b_l,s_l),< i,j}(\pi) \ceq  & (D^{(k_{i,j-1})}_{\sigma_l^{i j-1}}\circ\ldots\circ D^{(k_{i,1})}_{\sigma_l^{i 1}}) \\
    & \mathrel{\circ} \dotsm \mathrel{\circ} \\
    & \left( 
        \begin{array}{c}
        D^{(k_{1,b_l})}_{\sigma_l^{1 b_l}} \mathrel{\circ} \dotsm \mathrel{\circ} D^{(k_{1,1})}_{\sigma_l^{11}} 
        \end{array}
    \right)(\pi).
\end{align*}
Denote $D^{\max}_{(a_l,b_l,s_l),\le i,j}= D^{\max}_{\sigma_l^{ij}}\circ D^{\max}_{(a_l,b_l,s_l),< i,j}$.
We define these notions for the derivatives $M^{\max}_{(a_l,b_l,s_l)}$ and $L^{\max}_{(a_l,b_l,s_l)}$ analogously. 
Let $\tau \in \mathscr{R}(GL)$. We also denote:
    \begin{align*}
        d^{ij}_{l}(\pi) & \ceq \text{mult} (D_{\sigma_l^{ij}},D^{\max}_{(a_l,b_l,s_l),< i,j}(\pi)),\\
    m^{ij}_{l}(\tau) & \ceq \text{mult} (M_{\sigma_l^{ij}},M^{\max}_{(a_l,b_l,s_l),< i,j}(\tau)),\\
    l^{ij}_{l}(\tau) & \ceq \text{mult} (L_{\sigma_l^{ij}},L^{\max}_{(a_l,b_l,s_l),< i,j}(\tau)).\\
    \end{align*}
    For an essentially Speh representation $\uab|\cdot|^s$ (without an index), we simply denote $d^{ij}(\pi)$, $m^{ij}(\tau)$ and $l^{ij}(\tau)$ for these.
    If in one of these blocks of derivatives, the indices $i,j$ correspond to the exponent $0$ (i.e. $B_l+s_l-i+j=0$), we denote 
    \begin{align*}
        1_{ij}=0.
    \end{align*}
    For all other indices we set 
    \begin{align*}
        1_{ij}=1.
    \end{align*}
\end{de}
\begin{rmk} \label{exact1}
    Since we extensively use the operator $\text{mult}$, in this remark we will describe how it is applied in the most common setting in the paper. Namely, we often show that $\text{mult}(D^{\max}_{\sigma})$ is the same for two irreducible representations which are related in some way.
    
    Let $\sigma,\pi_1,\pi_1'\in \mathrm{Irr}(GL)$, $\sigma$ supercuspidal and $\pi, \pi_2\in \mathrm{Irr}(G)$. Assume that $\pi_1\rtimes\pi_2$ is an irreducible representation such that
    $$\pi_1\rtimes\pi_2\hookrightarrow \pi_1'\rtimes\pi.$$
    Moreover, assume that
    $\text{mult}(M_{\sigma},\pi_1)=\text{mult}(M_{\sigma},\pi_1')=:k_1$ and $\pi_1'\times\sigma$ is an irreducible representation.
    If we denote $k=\text{mult}(D_{\sigma},\pi)$ and $k_2=\text{mult}(D_{\sigma},\pi_2)$, using the assumptions we get:
    $$ \pi_1\rtimes\pi_2\hookrightarrow \sigma^{k_1+k}\times M^{\max}_{\sigma}(\pi_1')\rtimes D^{\max}_{\sigma}(\pi).$$
    From the Frobenius reciprocity, it follows $\text{mult}(D_{\sigma},\pi_1\rtimes\pi_2)\ge k_1 + k$.
    On the other hand, the exactness of the Jacquet functor implies  $\text{mult}(D^{\max}_{\sigma}(\pi_1\rtimes\pi_2))\le k_1 + k$. 
    Hence, we have the equalities 
    $$ k_1 + k_2 = \text{mult}(D_{\sigma},\pi_1\rtimes\pi_2) = k_1 + k. $$
    This implies $\text{mult}(D_{\sigma},\pi_2)=\text{mult}(D_{\sigma},\pi).$
\end{rmk}

\nopagebreak[4]
\section{Representations of Arthur type} \label{Arthur}

\subsection{Arthur parameters} \label{a-parameters}
Denote by $\hat{G}_n$ the complex dual group of $G_n$. Namely, $\hat{G}_n = \mathrm{Sp}_{2n}(\mathbb{C})$ if $G_n = \mathrm{SO}_{2n+1}(F)$, and $\hat{G}_n = \mathrm{SO}_{2n+1}(\mathbb{C})$ if $G_n = \mathrm{Sp}_{2n}(F)$. An Arthur parameter for $G_n$ is the $\hat{G}_n$-conjugacy class of an admissible homomorphism
\[
\psi: W_F \times \mathrm{SL}_2(\mathbb{C}) \times \mathrm{SL}_2(\mathbb{C}) \to \hat{G}_n,
\]
such that the image of the Weil group $W_F$ is bounded. By composing with the standard representation of $\hat{G}_n$, we can regard $\psi$ as a representation of $W_F \times \mathrm{SL}_2(\mathbb{C}) \times \mathrm{SL}_2(\mathbb{C})$. It decomposes as
\[
\psi = \bigoplus_\rho \left( \bigoplus_{i \in I_\rho} \rho \boxtimes S_{a_i} \boxtimes S_{b_i} \right),
\]
where
\begin{itemize}
    \item $\rho$ runs over the equivalence classes of unitary irreducible supercuspidal representations of $\mathrm{GL}_d(F)$, which is identified with an irreducible bounded representation of $W_F$ by the local Langlands correspondence for the general linear groups;
    \item $S_a$ is the unique irreducible algebraic representation of $\mathrm{SL}_2(\mathbb{C})$ of dimension $a$.
\end{itemize}
For $\rho \in \mathrm{Cusp}^\perp(\mathrm{GL}_d(F))$, there exists a non-degenerate $\mathrm{GL}_d(F)$-invariant bilinear form on $\rho$. We call $\rho$ \textit{orthogonal} if this form is symmetric and \textit{symplectic} if it is skew-symmetric.
Let $\psi$ be as above. We say that $\psi$ is of \textit{good parity} if $\rho \boxtimes S_{a_i} \boxtimes S_{b_i}$ is self-dual of the same type as $\psi$ for any $\rho$ and $i \in I_\rho$, i.e.,
\begin{itemize}
    \item $\rho \in \mathrm{Cusp}^\perp(\mathrm{GL}_d(F))$ is orthogonal and $a_i + b_i \equiv 0 \pmod{2}$ if $G_n = \mathrm{Sp}_{2n}(F)$ (resp. $a_i + b_i \equiv 1 \pmod{2}$ if $G_n = \mathrm{SO}_{2n+1}(F)$); or
    \item $\rho \in \mathrm{Cusp}^\perp(\mathrm{GL}_d(F))$ is symplectic and $a_i + b_i \equiv 1 \pmod{2}$ if $G_n = \mathrm{Sp}_{2n}(F)$ (resp. $a_i + b_i \equiv 0 \pmod{2}$ if $G_n = \mathrm{SO}_{2n+1}(F)$).
\end{itemize}
Let $\Psi(G_n) \supset \Psi^\mathrm{gp}(G_n)$ be the sets of equivalence classes of Arthur parameters and Arthur parameters of good parity, respectively. Also, we set $\Phi^\mathrm{temp}(G_n)$ to be the subset of $\Psi(G_n)$ consisting of tempered Arthur parameters, i.e., Arthur parameters $\psi$ which are trivial on the second $\mathrm{SL}_2(\mathbb{C})$. Finally, we set $\Phi^\mathrm{gp}(G_n) = \Psi^\mathrm{gp}(G_n) \cap \Phi^\mathrm{temp}(G_n)$.

For $\psi = \bigoplus_\rho \left( \bigoplus_{i \in I_\rho} \rho \boxtimes S_{a_i} \boxtimes S_{b_i} \right) \in \Psi^\mathrm{gp}(G_n)$, define the enhanced component group by
\[
\mathcal{A}_\psi = \bigoplus_\rho \bigoplus_{i \in I_\rho} (\mathbb{Z}/2\mathbb{Z})_{\alpha_{\rho,i}},
\]
i.e., $\mathcal{A}_\psi$ is a $(\mathbb{Z}/2\mathbb{Z})$-vector space with a canonical basis $\alpha_{\rho,i}$ corresponding to $\rho \boxtimes S_{a_i} \boxtimes S_{b_i}$. The component group $\mathcal{S}_\psi$ is the quotient of $\mathcal{A}_\psi$ by the subgroup generated by
\begin{itemize}
    \item $\alpha_{\rho,i} + \alpha_{\rho,j}$ such that $\rho \boxtimes S_{a_i} \boxtimes S_{b_i} = \rho \boxtimes S_{a_j} \boxtimes S_{b_j}$; and
    \item $z_\psi = \sum_\rho \sum_{i \in I_\rho} \alpha_{\rho,i}$, which is called the central element of $\mathcal{A}_\psi$.
\end{itemize}

Let $\hat{\mathcal{S}_\psi} \subset \hat{\mathcal{A}_\psi}$ be the Pontryagin duals of $\mathcal{S}_\psi$ and $\mathcal{A}_\psi$, respectively. When $\varepsilon \in \hat{\mathcal{A}}_\psi$, we write $\varepsilon(\rho \boxtimes S_{a_i} \boxtimes S_{b_i}) \ceq \varepsilon(\alpha_{\rho,i}) \in \{\pm1\}$.

\subsection{Arthur packets} \label{a-packs}
Let $\mathrm{Irr}^\mathrm{unit}(G_n)$ (resp. $\mathrm{Irr}^\mathrm{temp}(G_n))$ be the set of equivalence classes of irreducible unitary (resp. tempered) representations of $G_n$. To an Arthur parameter $\psi \in \Psi(G_n)$, Arthur \cite[Theorem 1.5.1(a)]{Arthur} associated an Arthur packet $\Pi_\psi$, which is a finite multi-set over $\mathrm{Irr}^\mathrm{unit}(G_n)$. We say that $\pi \in \mathrm{Irr}(G_n)$ is of Arthur type if $\pi \in \Pi_\psi$ for some $\psi \in \Psi(G_n)$. In particular, such a $\pi$ is unitary (\cite[Theorem 1.5.1]{Arthur}). If $\psi$ is of good parity, we say that $\pi$ is of good parity as well. 

Mœglin \cite{Mœglin2} showed that $\Pi_\psi$ is multiplicity-free, i.e., a subset of $\mathrm{Irr}^\mathrm{unit}(G_n)$. By \cite[Theorem 1.5.1(b)]{Arthur}, if $\phi \in \Phi^\mathrm{temp}(G_n)$ is a tempered Arthur parameter, then $\Pi_\phi$ is a subset of $\mathrm{Irr}^\mathrm{temp}(G_n)$ and
\[
\mathrm{Irr}^\mathrm{temp}(G_n) = \bigsqcup_{\phi \in \Phi^\mathrm{temp}(G_n)} \Pi_\phi \quad \text{(disjoint union)}.
\]

However, $\Pi_{\psi_1} \cap \Pi_{\psi_2} \neq \emptyset$ even if $\psi_1 \not\cong \psi_2$ in general.
If $\psi = \bigoplus_\rho \bigoplus_{i \in I_\rho} \rho \boxtimes S_{a_i} \boxtimes S_{b_i}$, set
\[
\tau_\psi = \bigtimes_\rho \bigtimes_{i \in I_\rho} u_\rho(a_i, b_i)
\]
to be a product of (unitary) Speh representations, which is an irreducible unitary representation of $\mathrm{GL}_m(F)$ with $m = \dim(\psi)$.

\begin{prop}[{\cite[Theorem 6]{Mœglin}}, {\cite[Proposition 8.11]{Xu2}}] \label{bad-parityy}
Any $\psi \in \Psi(G_n)$ can be decomposed as
\[
\psi = \psi_1 \oplus \psi_0 \oplus \psi_1^\vee,
\]
where
\begin{itemize}
    \item $\psi_0 \in \Psi^\mathrm{gp}(G_{n_0})$;
    \item $\psi_1$ is a direct sum of irreducible representations of $W_F \times \mathrm{SL}_2(\mathbb{C}) \times \mathrm{SL}_2(\mathbb{C})$ which are not self-dual of the same type as $\psi$.
\end{itemize}
For $\pi_0 \in \Pi_{\psi_0}$, the parabolically induced representation $\tau_{\psi_1} \rtimes \pi_0$ is irreducible and independent of the choice of $\psi_1$. Moreover,
\[
\Pi_\psi = \{\tau_{\psi_1} \rtimes \pi_0 \mid \pi_0 \in \Pi_{\psi_0} \}.
\]
\end{prop}

\begin{cor} \label{bad-parity-is-easy}
    If $\psi$ is an Arthur parameter of good parity, $\pi \in \Pi_\psi$ and $u_\rho(a,b)$ a Speh representation such that $\psi \oplus (\rho \boxtimes S_a \boxtimes S_b)^{\oplus 2} $ is not of good parity, then $u_\rho(a,b) \rtimes \pi$ is irreducible.
\end{cor}
\begin{proof}
    Since $\psi$ is of good parity, define
    \begin{align*}
        \psi_1' & \ceq \rho \boxtimes S_a \boxtimes S_b \\
        \psi'  & \ceq \psi_1' \oplus \psi \oplus (\psi_1')^\vee.
    \end{align*}
    Then Proposition \ref{bad-parityy} implies that $u_\rho(a,b) \rtimes \pi = \tau_{\psi_1} \rtimes \pi$ is irreducible. 
\end{proof}

We also say that a Speh representation $u_\rho(a,b)$ is of \textit{good parity with respect to $\pi$}, with the Arthur parameter $\psi$ corresponding to $\pi$, if 
\begin{align*}
    \psi \oplus (\rho \boxtimes S_a \boxtimes S_b)^{\oplus 2}
\end{align*}
is of good parity. 

Throughout this paper, we implicitly fix a Whittaker datum for $G_n$. Let $\psi \in \Psi^\mathrm{gp}(G_n)$ so that we have defined the component group $\mathcal{S}_\psi$. In \cite[Theorem 1.5.1(a)]{Arthur} Arthur gives a map
\[
\Pi_\psi \to \hat{\mathcal{S}}_\psi, \quad \pi \mapsto \langle \cdot, \pi \rangle_\psi.
\]
If $\psi = \phi \in \Phi^\mathrm{gp}(G_n)$ is tempered, this map is bijective. When $\pi \in \Pi_\phi$ corresponds to $\epsilon \in \hat{\mathcal{S}}_\phi$, we write $\pi = \pi(\phi, \epsilon)$.

\subsection{Extended multi-segments} \label{Extended multi-segments}
Extended multi-segments are an explicit way to parameterize the representations in a given Arthur packet. We define extended multi-segments in a slightly modified version of that given by Atobe in \cite[Definition 3.1]{A1}, to make the non-vanishing criterion (\cite[Theorem 4.4]{A1} reformulated to Theorem \ref{thm:nonvanishingstandard}), which we need in multiple places, easier to use. We use the symbol $\mathcal{E}$ for extended multi-segments with Atobe's parametrization and $\mathcal{S}$ for ours. 
\begin{de} (Extended multi-segments) 
    \begin{enumerate}
        \item We call a tuple $([A,B]_\rho,\mu)$ an \textit{extended segment}, if 
        \begin{enumerate}
            \item $\rho \in \mathrm{Cusp}^\perp (\mathrm{GL}_d (F))$ is an irreducible self-dual supercuspidal representation,
            \item $[A,B]_\rho$ is a segment,
            \item $\mu \in \Z$ with $\mu \equiv b \mod{2}$,
            \item $|\mu| \leq b$,\ \label{formal}
        \end{enumerate} 
        where $b \ceq A-B+1$. 
        \item We call the tuple $([A,B]_\rho,\mu)$ a \textit{formal extended segment}, if we drop condition (\ref{formal}) in the above definition.
        \item \label{ext-mult-seg} An \textit{extended multi-segment} $\mathcal{S}$ is a set of sequences of extended segments
        \begin{align*}
            \mathcal{S}=\bigcup_{\rho \in C_\mathcal{S}} \{(S_i^{\rho})_{i=1}^{n_\rho}\},
        \end{align*}
        where 
        \begin{align*}
            S_i^{\rho}=([A_i^\rho,B_i^\rho]_\rho, \mu_i^\rho)
        \end{align*}
        such that
        \begin{enumerate}
            \item $C_\mathcal{S} \subseteq \mathrm{Cusp}^\perp (\mathrm{GL}_d (F))$
            \item for every $\rho \in C_\mathcal{S}$, the sequence of extended segments is in \textit{admissible} order, which means that $A_i^\rho>A_j^\rho$ and $B_i^\rho>B_j^\rho$ imply $i>j$,
            \item $A_i^\rho+B_i^\rho \geq 0$ for all $\rho$ and all $i$ such that $1 \leq i \leq  n_\rho$,
            \item as a representation of $W_F \times \mathrm{SL}_2(\C) \times \mathrm{SL}_2(\C)$,
            \begin{align*}
                \psi_{\mathcal{S}} \ceq  \bigoplus_{\rho \in C_\mathcal{S}} \bigoplus_{i=1}^{n_\rho} \rho \boxtimes S_{a_i^\rho} \boxtimes S_{b_i^\rho} \in \Psi_{gp}(G_n),
            \end{align*}
            where $a_i^\rho=A_i^\rho+B_i^\rho+1$ and $b_i^\rho=A_i^\rho-B_i^\rho+1$,
            \item the following “sign condition” holds:
            \begin{align*} \label{sign}
                \sum_{\rho \in C_\mathcal{S}} \sum_{i=1}^{n_\rho}\left( \left \lfloor \frac{\mu_i^\rho}{2} \right \rfloor + \mu_i^\rho \sum_{j=1}^{i-1} (b_j^\rho-1)\right) \equiv 0 \mod{2}.
            \end{align*}
        \end{enumerate}
        \item A \textit{formal extended multi-segment } $\mathcal{S}$ is a set of sequences of formal extended segments
        \begin{align*}
            \mathcal{S}=\bigcup_{\rho \in C_\mathcal{S}} \{(S_i^{\rho})_{i=1}^{n_\rho}\}
        \end{align*}
        satisfying all the other conditions in (\ref{ext-mult-seg}).
        \item If the set $C_\mathcal{S}$ contains only one element $\rho$, the extended (formal) multi-segment $\mathcal{S}$ is called $\rho$-homogeneous or simply homogeneous. Similar to \cite[Definition 3.8]{hazel} we set 
        \begin{align*}
            \mathcal{S}_\rho \ceq (S_i^\rho)_{i=1}^{n_\rho}
        \end{align*}
        and call it the $\rho$-homogeneous part or $\rho$-part of $\mathcal{S}$. With this notation, we have
        \begin{align*}
            \mathcal{S}=\bigcup_{\rho \in C_{\mathcal{S}}} \{\mathcal{S}_\rho\}.
        \end{align*}
\item 
We say that the sequence $(S_i^\rho)_{i=1}^{n_\rho}$ is in \textit{very admissible order}, if it satisfies: 
\begin{align*}
    For\ 1\leq i,j\leq n_{\rho}:\ B_i^{\rho}>B_j^{\rho}\ \Rightarrow i>j.
\end{align*}
\item We say that the extended multi-segment $\mathcal{S}=\bigcup_{\rho \in C_{\mathcal{S}}} \{\mathcal{S}_\rho\}$, with $\mathcal{S}_\rho = (([A_i^\rho,B_i^\rho],\mu_i^\rho))_{i=1}^{n_\rho}$ is \textit{admissible}, if it satisfies the following condition:
\begin{align}
    \ B_i^\rho<0\ for\ some\ 1\leq i \leq n_\rho \Rightarrow \mathcal{S}_\rho\ is\ in\ very\ admissible\ order.
\end{align}
\item We call the extended multi-segment $\mathcal{S}$ \textit{non-negative}, if $B_i^{\rho}\geq 0$ for all $i$ such that $1\leq i \leq n_{\rho}$ and for all $\rho$. Every non-negative extended multi-segment is admissible.
\item We denote by $\mathrm{ExMult}$ the set of extended multi-segments and by $\mathrm{ExMult}_{\geq 0}$, $\mathrm{AdExMult}$ and $\mathrm{FExMult}$ the sets of non-negative, admissible and formal extended multi-segments respectively, such that
\begin{align*}
    \mathrm{ExMult}_{\geq 0} \subset \mathrm{AdExMult}\subset \mathrm{ExMult} \subset \mathrm{FExMult}.
\end{align*}
Furthermore, we denote by $\mathrm{ExMult}'$ the set of extended multi-segments as defined by Atobe in \cite[Definition 1.1]{A1}
\end{enumerate}
\end{de}
In the sequence $\mathcal{S}_\rho = (([A_i^\rho,B_i^\rho],\mu_i^\rho))_{i=1}^{n_\rho}$, we say that the extended segments $([A_i^\rho,B_i^\rho],\mu_i^\rho)$ and $([A_{i+1}^\rho,B_{i+1}^\rho],\mu_{i+1}^\rho)$ are \textit{consecutive}.
We introduce the following notation for $1\leq i \leq n_\rho $, to simplify some expressions:
\begin{align*}
    \delta_i^{\rho} &\ceq \prod_{\substack{j=1}}^{i-1} (-1)^{b_j^{\rho}-1}.
\end{align*}
Atobe's definition of extended multi-segments can be recovered from the one given above via the following map:
\begin{de}
Define the map
$\mathcal{F}:\mathrm{ExMult} \rightarrow \mathrm{ExMult}'$ by
\begin{align*}
   \mathcal{F}: \bigcup_{\rho \in C_\mathcal{S}} \{(([A_i^{\rho},B_i^{\rho}]_\rho,\mu_i^{\rho}))_{i=1}^{n_\rho}\} \mapsto \overline{\bigcup_{\rho \in C_\mathcal{S}} \left\{\left([ A_i^{\rho},B_i^{\rho}]_\rho,\frac{b_i^{\rho}-|\mu_i^{\rho}|}{2},\delta_{i}^{\rho}\mathrm{sgn}(\mu_i^{\rho}) \right)\right\}_{i\in (I_\rho,>)}},
\end{align*}
where we set the convention $\mathrm{sgn}(0)=1$, as well as $I_\rho\ceq \{1, \ldots , n_\rho \}$ with the order $>$ of $\N$ on $I_\rho$ and the overline $\overline{\cdot}$ denotes the equivalence class defined in \cite[Definition 1.1, item (3)]{A1}. Furthermore, define the map $\mathcal{F}^{-1}:\mathrm{ExMult}' \rightarrow \mathrm{ExMult}$ (which is in fact the inverse of $\mathcal{F}$), by 
\begin{align*}
   \mathcal{F}^{-1}:  \overline{\bigcup_{\rho\in C_\mathcal{S}}\left\{\left(\left[ A_i^{\rho},B_i^{\rho}\right]_\rho,l_i^{\rho},\eta_i^{\rho} \right)\right\}_{i\in I_{\rho}}} \mapsto   \bigcup_{\rho\in C_\mathcal{S}} \{(([A_i^{\rho},B_i^{\rho}]_\rho,\delta_{i}^{\rho}\eta_i^{\rho} (b_i^{\rho}-2l_i^{\rho})))_{i=1}^{n_\rho}\},
\end{align*}
where we identify the totally ordered finite set $I_\rho$ with $\{1, \ldots n_\rho\}$.
\end{de}

\begin{prop}
    The maps $\mathcal{F}$ and $\mathcal{F}^{-1}$ are well-defined, inverses of each other and map $\mathrm{ExMult}$ bijectively to $\mathrm{ExMult}'$.
\end{prop}
\begin{proof}
    We check that $\mathcal{F}$ is well-defined: By the parity condition on $\mu_i^\rho$, the integer $l_i^\rho$ is well-defined and $0\leq l_i^\rho \leq \frac{b_i^\rho}{2}$. To check that $\mathcal{F}^{-1}$ is well-defined, we note that 
    \begin{align*}
        0\leq |\delta_{i}^\rho\eta_i^\rho(b_i^\rho-2l_i^\rho)|\leq b_i^\rho.
    \end{align*}
    We also have that $\delta_{i}^\rho\eta_i^\rho(b_i^\rho-2l_i^\rho) \equiv b_i^\rho \mod{2}$. Furthermore, elements of the same equivalence class of extended multi-segments in the definition of \cite{A1} have the same image under $\mathcal{F}^{-1}$, because they concern the case $l_i^\rho=\frac{b_i^\rho}{2}$, which gets mapped to $\mu_i^\rho=0$, regardless of the sign $\eta_i^\rho$.
    \begin{align*}
        &(\mathcal{F}\circ \mathcal{F}^{-1})\left(  \overline{\bigcup_{\rho\in C_\mathcal{S}}\left\{\left(\left[ A_i^{\rho},B_i^{\rho}\right]_\rho,l_i^{\rho},\eta_i^{\rho} \right)\right\}_{i\in I_{\rho}}} \right) \\
        =&  \overline{\bigcup_{\rho\in C_\mathcal{S}}\left\{\left(\left[ A_i^{\rho},B_i^{\rho}\right]_\rho,\frac{b_i^\rho-|b_i^\rho-2l_i^\rho|}{2},\eta_i^{\rho} \right)\right\}_{i\in I_{\rho}}} \\
        = &\overline{\bigcup_{\rho\in C_\mathcal{S}}\left\{\left(\left[ A_i^{\rho},B_i^{\rho}\right]_\rho,l_i^{\rho},\eta_i^{\rho} \right)\right\}_{i\in I_{\rho}}}
    \end{align*}
    and
    \begin{align*}
        &(\mathcal{F}^{-1}\circ \mathcal{F}^{-1})\left( \bigcup_{\rho \in C_\mathcal{S}} \{(([A_i^{\rho},B_i^{\rho}]_\rho,\mu_i^{\rho}))_{i=1}^{n_\rho}\} \right) \\
        = &\bigcup_{\rho \in C_\mathcal{S}} \{(([A_i^{\rho},B_i^{\rho}]_\rho,(\delta_{i}^\rho)^2\mathrm{sgn}(\mu_i^\rho)( b_i^\rho-(b_i^\rho-|\mu_i^\rho|)) ))_{i=1}^{n_\rho}\} \\
        =&\bigcup_{\rho \in C_\mathcal{S}} \{(([A_i^{\rho},B_i^{\rho}]_\rho,\mu_i^{\rho}))_{i=1}^{n_\rho}\}.
    \end{align*}
    Lastly, we check that the sign conditions agree:
\begin{align*}
   1 &= \prod_\rho \prod_{i\in I_{\rho}} (-1)^{\left\lfloor  \frac{b_i^{\rho}}{2} \right \rfloor+l_i^{\rho}}(\eta_i^{\rho})^{b_i^{\rho}}\\
    &= \prod_\rho \left(\prod_{i\in I_{\rho}}(-1)^{\left\lfloor\frac{\mu_i^{\rho}}{2} \right \rfloor}  (\delta_i^{\rho})^{\mu_k^{\rho}} \right)  \\
    & \Leftrightarrow \\
    0 & \equiv \sum_\rho \sum_{i=1}^{n_\rho}\left( \left \lfloor \frac{\mu_i^{\rho}}{2} \right \rfloor + \mu_i^{\rho} \sum_{j=1}^{i-1} (b_j^{\rho}-1)\right)   \mod{2}.
\end{align*}
\end{proof}

\begin{de} \label{define-the-pi}
    Let $\mathcal{S} \in \mathrm{AdExMult}$ be an admissible extended multi-segment. We define the representation 
    \begin{align*}
        \pi(\mathcal{S}) \ceq \pi(\mathcal{F}(\mathcal{S})),
    \end{align*}
    where the latter is the representation defined in \cite[Section 3.2]{A1}. If this representation does not vanish, it is a representation of Arthur type in the Arthur packet $\Pi_{\psi_{\mathcal{S}}}$. Furthermore, all representations of Arthur type (of good parity) can be constructed this way (\cite[Theorem 3.4]{A1}). 
\end{de}
In Section \ref{unit}, we will give a more explicit construction of these representations.
There is a caveat to this definition: the representation $\pi(\mathcal{S})$ may vanish. Below, we will define a subset $\mathrm{Rep}$ of $\mathrm{AdExMult}$ through combinatorial means. In Proposition \ref{rep-is-nonzero}, we will show that this set determines precisely those admissible extended multi-segments, that give a non-zero representation. One important step in the determination of this set $\mathrm{Rep}$ is the following:
\begin{de} \label{necessary}
    Let $\mathcal{S}=\bigcup_{\rho \in C_\mathcal{S}} \{(([A_i^{\rho},B_i^{\rho}]_\rho,\mu_i^{\rho}))_{i=1}^{n_\rho}\}\in \mathrm{FExMult}$ be a formal extended multi-segment. We say that $\mathcal{S}$ satisfies the \textit{necessary condition for non-vanishing}, if 
        \begin{align} \label{nec}
     |A_i^\rho-A_{i-1}^\rho| + |B_i^\rho-B_{i-1}^\rho| \geq |\mu_i^\rho-\mu_{i-1}^\rho | \qquad \forall\ 1 <i \leq n_\rho\ \forall\rho\in C_\mathcal{S}.
\end{align}  
\end{de}
Under the map $\mathcal{F}$, this condition is equivalent to the one given in \cite[Proposition 4.1]{A1} on the set of non-negative extended multi-segments. In this sense, it is indeed a necessary condition: If $\mathcal{S}\in \mathrm{ExMult}_{\geq 0}$ and $\pi(\mathcal{S})\neq 0$, then $\mathcal{S}$ satisfies (\ref{nec}). For a proof, see Lemma \ref{necess}.

Another caveat is that the map of extended multi-segments to representations is not injective: There are three types of equivalence relations of extended multi-segments, that leave the emerging representation invariant (\cite[Definition 3.4]{A3}). We will now state the first, since it is also necessary to define the set $\mathrm{Rep}$. Another (Deformation) can be found in the Appendix \ref{equivalences}.

\begin{de} \label{reorder}
On the set $\mathrm{FExMult}$ of formal extended multi-segments
    $\mathcal{S}=\bigcup_{\rho \in C_{\mathcal{S}}}\{(S_i^\rho)_{i=1}^{n_\rho}\}=\bigcup_{\rho \in C_{\mathcal{S}}}\{(([A_i^\rho,B_i^\rho]_\rho,\mu_i^\rho))_{i=1}^{n_\rho}\}$ we define the following operation $R_i^\rho$ for $1\leq i < n_\rho$:
    \begin{enumerate}
        \item If  $[A_{i}^\rho,B_{i}^\rho]_\rho\subseteq [A_{i+1}^ \rho,B_{i+1}^\rho]_\rho$, $R_{i}^\rho(\mathcal{S})\ceq\mathcal{S}'=\bigcup_{\rho \in C_{\mathcal{S}}}\{(S_i'^\rho)_{i=1}^{n_\rho}\}$, where
        \begin{align}\label{gamma}
            S_{i}'^\rho &\ceq ([A_{i+1}^\rho,B_{i+1}^\rho]_\rho,2\mu_{i}^\rho-\mu_{i+1}^\rho) \\
            S_{i+1}'^\rho &\ceq ([A_{i}^\rho,B_{i}^\rho]_\rho,\mu_{i}^\rho)  \nonumber \\
            S_j'^\rho &\ceq S_j^\rho\ else.               \nonumber
        \end{align}
        \item  If  $[A_{i}^\rho,B_{i}^\rho]_\rho\supseteq [A_{i+1}^ \rho,B_{i+1}^\rho]_\rho$, $R_{i}^\rho(\mathcal{S})\ceq\mathcal{S}'=\bigcup_{\rho \in C_{\mathcal{S}}}\{(S_i'^\rho)_{i=1}^{n_\rho}\}$, where
        \begin{align} \label{beta}
            S_{i}'^\rho &\ceq ([A_{i+1}^\rho,B_{i+1}^\rho]_\rho,\mu_{i+1}^\rho) \\
            S_{i+1}'^\rho &\ceq ([A_{i}^\rho,B_{i}^\rho]_\rho,2\mu_{i+1}^\rho - \mu_{i}^\rho) \nonumber\\
            S_{j}'^\rho &\ceq S_j^\rho\ else.                                 \nonumber
        \end{align} 
    \item If $A_{i+1}^\rho \geq  A_{i}^\rho$ and $ B_{i+1}^\rho \geq B_{i}^\rho$, $R_{i}^\rho(\mathcal{S})\ceq\mathcal{S}$ (as in \cite[Definition 3.15]{hazel})
    \end{enumerate}
\end{de}
\begin{rmk} \label{why-formal}
        We think of $R_i^\rho$ as transpositions on the index set, with changes to the $\mu$-parameter. It is easy to see that $(R_i^\rho)^2=\mathrm{id}$. 
        At this point the notion of formal extended multi-segments comes into play: If $\mathcal{S}$ is an extended multi-segment, $R_i^\rho(\mathcal{S})$ is not necessarily in $\mathrm{ExMult}$ again. However, if $\mathcal{S}$ satisfies the necessary condition (\ref{nec}), we have in fact $R_i^\rho(\mathcal{S}) \in \mathrm{ExMult}$. 
        
        Since the necessary condition is formulated in terms of the $\rho$-parts $\mathcal{S}_\rho$ of $\mathcal{S}$ and the maps $R_i^{\rho'}$ for $\rho' \not \cong \rho$ leave $\mathcal{S}_\rho$ invariant, we can consider the $\rho$-parts separately for most of the following proofs.
\end{rmk}
\begin{de}
We call two formal extended multi-segments a \textit{reordering} of each other, if one is obtained via a sequence of such reorders from the other, i.e. $\mathcal{S}\sim \mathcal{S}'$ if and only if $\mathcal{S}=R_{i_1}^{\rho_1'} \circ ... \circ R_{i_m}^{\rho_m'}(\mathcal{S}')$ for some $\rho_j' \in C_{\mathcal{S}}$ and $i_j< n_{\rho_j'}$. This is a equivalence relation on $\mathrm{FExMult}$ and we denote the equivalence class of $\mathcal{S}$ by $[\mathcal{S}]$.
\end{de}
    \begin{prop} \label{switch}
        The maps $R_i^\rho$ above coincide with $\mathcal{F}^{-1} \circ R_k \circ \mathcal{F}$ (with $R_k$ from \cite[Definition 3.15]{hazel} for $k$ corresponding to $i$ and $\rho$) on the set $\mathrm{FExMult}$ ($\mathcal{F}(\mathcal{S})$ is called a \textit{symbol} there). Suppose that $\mathcal{S} \in \mathrm{ExMult}_{\geq 0}$ and $\pi(\mathcal{S})\neq 0$ then 
        \begin{align*}
            \pi(\mathcal{S}) \cong \pi(R_i^\rho(\mathcal{S})).
        \end{align*}
    \end{prop}
    \begin{proof}
        We consider any $\rho$-part of $\mathcal{S}$ (we drop the index $\rho$). Let $S_{i-1}$ and $S_i$ be two consecutive formal extended segments in $\mathcal{S}$ as in the first case. Set $\varepsilon\ceq\eta_{i-1}\eta_i(-1)^{b_{i-1}-1}$. According to \cite[Definition 3.15]{hazel} the reordering of $\mathcal{F}(\mathcal{S})$ is given by $(l_{i-1}',\eta_{i-1}')=(l_{i-1},\eta_{i-1} (-1)^{b_{i}-1})$ and
    \begin{align*}
        (l_{i}',\eta_i') = \left\{ \begin{matrix}
            (b_{i} - l_{i} - |w_{i-1}|,(-1)^{b_{i-1}-1}\eta_{i}) & \ if \ \varepsilon=1 \ and\ |w_{i}| < 2|w_{i-1}| \\
             (l_{i} + |w_{i-1}|,(-1)^{b_{i-1}}\eta_{i}) & \ if \ \varepsilon=1 \ and\ |w_{i}| \geq 2|w_{i-1}| \\
            (l_{i}-|w_{i-1}|,(-1)^{b_{i-1}}\eta_{i}) & \ if \ \varepsilon=-1. \\
        \end{matrix} \right.
    \end{align*}
    Thus
    $$
    \mu_{i}'=\delta_{i} (-1)^{b_{i-1}-1} \eta_{i}' (b_{i}-2l_{i}')= 2\mu_{i-1}-\mu_i, 
    $$
    and
    \begin{align*}
    \mu_{i-1}'&=\delta_{i-1}(-1)^{b_i-1}\eta_{i-1}' (b_{i-1}-2l_{i-1})=\delta_{i-1}\eta_{i-1} (b_{i-1}-2l_{i-1}) \\
        & = \mu_{i-1}.
    \end{align*}
    The other case is proven analogously. Now the second assertion follows from \cite[Theorem 3.18]{hazel}.
    \end{proof}
We can now define the set $\mathrm{Rep}$ of which we will show (Proposition \ref{rep-is-nonzero}) that it produces exactly the non-vanishing representations:
\begin{de} \label{rep}
    We define the subset $\mathrm{Rep}$ of $\mathrm{AdExMult}$ as those admissible extended multi-segments $\mathcal{S}=\bigcup_{\rho \in C_\mathcal{S}}\{(([A_i^\rho,B_i^\rho],\mu_i^\rho))_{i=1}^{n_\rho}\}$, that satisfy:
    \begin{enumerate}
        \item Every reordering $\mathcal{S}'\in [\mathcal{S}]$ of $\mathcal{S}$ satisfies the necessary condition for non-vanishing (\ref{nec}).
        \item For all $1\leq i\leq n_\rho$ and all $\rho \in C_{\mathcal{S}}$
    \begin{align} \label{nec-negative}
        |\Hat{\mu_i^\rho}| \leq a_i^\rho, 
    \end{align}
    where we define
    \begin{align} \label{muhat}
        \Hat{\mu_i^\rho} \ceq  \left\{ \begin{matrix}
            \mu_i^\rho &\ if\ B_i^\rho \in \Z, \\
            \mu_i^\rho - 1 &\ if\ B_i^\rho \not \in \Z.
        \end{matrix} \right.
    \end{align}
    \end{enumerate}
\end{de}
\begin{rmk}
    Note that if $\mathcal{S}\in \mathrm{ExMult}_{\geq 0}$, the second condition is always satisfied automatically. 
\end{rmk}
\begin{prop} \label{reorders-stay-in-ext}
    If $\mathcal{S} \in \mathrm{Rep}$, then $[\mathcal{S}] \subseteq \mathrm{ExtMult}$.
\end{prop}
Since we are interested in the set $\mathrm{Rep}$ and all reorderings give extended multi-segments again, the notion of formal extended multi-segments, necessary for the definition, becomes obsolete from here on out. 
\begin{proof}
    Let $\mathcal{S} \in \mathrm{Rep}$ and let $\mathcal{S}' \sim \mathcal{S}$. By definition, this means that $\mathcal{S}'=R_{i_1}^{\rho_1'} \circ ... \circ R_{i_m}^{\rho_m'}(\mathcal{S})$ for some $\rho_j' \in C_{\mathcal{S}}$ and $i_j< n_{\rho_j'}$. We show that $\mathcal{S}'\in\mathrm{ExMult}$ by induction over the number $m$. If $m=0$, this is clear. For the inductive step, we may assume that $\mathcal{S}''\ceq R_{i_2}^{\rho_2'} \circ ... \circ R_{i_m}^{\rho_m'}(\mathcal{S}) \in \mathrm{ExMult}$. By definition of the set $\mathrm{Rep}$, the extended multi-segment $\mathcal{S}''$ satisfies (\ref{nec}). Now by Remark \ref{why-formal}, it follows that $\mathcal{S}'=R_{i_1}^{\rho_1'}(\mathcal{S}'') \in \mathrm{ExMult}$.
\end{proof}

\begin{de}
    Let $\mathcal{S}=\bigcup_{\rho \in C_\mathcal{S}}\{(([A_i^\rho,B_i^\rho],\mu_i^\rho))_{i=1}^{n_\rho}\}$ be an extended multi-segment. 
$\mathcal{S}$ is said to be \textit{standard}, if 
\begin{align*}
    B_i^\rho < B_j^\rho\ or\ B_i^\rho = B_j^\rho \ and\ A_i^\rho  > A_j^\rho \Rightarrow i < j  \qquad for\ 1\leq i,j \leq n_\rho, \quad \forall \rho \in C_\mathcal{S}.\\
\end{align*}
A standard extended multi-segment is always admissible. 
\end{de}
\begin{prop}
    Let $\mathcal{S}\in \mathrm{Rep}$. Then there is a unique standard extended multi-segment $\mathcal{S}' \in [\mathcal{S}]$. We call it the \textit{standard form} of $\mathcal{S}$.
\end{prop}
\begin{proof}
    For $\mathcal{S}\in \mathrm{Rep}$, we may apply operations $R_i^\rho$ until the pairs $(B_i^\rho,-A_i^\rho)$ are in lexicographical order in the $\rho$-part and define the resulting standard extended multi-segment as $\mathcal{S}'$. The only ambiguity could stem from several extended segments $([A,B]_\rho,\mu_1),...,([A,B]_\rho,\mu_k)$, since we did not define their order. However, since $\mathcal{S}'$ must satisfy (\ref{nec}), we have that $\mu_1 = \ldots =\mu_k$, so these extended segments are all the same.
\end{proof}

In Corollary \ref{rep-of-standard}, we will show that $\mathcal{S}$ and its standard form give rise to the same representation. It is particularly convenient to choose a standard reordering of an extended multi-segment, because the non-vanishing result can be formulated much simpler in this case:
\begin{de}
    Let $\mathcal{S}= \bigcup_{\rho \in C_\mathcal{S}} \{(S_i^\rho)_{i=1}^{n_\rho}\}=\bigcup_{\rho \in C_\mathcal{S}} \{(([A_i^\rho,B_i^\rho],\mu_i^\rho))_{i=1}^{n_\rho}\}$ be a standard extended multi-segment. 
We call $S_i^\rho$ and $S_j^\rho$  with $1\leq i<j\leq n_\rho$ \textit{connected} in $\mathcal{S}$, if there is no $k$ with $i<k<j$ such that 
\begin{align*}
    &A_i^\rho \geq  A_k^\rho \geq A_j^\rho
\end{align*}
or
\begin{align*}
    &A_i^\rho < A_k^\rho < A_j^\rho.
\end{align*}
    In this case, we write $S_i^\rho \sim S_j^\rho.$ 
\end{de}

\begin{rmk}
Note that if $S_i^\rho$ and $S_j^\rho$ are connected, there exists some reordering $\mathcal{S}'$ of $\mathcal{S}$ where $(S_i^\rho)'$ and $(S_j^\rho)'$ are consecutive. Furthermore, for every such $\mathcal{S}'$, there is only one possibility for the position of $(S_k^\rho)'$ relative to $(S_i^\rho)'$ and $(S_j^\rho)'$, for every $i<k<j$: it always comes either before or after both in the sequence $\mathcal{S}_\rho'$.
\end{rmk}

\begin{de} \label{delta}
    Let $\mathcal{S}$ be a standard extended multi-segment and let $S_i^\rho$ and $S_j^\rho$ be connected in $\mathcal{S}$. We define
\begin{align} \label{formula}
    \Delta_{\mathcal{S}}(\mu_i^\rho,\mu_{j}^\rho)\ceq \sum_{m=i}^{j-1} (-1)^{ | \{ m<k<j|\ A_i^\rho \geq A_k^\rho \} | } (\mu_{m+1}^\rho-\mu_m^\rho).
\end{align}
Furthermore, we define $\Delta_{\mathcal{S}}(\mu_i^\rho,\mu_{j}^\rho)\ceq  -\Delta_{\mathcal{S}}(\mu_j^\rho,\mu_{i}^\rho)$ and $\Delta_{\mathcal{S}}(\mu_i^\rho,\mu_{i}^\rho)\ceq 0$ .
\end{de}
\begin{restatable}{thm}{nonvanishingstandard}
\label{thm:nonvanishingstandard}
    Let $\mathcal{S}= \bigcup_{\rho \in C_\mathcal{S}} \{(S_i^\rho)_{i=1}^{n_\rho}\}= \bigcup_{\rho \in C_\mathcal{S}} \{(( [A_i^{\rho},B_i^{\rho}]_\rho,\mu_i^{\rho}))_{i=1}^{n_\rho}\}$ be a standard extended multi-segment. $\mathcal{S}\in \mathrm{Rep}$ if and only if 
    \begin{align} \label{ine}
         |A_j^{\rho}-A_i^{\rho}|+|B_j^{\rho}-B_i^{\rho}| \geq |\Delta_{\mathcal{S}}(\mu_i^{\rho},\mu_j^{\rho})| \quad \forall\ 1\leq i <j\leq n_\rho\ \text{with}\ S_{i}^{\rho}\sim S_{j}^{\rho} \ \forall \rho \in C_{\mathcal{S}}
    \end{align}
    and $\mathcal{S}$ satisfies condition (\ref{nec-negative}).
\end{restatable}
For a proof, see Appendix \ref{app}.

\subsection{Representations corresponding to extended multi-segments} \label{unit}

Given an admissible extended multi-segment $\mathcal{S}= \bigcup_{\rho \in C_\mathcal{S}} \{(( [A_i^{\rho},B_i^{\rho}]_\rho,\mu_i^{\rho}))_{i=1}^{n_\rho}\}$, we have defined a corresponding representation $\pi(\mathcal{S})$ in Definition \ref{define-the-pi}. Let us make that more explicit. 
We distinguish two cases: \\
For the first case, assume that $B_{i+1}^{\rho}>A_i^{\rho}$ for $1\leq i < n_\rho$ and that $\mathcal{S}$ is non-negative. Then, we have

\begin{align*}
    \pi(\mathcal{S})=\mathrm{soc}\left(  \bigtimes_{\rho \in C_\mathcal{S}} \bigtimes_{i=1}^{n_\rho} u_{ess}\begin{pmatrix}
-\frac{a_i^{\rho}+b_i^{\rho}}{2}+1 \hspace{3mm} &  \frac{a_i^{\rho}-b_i^{\rho}}{2}  \hspace{3mm} \\
\hspace{3mm}  -\frac{a_i^{\rho}+|\mu_i^{\rho}|}{2}  & \hspace{3mm} \frac{a_i^{\rho}-|\mu_i^{\rho}|}{2}-1
\end{pmatrix}^{(\rho)}\rtimes \pi(\phi,\varepsilon)\right),
\end{align*}
where 
$$
\phi=\bigoplus_{\rho \in C_{\mathcal{S}}} \bigoplus_{i=1}^{n_\rho} \rho \boxtimes \left(S_{a_i^{\rho}-|\mu_i^{\rho}|+1} \oplus \ldots \oplus S_{a_i^{\rho}+|\mu_i^{\rho}|-1} \right)
$$
and 
$$
\varepsilon\left(\rho \boxtimes S_{a_i^{\rho}-|\mu_i^{\rho}|+2t+1}\right)=\mathrm{sgn}(\mu_i^{\rho}\delta_{i}^{\rho})(-1)^t,
$$
for $1 \leq i \leq n_\rho$. \\
In the second case (when the conditions of the first case do not hold), we find non-negative integers $t_i^\rho$ for $1\leq i \leq n_\rho$ and for $\rho \in C_{\mathcal{S}}$ such that 
$$
\mathcal{S}'\ceq\bigcup_{\rho \in C_{\mathcal{S}}}\left\{\left(\left(\left[A_i^\rho+t_i^\rho,B_i^\rho+t_i^\rho\right]_\rho,\mu_i^\rho\right)\right)_{i=1}^{n_\rho}\right\}
$$
satisfies the conditions of the first case. We have
$$
\pi(\mathcal{S}) = \circ_{\rho \in C_\mathcal{S}} \circ_{i \leq n_\rho} D_{\rho,i}  (\pi(\mathcal{S}')),
$$
where 
$$
D_{\rho,i} \ceq  D_{\rho |\cdot|^{B_i^\rho+1},...,\rho|\cdot|^{A_i^\rho+1}} \circ ... \circ D_{\rho |\cdot|^{B_i^\rho+t_i^\rho},...,\rho|\cdot|^{A_i^\rho+t_i^\rho}}.
$$
Let us now prove that the set $\mathrm{Rep}$ is indeed what we claimed it to be:
\begin{prop} \label{rep-is-nonzero}
    Let $\mathcal{S}\in \mathrm{AdExMult}$. Then 
    \begin{align*}
        \pi(\mathcal{S}) \neq 0 \Leftrightarrow \mathcal{S} \in \mathrm{Rep}.
    \end{align*}
Moreover, if $\mathcal{S}\in \mathrm{Rep}$ and $\mathcal{S}' \in [\mathcal{S}]\cap \mathrm{AdExMult}$, then 
\begin{align*}
    \pi(\mathcal{S}') \cong\pi(\mathcal{S}). 
\end{align*}
\end{prop}
\begin{proof}
    To prove the first assertion, we distinguish two cases:
    \begin{enumerate}
        \item \underline{$\mathcal{S}\in \mathrm{ExMult}_{\geq 0}$:} \\
        This follows from \cite[Theorem 3.20 (ii)]{hazel} for $\mathcal{F}(\mathcal{S})$, if we take Proposition \ref{switch} into consideration.
        \item \underline{$\mathcal{S}\not\in \mathrm{ExMult}_{\geq 0}$:} \\
        Set $\mathcal{E}\ceq \mathcal{F}(\mathcal{S})$. Say $\mathcal{E}=\bigcup_{\rho\in C_\mathcal{S}}\left\{\left(\left[ A_i^{\rho},B_i^{\rho}\right]_\rho,l_i^{\rho},\eta_i^{\rho} \right)\right\}_{i\in I_{\rho}}$. Now \cite[Theorem 3.20 (i)]{hazel} says that $\pi(\mathcal{S})\neq 0$ if and only if $\pi(sh^d(\mathcal{E}))\neq 0$ for any integer $d>0$ such that $sh^d(\mathcal{E})$ is non-negative and if the following condition holds for all $i\in I_\rho$ and all $\rho \in C_\mathcal{S}$:
        \begin{align*}
            B_i^\rho+l_i^\rho \geq \left\{ \begin{matrix}
                0 & if\ B_i^\rho \in \Z,\\
                \frac{1}{2} & if\ B_i^\rho \not\in \Z\ and\ \eta_i^\rho=(-1)^{\alpha_i^\rho+1},\\
                -\frac{1}{2} & if\ B_i^\rho \not\in \Z\ and\ \eta_i^\rho=(-1)^{\alpha_i^\rho},
            \end{matrix} \right.
        \end{align*}
        where 
        \begin{align*}
            \alpha_i^\rho \ceq \sum_{\substack{j\in I_\rho \\ j<i}} A_j^\rho+ B_j^\rho + 1.
        \end{align*}
    Now this first condition is (by the first case) equivalent to the fact that all extended multi-segments in $[\mathcal{F}^{-1}(sh^d(\mathcal{E}))]=[sh^d(\mathcal{S})]=sh^d([\mathcal{S}])$ satisfy (\ref{nec}), where we define the shift $sh^d(\mathcal{S})$ analogous to $sh^d(\mathcal{E})$. We immediately see that $sh^d(\mathcal{S'})$ satisfies (\ref{nec}) if and only if $\mathcal{S}'$ does. \\
    For the latter condition, we note that 
    \begin{align*}
          \frac{a_i^\rho-|\mu_i^\rho|}{2}=\frac{a_i^\rho-b_i^\rho}{2}+ \frac{b_i^\rho-|\mu_i^\rho|}{2}=B_i^\rho+l_i^\rho
    \end{align*}
    and
    so the condition reads
    \begin{align*}
    a_i^\rho \geq \left\{
        \begin{matrix}
            |\mu_i^\rho| &\ if\ B_i^\rho \in \Z, \\
            |\mu_i^\rho| - \eta_i^\rho(-1)^{\alpha_i^\rho} &\ if\ B_i^\rho \not \in \Z.
        \end{matrix} \right.
    \end{align*}
    But if $B_i^\rho \not\in\Z$,
    \begin{align*}
      |\mu_i^\rho|-\eta_i^\rho (-1)^{\alpha_i^\rho} = \eta_i^\rho (-1)^{\alpha_i^\rho} (\delta_i^\rho (-1)^{\alpha_i^\rho} \mu_i^\rho- 1)
    \end{align*}
    and 
    \begin{align*}
        \delta_i^\rho (-1)^{\alpha_i^\rho} = \prod_{j<i} (-1)^{a_j^\rho+b_j^\rho-1} = \prod_{j<i} (-1)^{2A_j^\rho+1}=1
    \end{align*}
    imply 
    \begin{align*}
        |\mu_i^\rho|-\eta_i^\rho(-1)^{\alpha_i^\rho}=| |\mu_i^\rho|-\eta_i^\rho(-1)^{\alpha_i^\rho}| = |\mu_i^\rho-1|,
    \end{align*}
    except when $\mu_i^\rho=0$ and $\eta_i^\rho(-1)^{\alpha_i^\rho}=1$, where $a_i^\rho\geq 1$ so the condition is satisfied anyways. \\
    Now suppose $\mathcal{S}\in \mathrm{Rep}$ and $\mathcal{S}'$ is a reordering that is admissible. We may choose $\mathcal{S}'=R_{i_1}^{\rho_1'} \circ ... \circ R_{i_m}^{\rho_m'}(\mathcal{S})$ for some $\rho_j' \in C_{\mathcal{S}}$ and $i_j< n_{\rho_j'}$ such that $\mathcal{S}_k \ceq R_{i_k}^{\rho_k'} \circ ... \circ R_{i_m}^{\rho_m'}(\mathcal{S})$ is admissible for every $1 \leq k \leq m$. Now $\pi(\mathcal{S})\neq 0$ and hence $\pi(\mathcal{S}_k)\cong \pi(\mathcal{S})$ for every $1\leq k\leq m$ inductively, because of \cite[Theorem 3.5]{A3}.
    \end{enumerate}
\end{proof}
\begin{cor} \label{rep-of-standard}
Let $\mathcal{S}\in \mathrm{Rep}$.
    \begin{enumerate}
    \item If $\mathrm{ExMult}_{\geq 0}$, then $\pi(\mathcal{S}) \cong \pi(\mathcal{S}')$ for all $\mathcal{S}'\in [\mathcal{S}]$.
        \item Let $\mathcal{S}'$ be the standard form of $\mathcal{S}$. Then $\mathcal{S}'\in \mathrm{Rep}$ and $\pi(\mathcal{S}) \cong \pi(\mathcal{S}')$.
    \end{enumerate}
\end{cor}
This explains the first of the three equivalence relations on $\mathrm{Rep}$. It implies that we should think of the representation $\pi(\mathcal{S})$ as being parametrized by the equivalence classes $[\mathcal{S}]$. For these we can always choose a standard representative. Hence every representation of Arthur type of good parity can be written as $\pi(\mathcal{S)}$ for a standard extended multi-segment $\mathcal{S} \in \mathrm{Rep}$, the latter set being determined by Definition \ref{rep}.
\begin{de}
    Denote the set of standard extended multi-segments $\mathcal{S}$ that lie in $\mathrm{Rep}$ by $\mathrm{SRep}$.
\end{de}

\subsection{Parabolic induction of Speh and Arthur type representations}

As we are working towards a proof of Theorem \ref{thm:IRR}, it is important to understand how representations of Arthur type behave with parabolic induction. If we are inducing with Speh representations, there is a description in terms of extended multi-segments, which we will give now. 

\begin{de}
    Let $\mathcal{S}\in \mathrm{SRep}$. For an extended segment $([C,D]_\rho,\nu)$, denote by $\mathcal{S}_{([C,D]_\rho,\nu)}$ the extended multi-segment that is obtained by inserting $([C,D]_\rho,\nu)$ twice into $\mathcal{S}_\rho$ at consecutive indices $i'$ and $i'+1$, such that the resulting extended multi-segments is standard. That is, the sequence $(\mathcal{S}_{([C,D]_\rho,\nu)})_\rho$ is of the form
    \begin{align*}
        (( [A_1^{\rho},B_1^{\rho}]_\rho,\mu_1^{\rho}),\ldots ,([C,D]_\rho,\nu),([C,D]_\rho,\nu),\ldots ,( [A_{n_\rho}^{\rho},B_{n_\rho}^{\rho}]_\rho,\mu_{n_\rho}^{\rho}))
    \end{align*}
   with $([C,D]_\rho,\nu)$ as the $i'$-th and $(i'+1)$-th element of the sequence. More generally, we denote
    \begin{align*}
        \mathcal{S}_{([C_1,D_1]_{\rho_1},\nu_1),...,([C_k,D_k]_{\rho_k},\nu_k)} \ceq (\ldots (\mathcal{S}_{([C_1,D_1]_{\rho_1},\nu_1)})_{([C_2,D_2]_{\rho_2},\nu_2)} \ldots )_{([C_k,D_k]_{\rho_k},\nu_k)}.
    \end{align*}
    the extended multi-segment that is obtained by successively inserting each of the extended segments $([C_1,D_1]_{\rho_1},\nu_1),...,([C_k,D_k]_{\rho_k},\nu_k)$ twice into $\mathcal{S}_{\rho_i}$ respectively such that the resulting extended multi-segment is standard. We call two identical consecutive extended segments a block.
\end{de}

\begin{rmk}
    Note that the position of the index in our definition differs from the one given in \cite[Theorem 4.4]{A2} for $\mathcal{E}_{(l,\eta)}$, but the corresponding extended multi-segments are equivalent and the $\mu$-parameters coincide, except at the inserted extended segments, but that will be of no concern for us, because the decomposition in Proposition \ref{decomp} sums over all allowed parameters (see the proof).
\end{rmk}

\begin{lemma} \label{basics}
    Let $\mathcal{S} \in \mathrm{SRep}$. Let 
    $$
    ([C_1,D_1]_\rho,\nu_1),\ldots,([C_k,D_k]_\rho,\nu_k)
    $$
    be extended segments (with the same supercuspidal representation $\rho$).
    \begin{enumerate}
        \item For every permutation $\sigma$ of $\{1,...,k\}$, we have
    \begin{align*}
        \mathcal{S}_{([C_{\sigma(1)},D_{\sigma(1)}]_\rho,\nu_{\sigma(1)}),...,([C_{\sigma(k)},D_{\sigma(k)}]_\rho,\nu_{\sigma(k)})}= \mathcal{S}_{([C_1,D_1]_{\rho},\nu_1),...,([C_k,D_k]_{\rho},\nu_k)}.
    \end{align*}
    
        \item $\mathcal{S}_{([C_i,D_i]_\rho,\nu_i)}$ is indeed an extended multi-segment, i.e. it satisfies the sign condition.
        \item Condition (\ref{nec-negative}) is satisfied for $\mathcal{S}_{([C_i,D_i]_\rho,\nu_i)}$ if and only if it is also satisfied for $\mathcal{S}$ and for the indices $i'$ and $i'+1$ of $([C_i,D_i]_\rho,\nu_i)$ in $\mathcal{S}_{([C_i,D_i]_\rho,\nu_i)}$ . 
    \end{enumerate}
\end{lemma}
\begin{proof}
    \begin{enumerate}
        \item We consider 
        \begin{align*}
        \mathcal{S}_{([C_{1},D_{1}]_\rho,\nu_{1}),([C_{2},D_{2}]_\rho,\nu_{2})} = (\mathcal{S}_{([C_{1},D_{1}]_\rho,\nu_{1})})_{([C_{2},D_{2}]_\rho,\nu_{2})} 
    \end{align*}
    and 
    \begin{align*}
        \mathcal{S}_{([C_{2},D_{2}]_\rho,\nu_{2}),([C_{1},D_{1}]_\rho,\nu_{1})} = (\mathcal{S}_{([C_{2},D_{2}]_\rho,\nu_{2})})_{([C_{1},D_{1}]_\rho,\nu_{1})}. 
    \end{align*}
    These are the same, except when $D_1=D_2$ and $C_1=C_2$, in which case it is easy to see that reordering the blocks does not change the parameters $\nu_1$ and $\nu_2$. Now any permutation can be obtained through several such transpositions.
    \item The left hand side of the sign condition reads the same but with the addition of 
    \begin{align*}
        \left\lfloor \frac{\nu_i}{2} \right\rfloor + \nu_i \sum_{j<i'} (b_j-1) + \left\lfloor \frac{\nu_i}{2} \right\rfloor + \nu_i \sum_{j<i'+1} (b_j-1) \equiv \nu_i (b_i-1) \equiv 0 \mod{2}.
    \end{align*}
    This implies that the sign condition is the same for $\mathcal{S}$ and $\mathcal{S}_{([C_i,D_i]_\rho,\nu_i)}$.
    \item This is clear. 
    \end{enumerate}
\end{proof}

\begin{prop} (\cite[Theorem 4.4]{A2}) \label{decomp} \\
    For $\mathcal{S}\in \mathrm{SRep}$, $a,b\in \Z_{\geq 0 }$ such that $\psi_{\mathcal{S}} \oplus (\rho \boxtimes S_a \boxtimes S_b)^{\oplus 2} $ is of good parity and $A=\frac{a+b}{2}-1$, $B=\frac{a-b}{2}$. We have that 
\begin{align*}
    u_\rho(a,b) \rtimes \pi(\mathcal{S}) \cong \bigoplus_{\substack{|\mu| \leq b \\ \mu \equiv b \mod{2}}} \pi(\mathcal{S}_{([A,B]_\rho,\mu)}).
\end{align*}
\end{prop}
\begin{proof}
    This is just a rewriting of Atobe's theorem into our notation. Set $\mathcal{E} \ceq \mathcal{F}(\mathcal{S})$. Then by \cite[Theorem 4.4]{A2}, we have
    \begin{align*}
    u_\rho(a,b) \rtimes \pi(\mathcal{E}) \cong \bigoplus_{(l,\eta)} \pi(\mathcal{E}_{(l,\eta)}), 
\end{align*}
where $(l,\eta)$ sums over all pairs of parameters that make $([A,B]_\rho,l,\eta)$ an extended segment in the sense of \cite{A2}. Now $\mathcal{S}'_{(l,\eta)} \ceq \mathcal{F}^{-1}(\mathcal{E}_{(l,\eta)})$ is the extended multi-segment $\mathcal{S}$ with the extended segment $([A,B]_\rho,\delta_{i''}^\rho \eta(b-2l))$ (we set $\mu(l,\eta) \ceq \delta_{i''}^\rho \eta(b-2l))$) added twice at the positions $i''$ and $i''+1$ of the sequence $\mathcal{S}_\rho$, with the property that $B_j^\rho>B$ for $j\geq i''$ and $B_j^\rho \leq B$ for $j <i''$ in $\mathcal{S}_\rho$ (recall that $\mathcal{S}$ is standard). Now to get the standard form of $\mathcal{S}'_{(l,\eta)}$, we have to reorder both extended segments $([A,B]_\rho,\delta_{i''}^\rho \eta(b-2l))$ with the extended segments $S_{i''-1}, \ldots , S_{i'}$ in this order, where $i'$ the largest index such that $B_{i'-1}^\rho < B $ or $B_{i'-1}^\rho = B $ and $A_{i'-1}^\rho \geq  A $. Hence the standard form of $\mathcal{S}'_{(l,\eta)}$ is given by $\mathcal{S}_{([A,B]_\rho,\mu(l,\eta))}$, where we define
\begin{align*}
    \mu(l,\eta) \ceq  2\cdot \sum_{j=0}^{i''-i'-1} (-1)^j \mu_{i'+j}^\rho + (-1)^{i''-i'} (\delta_{i''}^\rho \eta(b-2l)) \equiv b \mod{2}.
\end{align*}
If $\pi(\mathcal{E}_{(l,\eta)})=\pi(\mathcal{S}_{(l,\eta)}')\neq 0$, then $\pi(\mathcal{S}_{([A,B]_\rho,\mu(l,\eta))})\cong \pi(\mathcal{S}_{(l,\eta)}')$ and $|\mu(l,\eta)|\leq b$ (Proposition \ref{reorders-stay-in-ext}).  The mapping of parameters $\mu(l,\eta)$ is injective except for $2l=b$ and $\eta= \pm 1$, but this corresponds to the two equivalent extended multi-segments $\mathcal{E}_{(l,\eta)}$. On the other hand, if $\pi(\mathcal{S}_{([A,B]_\rho,\mu)})\neq 0$, we have that $\mathcal{S}_{([A,B]_\rho,\mu)} \sim \mathcal{S}_{(l,\eta)}'$ where $\mu(l,\eta)=\mu$ and hence $\mathcal{E}_{(l,\eta)}$ must be a valid extended multi-segment.
\end{proof}
\begin{lemma} \label{delta-lemma}
    Let $\mathcal{S}= \bigcup_{\rho \in C_\mathcal{S}} \{(( [A_i^{\rho},B_i^{\rho}]_\rho,\mu_i^{\rho}))_{i=1}^{n_\rho}\} \in \mathrm{AdExMult}$ be standard. Let $c,d\in \Z_{\geq 0 }$ and $\rho \in \mathrm{Cusp}^\perp (\mathrm{GL}_d (F))$ such that $\psi_{\mathcal{S}} \oplus (\rho \boxtimes S_c \boxtimes S_d)^{\oplus 2} $ is of good parity and $C=\frac{c+d}{2}-1$, $D=\frac{c-d}{2}$. Let $([C,D]_\rho,\nu)$ be an extended segment and let $1\leq i <j\leq n_\rho$ such that for the index $i'$ of the first $([C,D]_\rho,\nu)$ in $(\mathcal{S}_{([C,D]_\rho,\nu)})_\rho$, we have $i<i'\leq j$. Assume that $S_i^\rho$ and $S_j^\rho$ are connected in $\mathcal{S}$.
   \begin{enumerate}
    \item If $S_i^\rho$ and $S_j^\rho$ are connected in $\mathcal{S}_{([C,D]_\rho,\nu)}$, then 
    \begin{align*}
        \Delta_{\mathcal{S}}(\mu_i^{\rho},\mu_j^\rho) =  \Delta_{\mathcal{S}_{([C,D]_\rho,\nu)}}(\mu_i^\rho,\mu_j^\rho).
    \end{align*}
    \item If $S_i^\rho$ and $S_j^\rho$ are not connected in $\mathcal{S}_{([C,D]_\rho,\nu)}$, then $S_i^\rho$ and $([C,D]_\rho,\nu)$ as well as $([C,D]_\rho,\nu)$ and $S_j^\rho$ are connected and
    \begin{align*}
        |\Delta_{\mathcal{S}}(\mu_i^{\rho},\mu_j^\rho)| & \leq |\Delta_{\mathcal{S}_{([C,D]_\rho,\nu)}}(\mu_i^\rho,\nu)| + |\Delta_{\mathcal{S}_{([C,D]_\rho,\nu)}}(\nu,\mu_{j}^\rho)|.
    \end{align*}
    \end{enumerate}
    \end{lemma}
    \begin{proof}
         \begin{enumerate}
             \item We have
             \begin{align*}
                \Delta_{\mathcal{S}_{([C,D]_\rho,\nu)}}(\mu_i^\rho,\mu_j^\rho)=  & \sum_{m=i}^{i'-2} (-1)^{ | \{ m<k<j|\ A_i^\rho \geq A_k^\rho \} | } (\mu_{m+1}^\rho-\mu_m^\rho)\\
                 + &(-1)^{ | \{ i'-1<k<j|\ A_i^\rho \geq A_k^\rho \} | } (\nu-\mu_{i'-1}^\rho) +  (-1)^{ | \{ i'-1<k<j|\ A_i^\rho \geq A_k^\rho \} | } (\mu_{i'}^\rho-\nu) \\
                 + &\sum_{m=i'}^{j-1} (-1)^{ | \{ m<k<j|\ A_i^\rho \geq A_k^\rho \} | } (\mu_{m+1}^\rho-\mu_m^\rho) \\
                 =& \sum_{m=i}^{j-1} (-1)^{ | \{ m<k<j|\ A_i^\rho \geq A_k^\rho \} | } (\mu_{m+1}^\rho-\mu_m^\rho) \\
                 =& \Delta_{\mathcal{S}}(\mu_i^\rho,\mu_j^\rho).
             \end{align*}
             \item There are two cases: If $A_i^\rho <A_j^\rho$, we must have $A_i^\rho< C <A_j^\rho$. Note that for $i' \leq k <j$, we have 
             \begin{align*}
                 A_i^\rho \geq A_k^\rho \Leftrightarrow C \geq A_k^\rho,
             \end{align*}
             since otherwise $S_i$ and $S_j$ would not be connected. Similarly, if $A_i^\rho  \geq A_j^\rho$, we must have $A_i^\rho \geq C \geq A_j^\rho$ and for $i' \leq k <j$, we have 
             \begin{align*}
                 A_i^\rho \geq A_k^\rho \Leftrightarrow C \geq A_k^\rho,
             \end{align*}
             since otherwise $S_i$ and $S_j$ would not be connected. In either case, we have 
             \begin{align*} 
    \Delta_{\mathcal{S}}(\mu_i^\rho,\mu_{j}^\rho)&=\sum_{m=i}^{j-1} (-1)^{ | \{ m<k<j|\ A_i^\rho \geq A_k^\rho \} | } (\mu_{m+1}^\rho-\mu_m^\rho)  \\
    \Delta_{\mathcal{S}_{([C,D]_\rho,\nu)}}(\mu_i^\rho,\nu)&= \sum_{m=i}^{i'-2} (-1)^{ | \{ m<k<i'|\ A_i^\rho \geq A_k^\rho \} | } (\mu_{m+1}^\rho-\mu_m^\rho) + (\nu-\mu_{i'-1}^\rho)  \\
    \Delta_{\mathcal{S}_{([C,D]_\rho,\nu)}}(\nu,\mu_{j}^\rho)&= (-1)^{ | \{ i'-1<k<j|\ A_i^\rho \geq A_k^\rho \} | } (\mu_{i'}^\rho-\nu) \\
    &+ \sum_{m=i'}^{j-1} (-1)^{ | \{ m<k<j|\ A_i^\rho \geq A_k^\rho \} | } (\mu_{m+1}^\rho-\mu_m^\rho).
\end{align*}
Now we note 
 \begin{align*} 
    &|\Delta_{\mathcal{S}}(\mu_i^\rho,\mu_{j}^\rho)|=\left|\sum_{m=i}^{i'-2} (-1)^{ | \{ m<k<j|\ A_i^\rho \geq A_k^\rho \} | } (\mu_{m+1}^\rho-\mu_m^\rho) \right. \\
     &\left. + (-1)^{ | \{ i'-1<k<j|\ A_i^\rho \geq A_k^\rho \} | } (\mu_{i'}^\rho-\mu_{i'-1}^\rho)  + \sum_{m=i'}^{j-1} (-1)^{ | \{ m<k<j|\ A_i^\rho \geq A_k^\rho \} | } (\mu_{m+1}^\rho-\mu_m^\rho)\right|  \\
    &=\left| (-1)^{ | \{ i'-1<k<j|\ A_i^\rho \geq A_k^\rho \} | } \left(\sum_{m=i}^{i'-2} (-1)^{ | \{ m<k<i'|\ A_i^\rho \geq A_k^\rho \} | } (\mu_{m+1}^\rho-\mu_m^\rho)  +   (\nu-\mu_{i'-1}^\rho) \right) \right. \\
     &\left. + (-1)^{ | \{ i'-1<k<j|\ A_i^\rho \geq A_k^\rho \} | } (\mu_{i'}^\rho-\nu)  + \sum_{m=i'}^{j-1} (-1)^{ | \{ m<k<j|\ A_i^\rho \geq A_k^\rho \} | } (\mu_{m+1}^\rho-\mu_m^\rho)\right|  \\
    &\leq |\Delta_{\mathcal{S}_{([C,D]_\rho,\nu)}}(\mu_i^\rho,\nu)| + |\Delta_{\mathcal{S}_{([C,D]_\rho,\nu)}}(\nu,\mu_{j}^\rho)|.
\end{align*}
         \end{enumerate}
    \end{proof}

\begin{lemma} \label{reduce}
    Let $\mathcal{S}= \bigcup_{\rho \in C_\mathcal{S}} \{(( [A_i^{\rho},B_i^{\rho}]_\rho,\mu_i^{\rho}))_{i=1}^{n_\rho}\}\in \mathrm{AdExMult}$ be standard. Let $c,d\in \Z_{\geq 0 }$ and $\rho \in \mathrm{Cusp}^\perp (\mathrm{GL}_d (F))$ such that $\psi_{\mathcal{S}} \oplus (\rho \boxtimes S_c \boxtimes S_d)^{\oplus 2} $ is of good parity and $C=\frac{c+d}{2}-1$, $D=\frac{c-d}{2}$. Let $([C,D]_\rho,\nu)$ be an extended segment. 
    We write the sequence $(\mathcal{S}_{([C,D]_\rho,\nu)})_\rho$ as
    \begin{align*}
        (( [A_1^{\rho},B_1^{\rho}]_\rho,\mu_1^{\rho}),\ldots ,([C,D]_\rho,\nu),([C,D]_\rho,\nu),\ldots ,( [A_{n_\rho}^{\rho},B_{n_\rho}^{\rho}]_\rho,\mu_{n_\rho}^{\rho}))
    \end{align*}
   with $([C,D]_\rho,\nu)$ as the $i'$-th and $(i'+1)$-th element of the sequence. Then $\mathcal{S}_{([C,D]_\rho,\nu)} \in \mathrm{SRep}$ if and only 
    \begin{enumerate}
        \item $\mathcal{S}\in \mathrm{SRep}$.
        \item Condition (\ref{nec-negative}) is satisfied for $(\mathcal{S}_{([C,D]_\rho,\nu)})_\rho$ at the index $i'$.
        \item For all extended segments $S_{i}^{\rho}=([A_i^\rho,B_i^\rho],\mu_i^\rho)$ appearing in $\mathcal{S}_\rho$ such that $S_{i}^{\rho}\sim ([C,D]_\rho,\nu)$ in $\mathcal{S}_{([C,D]_\rho,\nu)}$, it holds that 
        \begin{align*} 
         |C-A_i^{\rho}|+|D-B_i^{\rho}| \geq |\Delta_{\mathcal{S}_{([C,D]_\rho,\nu)}}(\mu_i^{\rho},\nu)|.
    \end{align*}
    \end{enumerate}
\end{lemma}
\begin{proof}
    For the first direction, assume $\mathcal{S}_{([C,D]_\rho,\nu)} \in \mathrm{SRep}$. Items 2. and 3. follow directly from Theorem \ref{thm:nonvanishingstandard}. To prove 1., we use Theorem again \ref{thm:nonvanishingstandard}: Assume $1\leq i <j\leq n_\rho$ such that $S_i^\rho$ and $S_j^\rho$ are connected in $\mathcal{S}$. If it is not the case that $i<i'\leq j$, then $\Delta_{\mathcal{S}_{([C,D]_\rho,\nu)}}(\mu_i^{\rho},\mu_j^\rho)=\Delta_{\mathcal{S}}(\mu_i^{\rho},\mu_j^\rho)$ and the desired inequality (\ref{ine}) holds. If $i<i'\leq j$, there are two cases: If $S_i^\rho$ and $S_j^\rho$ are connected in $\mathcal{S}_{([C,D]_\rho,\nu)}$, we have $\Delta_{\mathcal{S}}(\mu_i^{\rho},\mu_j^\rho) = \Delta_{\mathcal{S}_{([C,D]_\rho,\nu)}}(\mu_i^{\rho},\mu_j^\rho)$ by 1. in Lemma \ref{delta-lemma}. If $S_i^\rho$ and $S_j^\rho$ are not connected in $\mathcal{S}_{([C,D]_\rho,\nu)}$, by 2. in Lemma \ref{delta-lemma}
 \begin{align*} 
    |\Delta_{\mathcal{S}}(\mu_i^\rho,\mu_{j}^\rho)|
    &\leq |\Delta_{\mathcal{S}_{([C,D]_\rho,\nu)}}(\mu_i^\rho,\nu)| + |\Delta_{\mathcal{S}_{([C,D]_\rho,\nu)}}(\nu,\mu_{j}^\rho)| \\
    &\leq |C-A_i^\rho|+|D-B_i^\rho| + |A_j^\rho-C|+|B_j^\rho-D|\\
    &= |A_j^\rho-A_i^\rho|+|B_j^\rho-B_i^\rho|.
\end{align*}
For the other direction, the only case of Theorem \ref{thm:nonvanishingstandard} we have to consider is if $S_i^\rho$ and $S_j^\rho$ are connected in $\mathcal{S}_{([C,D]_\rho,\nu)}$ and $i<i'\leq j$. Now we have (1. in Lemma \ref{delta-lemma})
\begin{align*}
     \Delta_{\mathcal{S}_{([C,D]_\rho,\nu)}}(\mu_i^{\rho},\mu_j^\rho)=\Delta_{\mathcal{S}}(\mu_i^{\rho},\mu_j^\rho).
\end{align*}
\end{proof}
\begin{cor} \label{dec}
     Let $\mathcal{S}= \bigcup_{\rho \in C_\mathcal{S}} \{(( [A_i^{\rho},B_i^{\rho}]_\rho,\mu_i^{\rho}))_{i=1}^{n_\rho}\}\in \mathrm{SRep}$ and let $c,d\in \Z_{\geq 0 }$ and $\rho \in \mathrm{Cusp}^\perp (\mathrm{GL}_d (F))$ such that $\psi_{\mathcal{S}} \oplus (\rho \boxtimes S_c \boxtimes S_d)^{\oplus 2} $ is of good parity and $C=\frac{c+d}{2}-1$, $D=\frac{c-d}{2}$. Let $I \subset \{1, \ldots, n_\rho \}$ be the indices such that $$
        S_{i}^{\rho}\sim ([C,D]_\rho,\nu)
        $$  in $\mathcal{S}_{([C,D]_\rho,\nu)}$ (for any $\nu$) and denote for $i\in I$ by $N_i$ the set 
        $$
        \{ \nu \in \Z|\   |C-A_i^{\rho}|+|D-B_i^{\rho}| \geq |\Delta_{\mathcal{S}_{([C,D]_\rho,\nu)}}(\mu_i^{\rho},\nu)|  \}.
        $$
        If we denote 
        \begin{align*}
            N \ceq \{\nu \in \Z|\ |\nu|\leq d,\ |\Hat{\nu}|\leq c,\ \nu \equiv d \mod{2}   \} \cap \bigcap_{i \in I} N_i,
        \end{align*}
        then
        \begin{align*}
    u_\rho(c,d) \rtimes \pi(\mathcal{S}) \cong \bigoplus_{\nu \in N} \pi(\mathcal{S}_{([C,D]_\rho,\nu)}).
\end{align*}
Moreover $u_\rho(c,d) \rtimes \pi(\mathcal{S}) $ is irreducible if and only if $N$ is a singleton. The set $N$ has the form $\{ \nu, \nu+2, \ldots,\ \nu + 2n\}$ for some integers $\nu$ and $n$.
\end{cor}
\begin{proof}
    The first part is a direct consequence of Proposition \ref{decomp} and Lemma \ref{reduce}. Since all inequalities that elements $\nu$ of $N$ must satisfy are of the form $|\nu-x| \leq y$, their intersection is an interval of $\R$. Intersecting with $2\Z+d$ shows that $N$ has the stated form. 
\end{proof}
For the proof of Theorem \ref{thm:IRR}, we also need to understand how parabolic induction with multiple Speh representations behaves with representations of Arthur type. We have the following:
\begin{restatable}{thm}{multipleinsertions}
\label{prop:multipleinsertions}
    Let $\mathcal{S}= \bigcup_{\rho \in C_\mathcal{S}} \{(( [A_i^{\rho},B_i^{\rho}]_\rho,\mu_i^{\rho}))_{i=1}^{n_\rho}\}\in \mathrm{SRep}$. Let $([C_i,D_i]_{\rho_i},\nu_i)$ be extended segments for $i=1,...,k$ such that 
    \begin{align*}
        \psi_{\mathcal{S}} \oplus \bigoplus_{i=1}^k (\rho_i \boxtimes S_{c_i} \boxtimes S_{d_i})^{\oplus 2} 
    \end{align*}
    is of good parity and consider $\mathcal{S}_{([C_1,D_1]_{\rho_1},\nu_1),...,([C_k,D_k]_{\rho_k},\nu_k)}$. Then 
    \begin{align*}
        \mathcal{S}_{([C_1,D_1]_{\rho_1},\nu_1),...,([C_k,D_k]_{\rho_k},\nu_k)} \in \mathrm{SRep}
    \end{align*}
    if and only if 
    \begin{align*}
        \mathcal{S}_{([C_i,D_i]_{\rho_i},\nu_i)} \in \mathrm{SRep}
    \end{align*}
    for $i=1,...,k$ and
    \begin{align*}
        |C_l-C_m|+|D_l-D_m| \geq |\Delta_{\mathcal{S}_{([C_1,D_1]_{\rho_1},\nu_1),...,([C_k,D_k]_{\rho_k},\nu_k)}}(\nu_m,\nu_l)|
    \end{align*}
     for all pairs $([C_l,D_l]_{\rho_l},\nu_l)$ and $([C_m,D_m]_{\rho_m},\nu_m)$ of extended segments with $\rho_l \cong \rho_m$ that are connected in $\mathcal{S}_{([C_1,D_1]_{\rho_1},\nu_1),...,([C_k,D_k]_{\rho_k},\nu_k)}$. We denote by 
     \begin{align*}
         N(\mathcal{S},[C_1,D_1]_{\rho_1},...,[C_k,D_k]_{\rho_k} ) 
     \end{align*}
     the set of tuples $(\nu_1,...,\nu_k)$ such that $\mathcal{S}_{([C_1,D_1]_{\rho_1},\nu_1),...,([C_k,D_k]_{\rho_k},\nu_k)} \in \mathrm{SRep}$, according to the above criterion. 
\end{restatable}

\begin{proof}
    First, we prove the “$\Leftarrow$” -direction: \\
    We prove the statement by induction over $k$. If $k=1$, this is trivial. Assume now we have 
    \begin{align*}
        |C_l-C_m|+|D_l-D_m| \geq |\Delta_{\mathcal{S}_{([C_1,D_1]_{\rho_1},\nu_1),...,([C_k,D_k]_{\rho_k},\nu_k)}}(\nu_m,\nu_l)|
    \end{align*}
     for all pairs $([C_l,D_l]_{\rho_l},\nu_l)$ and $([C_m,D_m]_{\rho_m},\nu_m)$ of extended segments with $\rho_l \cong \rho_m$ that are connected in $\mathcal{S}_{([C_1,D_1]_{\rho_1},\nu_1),...,([C_k,D_k]_{\rho_k},\nu_k)}$. By Lemma \ref{delta-lemma}, this implies that
     \begin{align*}
        |C_l-C_m|+|D_l-D_m| \geq |\Delta_{\mathcal{S}_{([C_1,D_1]_{\rho_1},\nu_1),...,([C_{k-1},D_{k-1}]_{\rho_{k-1}},\nu_{k-1})}}(\nu_m,\nu_l)|
    \end{align*}
     for all pairs $([C_l,D_l]_{\rho_l},\nu_l)$ and $([C_m,D_m]_{\rho_m},\nu_m)$ of extended segments with $\rho_l \cong \rho_m$ that are connected in $\mathcal{S}_{([C_1,D_1]_{\rho_1},\nu_1),...,([C_{k-1},D_{k-1}]_{\rho_{k-1}},\nu_{k-1})}$.
     By induction hypothesis we have
    $$
    \mathcal{S}_{([C_1,D_1]_{\rho_1},\nu_1),...,([C_i,D_i]_{\rho_i},\nu_i),...,([C_{k-1},D_{k-1}]_{\rho_{k-1}},\nu_{k-1})}\in\mathrm{SRep}.
    $$
    Now we want to apply Lemma \ref{reduce} to 
    $$
    \mathcal{S}_{([C_1,D_1]_{\rho_1},\nu_1),...,([C_i,D_i]_{\rho_i},\nu_i),...,([C_{k-1},D_{k-1}]_{\rho_{k-1}},\nu_{k-1})}
    $$ 
    and the extended segment $([C_k,D_k]_{\rho_k},\nu_k)$.
    Condition 1. and 3. are satisfied. For condition 2., we have to prove that condition (\ref{nec-negative}) is satisfied at the extended segments $([C_k,D_k]_{\rho_k},\nu_k)$. But this follows from 
    \begin{align*}
        \mathcal{S}_{([C_k,D_k]_{\rho_k},\nu_k)}\in\mathrm{SRep},
    \end{align*}
    because the condition is exactly the same in that situation.\\
    On the other hand, assume $\mathcal{S}_{([C_1,D_1]_{\rho_1},\nu_1),...,([C_i,D_i]_{\rho_i},\nu_i),...,([C_k,D_k]_{\rho_k},\nu_k)}\in\mathrm{SRep}$. The inequalities for the extended segments $([C_i,D_i]_{\rho_i},\nu_i)$ follow immediately from Theorem \ref{thm:nonvanishingstandard}. It remains to show that $\mathcal{S}_{([C_i,D_i]_{\rho_i},\nu_i)}\in\mathrm{SRep}$ for all $i$. But we simply successively remove the extended segments $([C_j,D_j]_{\rho_j},\nu_j)$ for $j\neq i$. Because of Lemma \ref{basics} and Lemma \ref{reduce}, we have
    \begin{align*}
        &\mathcal{S}_{([C_1,D_1]_{\rho_1},\nu_1),...,([C_i,D_i]_{\rho_i},\nu_i),...,([C_k,D_k]_{\rho_k},\nu_k)}\in\mathrm{SRep}\\
        \Rightarrow  \qquad 
        &\mathcal{S}_{([C_1,D_1]_{\rho_1},\nu_1),...,([C_i,D_i]_{\rho_i},\nu_i),...,([C_{k-1},D_{k-1}]_{\rho_{k-1}},\nu_{k-1})}\in\mathrm{SRep}\\
        ... \\
        \Rightarrow \qquad
        &\mathcal{S}_{([C_i,D_i]_{\rho_i},\nu_i)}\in\mathrm{SRep}.\\
    \end{align*}

\end{proof}

\begin{cor} \label{Nindex}
Let $\mathcal{S}\in \mathrm{SRep}$. Let $c_i,d_i \in \Z_{\geq 0}$ and $\rho_i \in \mathrm{Cusp}^\perp (\mathrm{GL}_d (F))$ such that 
    \begin{align*}
        \psi_{\mathcal{S}} \oplus \bigoplus_{i=1}^k (\rho_i \boxtimes S_{c_i} \boxtimes S_{d_i})^{\oplus 2} 
    \end{align*}
    is of good parity. Then the parabolically induced representation
     \begin{align*}
        u_{\rho_1}(c_1,d_1) \times ... \times u_{\rho_k}(c_k,d_k) \rtimes\pi(\mathcal{S}) 
        \end{align*}
        decomposes as 
      \begin{align*}
\bigoplus_{(\nu_1,\ldots,\nu_k)} 
\pi\left(\mathcal{S}{([C_1,D_1]_{\rho_1},\nu_1),\ldots,([C_k,D_k]_{\rho_k},\nu_k)}\right),
\end{align*}
     where we sum over $(\nu_1,\ldots,\nu_k) \in N(\mathcal{S},[C_1,D_1]_{\rho_1},...,[C_k,D_k]_{\rho_k})$ determined in Theorem \ref{prop:multipleinsertions}.
    Moreover, the constituents on the right-hand side are pairwise non-isomorphic, i.e. the decomposition is multiplicity-free. 
\end{cor}
\begin{proof}
    This follows directly from Proposition \ref{prop:multipleinsertions} and Proposition \ref{decomp}. The sum is multiplicity free due to \cite[Proposition 4.2]{A2} (this in turn follows from the proof of \cite[Proposition 2.4.3]{Arthur}) and by Mœglin's result that Arthur packets are multiplicity-free \cite{Mœglin2}.
\end{proof}

\begin{cor} 
Let $\mathcal{S}\in \mathrm{SRep}$ and let $k\geq 2$ be an integer. Let $c_i,d_i \in \Z_{\geq 0}$ and $\rho_i \in \mathrm{Cusp}^\perp (\mathrm{GL}_d (F))$ such that 
    \begin{align*}
        \psi_{\mathcal{S}} \oplus \bigoplus_{i=1}^k (\rho_i \boxtimes S_{c_i} \boxtimes S_{d_i})^{\oplus 2} 
    \end{align*}
    is of good parity and such that $(c_i,d_i)\neq (c_j,d_j)$ if $i \neq j$. Assume moreover that $u_{\rho_i}(c_i,d_i) \rtimes \pi(\mathcal{S})$ is reducible for each $i=1,\ldots, k$. Then there is one of the Speh representations $u_{\rho_i}(c_i,d_i)$ and a constituent $\pi'$ of 
    \begin{align*}
        \bigtimes_{\substack{j=1 \\ j\neq i}}^k u_{\rho_j}(c_j,d_j) \rtimes \pi(\mathcal{S}),  
    \end{align*}
    such that 
      \begin{align*}
    u_{\rho_i}(c_i,d_i) \rtimes \pi'
\end{align*}
is reducible.
\end{cor}
\begin{proof}
Let us denote 
\begin{align*}
    N_i \ceq \{\nu_i \in 2\Z + d_i|\  \mathcal{S}_{([C_i,D_i],\nu_i)} \in \mathrm{SRep} \},
\end{align*}
which consist by Corollary \ref{dec} of consecutive integers of same parity. By assumption $|N_i| \geq 2$ for all $i$. The statement we want to prove is equivalent to the fact that $N(\mathcal{S},[C_1,D_1]_{\rho_1},...,[C_k,D_k]_{\rho_k})$ contains two elements $(\nu_1,\ldots, \nu_{i-1}, \nu_i, \nu_{i+1}, \ldots, \nu_k)$ and $(\nu_1,\ldots, \nu_{i-1}, \nu_i+2, \nu_{i+1}, \ldots, \nu_k)$ for some $i$. We call these \textit{adjacent lattice points}. We prove it by induction over $k$. \\
If $k=2$, we set $N\ceq N(\mathcal{S},[C_1,D_1]_{\rho_1},[C_2,D_2]_{\rho_2})$. If the two extended segments are not connected, we are done. Else, according to Theorem \ref{prop:multipleinsertions}, $N$ is given by
\begin{align*}
    N= \{ (\nu_1,\nu_2) \in N_1 \times N_2|\ | \Delta_{\mathcal{S}_{([C_1,D_1]_{\rho_1},\nu_1),([C_{2},D_{2}]_{\rho_{2}},\nu_{2})}}(\nu_1,\nu_2)| \leq |C_2-C_1| +|D_2-D_1|  \}.
\end{align*}
This inequality is of the shape
\begin{align*}
    |\nu_1 + \varepsilon \nu_2 + x| \leq y,
\end{align*}
where $\varepsilon = \pm 1$ and $x,y$ are integers with $y\geq 1$. Moreover, it holds that $\nu_1 + \varepsilon \nu_2 + x \equiv y \mod{2}$. 
Let $(\nu_1,\nu_2) \in N$. Define
\begin{align*}
    \delta \ceq - \mathrm{sgn}(\nu_1 + \varepsilon \nu_2 + x)
\end{align*}
with the convention $\mathrm{sgn}(0)=1$. If $\nu_1+2\delta \in N_1$, we are done, because $(\nu_1,\nu_2)$ and $(\nu_1+2\delta,\nu_2)$ are then adjacent lattice points in $N$. Similarly if $\nu_2+2\delta\varepsilon \in N_2$. So let us assume that neither is true. Since $|N_1|\geq 2$, we have $\nu_1-2\delta \in N_1$. There must be some $\nu_2' \in  N_2$ such that $(\nu_1-2\delta, \nu_2')\in N$. This means that the inequality
\begin{align*}
    |\nu_1- 2\delta + \varepsilon\nu_2' + x| \leq y
\end{align*}
holds. Now
\begin{align*}
    \nu_1- 2\delta + \varepsilon\nu_2' + x &= (\nu_1- 2\delta + \varepsilon\nu_2 + x) + \varepsilon(\nu_2' - \nu_2) \\
    &= (\nu_1- 2\delta + \varepsilon\nu_2 + x) + ( - 2\delta n)
\end{align*}
for some $n \in \Z_{\geq 0}$ and both bracketed expressions have the same sign. This implies 
\begin{align*}
    |\nu_1- 2\delta + \varepsilon\nu_2 + x| \leq y
\end{align*}
and hence $(\nu_1-2\delta,\nu_2)\in N$, which is an adjacent lattice point to $(\nu_1,\nu_2)$.
\\
Assume now that it holds for $k-1$. This means that there must be some $i$ such that 
$$
(\nu_1,\ldots, \nu_{i-1}, \nu_i, \nu_{i+1}, \ldots, \nu_{k-1})
$$ 
and 
$$
(\nu_1,\ldots, \nu_{i-1}, \nu_i+2, \nu_{i+1}, \ldots, \nu_{k-1})
$$ 
are in $N(\mathcal{S},[C_1,D_1]_{\rho_1},...,[C_{k-1},D_{k-1}]_{\rho_{k-1}})$. We set
\begin{align*}
    \mathcal{S}' \ceq \mathcal{S}_{([C_1,D_1]_{\rho_1},\nu_1),...,([C_{i-1},D_{i-1}]_{\rho_{i-1}},\nu_{i-1}),([C_{i+1},D_{i+1}]_{\rho_{i+1}},\nu_{i+1}),...,([C_{k-1},D_{k-1}]_{\rho_{k-1}},\nu_{k-1})},
\end{align*}
which must be in $\mathrm{SRep}$. Moreover, we denote
\begin{align*}
    N_j' \ceq \{\nu_j' \in 2\Z + d_j|\  \mathcal{S}'_{([C_j,D_j],\nu_j')} \in \mathrm{SRep} \}
\end{align*}
for $j=i,k$. By induction hypothesis, we have $\nu_i,\nu_i+2 \in N_{i}'$. If $|N_k'|=1$ (it must be non-empty), the lattice points 
$$
(\nu_1,\ldots, \nu_{i-1}, \nu_i, \nu_{i+1}, \ldots, \nu_{k-1},\nu_k)
$$ 
and 
$$
(\nu_1,\ldots, \nu_{i-1}, \nu_i+2, \nu_{i+1}, \ldots, \nu_{k-1},\nu_k)
$$ 
are adjacent (where $N_k'=\{\nu_k\}$) and must be in $N(\mathcal{S},[C_1,D_1]_{\rho_1},...,[C_{k},D_{k}]_{\rho_{k}})$. Else, if $|N_k'|\geq 1$ we see the existence of two adjacent lattice points in the same set by application of the basis of induction on $\mathcal{S}'$.
\end{proof}

\nopagebreak[4]

\section{The irreducibility criterion} \label{main}

Throughout this chapter, we will prove our main theorem:

\IRR*

\subsection{Outline of the proof}\label{outline}

One direction of this theorem is very easy to see. Namely, assume $u_1\times\ldots\times u_r \rtimes \pi$ is an irreducible representation. Assume (\ref{irr}) we false. If they exist, let $i<j$ be minimal indices such that $u_i\times u_j$ is reducible. But then 
\begin{align*}
    u_1 \times \ldots \times u_i \times \ldots \times u_j \times \ldots \times u_r \rtimes \pi \cong u_1 \times \ldots \times u_i \times u_j \times \ldots \times u_r \rtimes \pi
\end{align*}
were reducible because of the exactness of parabolic induction, a contradiction. If no such indices exist, we have 
\begin{align*}
    u_1 \times \ldots \times u_i \times \ldots \times u_r \rtimes \pi &\cong u_1 \times \ldots \times u_r \times u_i \rtimes \pi \\
    & \cong u_1 \times \ldots \times u_r \times u_i^\vee \rtimes \pi,
\end{align*}
which implies that $u_i \rtimes \pi$ and $u_i \times u_j^\vee$ are irreducible for all possible choices by an analogous argument. 

We will now explain the steps of the proof of the opposite direction, which will take up this entire chapter. Throughout the proof, we distinguish these cases:
\begin{de}
    Let $u=u_{\rho}(a,b)|\cdot|^{s}$ be an essentially Speh representation with $s\geq 0$. We say that $u$ is
    \begin{enumerate}
    \item of \textit{unitary type} if $s=0$,
        \item of \textit{small type} if $\rho \cong \rho^\vee$, $0<s\leq  \frac{a-1}{2}$ and $s \in \frac{1}{2}\Z$,
        \item of \textit{big type} if $\rho \cong \rho^\vee$, $s > \frac{a-1}{2}$ and $s \in \frac{1}{2}\Z$,
        \item of \textit{non-half-integral type} if $\rho \cong \rho^\vee$ and $s \not \in \frac{1}{2}\Z$ or if $\rho \not \cong \rho^\vee$.
    \end{enumerate}

\end{de}

Assume that $u_1,\ldots,u_n,\pi$ satisfy \eqref{irr}.
Since all the representations $u_i$ have pairwise irreducible parabolic induction, we may order them as we please and obtain the same parabolically induced representation (Theorem \ref{essred}). We group them according to type: Let $u_1,...,u_{h}$ denote the non-half-integral type representations, let $u_{h+1},...,u_{b}$ denote the big type representations, $u_{b+1},...,u_{s}$ the small type representations and $u_{s+1},...,u_{n}$ the unitary type representations. We denote 
\begin{align*}
    \Pi_{i} \ceq u_{n-i+1} \times \ldots \times u_n \rtimes \pi 
\end{align*}
for $i=0,\ldots,n$. 
The proof of Theorem \ref{thm:IRR} is structured as follows:
\begin{enumerate}
    \item Considering only the unitary type representations, the irreducibility of $\Pi_{n-s}$ is an easy consequence of the preceding chapter and we carry it out in Section \ref{general}. Moreover, it is of Arthur type again (Corollary \ref{multiple-spehs}).
    \item The main step is to prove the irreducibility of $\Pi_{n-b}$. First, we prove (Theorem \ref{still-irr}) that the representations $u_{b+1},\ldots, u_s, \Pi_{n-s}$ satisfy \eqref{irr} again. After some preparatory lemmas (Section \ref{prep}), we prove Theorem \ref{thm:IRR} in the case where all essentially Spehs are of small type (Proposition \ref{subrep}), which implies the irreducibility of $\Pi_{n-b}$.
    \item In Section \ref{bigcase} we will give a short inductive proof for the remaining cases (big and non-half-integral type) to show that $\Pi_n$ is irreducible.
\end{enumerate}

\subsection{General statements}\label{gen}

Let us prove a slightly altered irreducibility criterion of Atobe in \cite[Corollary 5.2]{A2}, which allows us to reduce the proof of irreducibility to socle irreducibility in many of the following cases (this is a standard argument in the theory):

\begin{prop} \label{ired}
 Let $\tau \in \mathrm{Irr}(GL)$ and $\pi\in\mathrm{Irr}(G)$.
Then $\tau \rtimes \pi$ is irreducible if and only if all of the following conditions hold:
\begin{itemize}
    \item[(i)] $\mathrm{soc}(\tau \rtimes \pi)$ is irreducible,
    \item[(ii)] $\mathrm{soc}(\tau^\vee \rtimes \pi)$ is irreducible,
    \item[(iii)] $\mathrm{soc}(\tau\rtimes \pi)\cong\mathrm{soc}(\tau^\vee  \rtimes \pi)$,
    \item[(iv)] $\mathrm{soc}(\tau  \rtimes \pi)$ appears in the composition series of the representation $\tau\rtimes \pi$ with multiplicity one.
\end{itemize}
\end{prop}
\begin{proof}

Note that if we assume that the representation $\tau \rtimes \pi$ is irreducible, then all four conditions follow directly.
Namely, an irreducible representation is isomorphic to its socle and, in the Grothendieck group $\mathscr{R}(G)$, we have $\tau \rtimes \pi = \tau^\vee  \rtimes \pi.$

On the other hand, let us assume the four conditions from the statement of the proposition hold.
We have 
\begin{align*}
    \mathrm{soc}\left(\tau \rtimes \pi\right) \cong & \mathrm{soc}\left(\tau^\vee  \rtimes \pi\right) 
   \cong  \mathrm{cosoc}\left(\tau \rtimes \pi^\vee\right)^\vee \\
   \cong &\mathrm{cosoc}\left(\tau \rtimes \pi^\vee\right)^{\mathrm{MVW}}
   \cong \mathrm{cosoc}\left(\tau\rtimes \pi\right).
\end{align*}
Now the socle and cosocle of $\tau\rtimes \pi$ are isomorphic, irreducible, and appear in its semisimplification with multiplicity one, which implies that the induced is irreducible.  
\end{proof}
\begin{lemma} \label{lemica}
Let $n\ge2$ be an integer, let $u_{\rho_i}(a_i,b_i)|\cdot|^{s_i}$ for $i=1,\ldots n$ be essentially Speh representations and let $\pi$ of Arthur type satisfy (\ref{irr}). Let us denote by $\rho^* \ceq \rho$ if $s_i\geq 0$ and $\rho^* \ceq \rho^\vee$ if $s_i<  0$.
Assume that $\mathrm{soc}(\bigtimes_{i=1}^n u_{\rho_i^*}(a_i,b_i)|\cdot|^{|s_i|} \rtimes \pi)$ is irreducible and appears in the composition series of the induced representation $\bigtimes_{i=1}^n u_{\rho_i^*}(a_i,b_i)|\cdot|^{|s_i|} \rtimes \pi$ with multiplicity one. Then the induced representation $\bigtimes_{i=1}^n u_{\rho_i}(a_i,b_i)|\cdot|^{s_i} \rtimes \pi$ is irreducible. 
\end{lemma}
\begin{proof}
The assumptions of (\ref{irr}) imply the following isomorphisms:
\begin{gather*}
    \bigtimes_{i=1}^n u_{\rho_i}(a_i,b_i)|\cdot|^{s_i} \rtimes \pi \cong u_{\rho_n}(a_n,b_n)|\cdot|^{s_n}\times \cdots\times u_{\rho_1}(a_1,b_1)|\cdot|^{s_1} \rtimes \pi \cong \\
    u_{\rho_n}(a_n,b_n)|\cdot|^{s_n}\times \cdots\times u_{\rho_2}(a_2,b_2)|\cdot|^{s_2} \rtimes (u_{\rho_1}(a_1,b_1)|\cdot|^{s_1} \rtimes \pi) \cong \\
    u_{\rho_n}(a_n,b_n)|\cdot|^{s_n}\times \cdots\times u_{\rho_2}(a_2,b_2)|\cdot|^{s_2} \rtimes (u_{\rho_1^\vee}(a_1,b_1)|\cdot|^{-s_1} \rtimes \pi) \cong \\
    u_{\rho_1^\vee}(a_1,b_1)|\cdot|^{-s_1}\times u_{\rho_n}(a_n,b_n)|\cdot|^{s_n}\times \cdots\times u_{\rho_2}(a_2,b_2)|\cdot|^{s_2}\rtimes \pi, \\
\end{gather*}
and similarly for any index $i$. Hence we can invert the sign of any $s_i$ if necessary. This shows that
\begin{gather*}
     \bigtimes_{i=1}^n u_{\rho_i}(a_i,b_i)|\cdot|^{s_i} \rtimes \pi\cong \bigtimes_{i=1}^n u_{(\rho_i^*)^\vee}(a_i,b_i)|\cdot|^{-|s_i|} \rtimes \pi \cong \bigtimes_{i=1}^n u_{\rho_i^*}(a_i,b_i)|\cdot|^{|s_i|} \rtimes \pi.
\end{gather*}
It follows that $\mathrm{soc}(\bigtimes_{i=1}^n u_{\rho_i^*}(a_i,b_i)|\cdot|^{|s_i|}  \rtimes \pi)\cong\mathrm{soc}(\bigtimes_{i=1}^n u_{(\rho_i^*)^\vee}(a_i,b_i)|\cdot|^{-|s_i|} \rtimes \pi)$. 
Note that
    \begin{align*}
        \tau \ceq \bigtimes_{i=1}^n u_{\rho_i^*}(a_i,b_i)|\cdot|^{|s_i|}
    \end{align*}
    is irreducible due to \cite[Corollary 4.3]{Cyclic}, because essentially Speh representations are square-irreducible (Theorem \ref{essred}).
Together with the assumptions of this lemma, it follows that all four conditions from Proposition \ref{ired} are satisfied for $\tau$ and $\pi$.
Hence the induced representation $\bigtimes_{i=1}^n u_{\rho_i^*}(a_i,b_i)|\cdot|^{|s_i|}  \rtimes \pi$ (and thus also $\bigtimes_{i=1}^n u_{\rho_i}(a_i,b_i)|\cdot|^{s_i}  \rtimes \pi$) is irreducible.
\end{proof}
Hence we see that in order to prove Theorem \ref{thm:IRR}, it suffices to prove it for $s_i \geq 0$ for all $i$ and to prove that the socle of the induced representation $\bigtimes_{i=1}^n u_{\rho_i}(a_i,b_i)|\cdot|^{s_i} \rtimes \pi$ is irreducible and of multiplicity one in its composition series. We assume $s_i$ to be non-negative for all $i$ from now on.

\subsection{The unitary type case} \label{general}
In this section we consider the case of Theorem \ref{thm:IRR} where all $u_i$ are unitary. We see that the assumptions $u_i \times u_j$ and $u_i \times u_j^\vee$ are always satisfied, since two Speh representations have irreducible induction (Theorem \ref{essred}). With the results of the last chapter, we can prove the following:
\begin{prop}\label{multiple-spehs}
    Let $\mathcal{S}\in \mathrm{SRep}$. Let $c_i,d_i \in \Z_{\geq 0}$ and $\rho_i \in \mathrm{Cusp}^\perp (\mathrm{GL}_d (F))$ such that 
    \begin{align*}
        \psi_{\mathcal{S}} \oplus \bigoplus_{i=1}^k (\rho_i \boxtimes S_{c_i} \boxtimes S_{d_i})^{\oplus 2} 
    \end{align*}
    is of good parity and such that $u_{\rho_i}(c_i,d_i) \rtimes\pi(\mathcal{S})$ is irreducible for all $i$. Then
     \begin{align*}
         u_{\rho_1}(c_1,d_1) \times ... \times u_{\rho_k}(c_k,d_k) \rtimes\pi(\mathcal{S})
     \end{align*}
     is irreducible.
\end{prop}
\begin{proof}
It follows from Proposition \ref{prop:multipleinsertions}, that for any 
$$
(\nu_1,...,\nu_k) \in N(\mathcal{S},[C_1,D_1]_{\rho_1},...,[C_k,D_k]_{\rho_k} )
$$ we must have $\pi(\mathcal{S}_{([C_i,D_i],_{\rho_i}\nu_i)})\neq 0$. But because of the irreducibility of $u_{\rho_i}(c_i,d_i) \rtimes\pi(\mathcal{S})$ and Proposition \ref{decomp}, this $\nu_i$ is unique. Hence $|N(\mathcal{S},[C_1,D_1]_{\rho_1},...,[C_k,D_k]_{\rho_k} )|\leq 1$ and irreducibility follows by Corollary \ref{Nindex}.
\end{proof}

Let us now treat the case of Arthur parameters which are not of good parity, which follows as an easy corollary. Namely, by Corollary \ref{bad-parity-is-easy}, for an Arthur type representation $\pi$ with Arthur parameter $\psi$ of good parity and an induced representation $u_\rho(a,b)\rtimes\pi$ such that $\psi \oplus (\rho \boxtimes S_{a} \boxtimes S_{b})^{\oplus 2}$ is not of good parity, we have that $u_\rho(a,b)\rtimes\pi$ is irreducible and again of Arthur type.

\begin{cor}\label{bad-parity-unitary}
Let $a_1,...a_k$, $b_1,...,b_k$ be non-negative integers, $u_i = u_{\rho_i}(a_i,b_i)$ Speh representations and $\pi$ a representation of Arthur type, such that $u_i\rtimes \pi$ is irreducible for each $i=1,\ldots,k$. Then the induced representation 
\begin{align*}
    u_1 \times ... \times u_k \rtimes \pi 
\end{align*}
is irreducible.
\end{cor}
\begin{proof} 
Since $\pi$ is of Arthur type, we have a decomposition 
    \begin{align*}
        \pi \cong \tau_{\psi_1} \rtimes \pi_0,
    \end{align*}
    where $\tau_{\psi_1}$ and $\pi_0$ are defined as in section \ref{a-packs} and let $\psi=\psi_1 \oplus \psi_0 \oplus \psi_1^\vee$ be the Arthur parameter (and its decomposition from Proposition \ref{bad-parityy}) corresponding to $\pi$. 
      Since Speh representations have pairwise irreducible parabolic induction, we may order them such that $1,...,s$ are exactly the indices for which $\psi_0 \oplus (\rho_i \boxtimes S_{a_i} \boxtimes S_{b_i})^{\oplus 2}$ is not of good parity. Since $u_i \times \tau_{\psi_1} \rtimes \pi_0 \cong  \tau_{\psi_1} \times u_i\rtimes \pi_0$ is irreducible by assumption, we have that also $u_i \rtimes \pi_0$ is irreducible for $i=s+1,\ldots,k$. We have
\begin{align*}
    \pi_1 \times ... \times \pi_k \rtimes \pi &\cong \pi_{1} \times ... \times \pi_{s} \times \bigtimes_{i =s+1}^k \pi_i \times \tau_{\psi_1} \rtimes \pi_0 \\
    & \cong \pi_{1} \times ... \times \pi_{s} \times \tau_{\psi_1} \times \bigtimes_{i =s+1}^k \pi_i  \rtimes \pi_0, 
\end{align*}

From Corollary \ref{multiple-spehs}, since the corresponding A-parameter is of good parity, it follows that the induced representation $\bigtimes_{i =s+1}^r \pi_i  \rtimes \pi_0 $ is irreducible, of Arthur type and of good parity. We set 
\begin{align*}
    \psi_1' \ceq (\rho_1 \boxtimes S_{a_1} \boxtimes S_{b_1}) \oplus  ... \oplus (\rho_s \boxtimes S_{a_s} \boxtimes S_{b_s})  \oplus \psi_1,
\end{align*} 
which is a direct sum of irreducible representations of $W_F \times \mathrm{SL}_2(\C) \times \mathrm{SL}_2(\C)$ which are not self-dual of the same type as $\psi$. 
By Proposition \ref{bad-parityy} mentioned above, it follows that
\begin{align*}
    &\pi_1 \times ... \times \pi_k \rtimes \pi \\
    \cong & \tau_{\psi_1'} \times \left( \bigtimes_{i =s+1}^r \pi_i  \rtimes \pi_0  \right)
\end{align*}
is an irreducible representation of Arthur type.
\end{proof}

\subsection{Preparatory Lemmas} \label{prep}

We prove three lemmas about derivatives, which we will need for the proof of Theorem \ref{thm:IRR} in the small type case.

\begin{lemma} \label{mult1}
Let
\begin{gather} \label{b0}
  \tau \ceq  u_{ess}\begin{pmatrix}
A \hspace{3mm} &  B \hspace{3mm} \\
\hspace{3mm} C & \hspace{3mm} D
\end{pmatrix}^{(\rho)},
\end{gather}
with $A+C>0$ and $D>0$ be an essentially Speh representation. Then
\begin{align*}
    M_{\rho|\cdot|^B,\ldots,\rho|\cdot|^D}^{max}(\tau)=L_{\rho|\cdot|^B,\ldots,\rho|\cdot|^D}^{max}(\tau)= u_{ess}\begin{pmatrix}
A \hspace{3mm} &  B-1 \hspace{3mm} \\
\hspace{3mm} C & \hspace{3mm} D-1
\end{pmatrix}^{(\rho)}
\end{align*}
and the amount of derivatives taken at each step is $1$ (except for the sequence $D_{Z_\rho[0,1]}^{max} \circ D_{\rho|\cdot|^1}^{max}$ where we first take $0$ and then $1$ derivative).
\end{lemma}
\begin{proof}
    From \eqref{lad} we successively see that 
    \begin{align*}
        M_{\rho|\cdot|^B,\ldots,\rho|\cdot|^{B+i}}^{max}(\tau)=L_{\rho|\cdot|^B,\ldots,\rho|\cdot|^{B+i}}^{max}(\tau) \\
        =L(\Delta_{\rho}[B-1,A],\ldots,\Delta_{\rho}[B+i-1,A+i],\\
     \Delta_{\rho}[B+i+1,A+i+1],\ldots,\Delta_{\rho}[D,C]),
    \end{align*}
    for $i=0,\ldots, D-B$ except when $B+i=0$ or $1$, in which case we rather have that the $\rho|\cdot|^1$-derivative does nothing and (after taking the $Z_\rho[0,1]$-derivatives)
    \begin{align*}
        M_{\rho|\cdot|^B,\ldots,\rho|\cdot|^{1}}^{max}(\tau)=L_{\rho|\cdot|^B,\ldots,\rho|\cdot|^{1}}^{max}(\tau) \\
        =L(\Delta_{\rho}[B-1,A],\ldots,\Delta_{\rho}[0,A-B+1],\\
     \Delta_{\rho}[2,A-B+2],\ldots,\Delta_{\rho}[D,C]),
    \end{align*}
\end{proof}
Let us now recall the definitions of the conditions $LC$ and $RC$ from \cite{LapidMinguez2}, which characterize the irreducibility of the parabolic induction of two $\mathrm{GL}$-representations, since we use them in the proof of the next lemma.

Let $\Delta$ be a segment and $\mathfrak{m}= \Delta_1+\ldots+ \Delta_r$ be a multi-segment.
We define $X_{\Delta;\mathfrak{m}}=\{j\in\{1,\ldots,r\} : \Delta \prec \Delta_j \}$ and $Y_{\Delta;\mathfrak{m}}=\{j\in\{1,\ldots,r\} : \overset{\leftarrow} {\Delta} \prec \Delta_j \} .$
Then $LC(\Delta,\mathfrak{m})$ is the condition that there exists an injective function $f: X_{\rho|\cdot|^{t};\mathfrak{m}} \to Y_{\rho|\cdot|^{t};\mathfrak{m}}$ such that $\Delta_{f(x)}\prec\Delta_{x}$, for all $x$.

Also, let us define $\tilde{X}_{\Delta;\mathfrak{m}}=\{j\in\{1,\ldots,r\} : \Delta_j \prec \Delta \}$ and $\tilde{Y}_{\Delta;\mathfrak{m}}=\{j\in\{1,\ldots,r\} : \overset{\leftarrow}{\Delta_j} \prec \Delta \}.$
Then $RC(\Delta,\mathfrak{m})$ is the condition that there exists an injective function $f: \tilde{X}_{\rho|\cdot|^{t};\mathfrak{m}} \to \tilde{Y}_{\rho|\cdot|^{t};\mathfrak{m}}$ such that $\Delta_{x}\prec\Delta_{f(x)}$, for all $x$.
From \cite[Proposition 4.15 (3)]{LapidMinguez2}, we have a criterion: 
    \begin{center}
        $L(\Delta) \times L(\mathfrak{m})$ is irreducible if and only if $LC(\Delta,\mathfrak{m})$ and $RC(\Delta,\mathfrak{m})$.
    \end{center}
\begin{lemma}
\label{glirred}
\begin{itemize}
    \item[(1)] Let $\uab|\cdot|^s$ denote an essentially Speh representation of small type and let $i,j$ be a pair of indices of the block of derivatives $D^{\max}_{(a,b,s)}$.
    Then the induced representation $\sigma^{ij} \times M^{\max}_{(a,b,s),<i,j}(\uab|\cdot|^s)$ is irreducible (recall the notation \eqref{sigma} for $\sigma^{ij}$). 
    \item[(2)] Let $\pi\in \mathrm{Irr}(G_n)$ be an irreducible representation and let $\displaystyle 
    u_{\rho}(a,b)|\cdot|^{s}$ be an essentially Speh representation of small type. 
    If $\pi'$ is an irreducible subrepresentation of $\displaystyle 
    u_{\rho}(a,b)|\cdot|^{s}\rtimes\pi$, then $d^{ij}(\pi')$ equals:
    $$1_{ij}+d^{ij}(\pi)$$
    for every pair $i,j$ of indicies in the block of derivatives $D_{(a,b,s)}^{\max}$.
\end{itemize}
\end{lemma}
\begin{proof}
\begin{itemize}
    \item[(1)]
    We consider the initial block of derivatives, since the argument afterwards is analogous. I.e.~we prove that the representation $$\rho|\cdot|^{B+s+i}\times (M^{\max}_{\rho |\cdot|^{B+s+i-1}}\circ\ldots\circ M^{\max}_{\rho |\cdot|^{B+s}})(\uab |\cdot|^s)$$
    is irreducible, for all $i\in\{0,1,\ldots,b-1\}$.
    
    Firstly, let us denote
    $ (M^{\max}_{\rho |\cdot|^{B+s+i-1}}\circ\ldots\circ M^{\max}_{\rho |\cdot|^{B+s}})(\uab |\cdot|^s) = L(\mathfrak{m}), $
    for $\mathfrak{m}=\Delta_b+\ldots+\Delta_{b-i+1}+\Delta_{b-i}+\ldots+\Delta_1$.
    Then $\Delta_b= [B+s-1,-A+s]_{\rho},\ldots,\Delta_{b-i+1}=[B+s+i-2,-A+s+i-1]_{\rho}$ and the unaltered segments $\Delta_{b-i}= [B+s+i,-A+s+i]_{\rho},\ldots,\Delta_1= [A+s,-B+s]_{\rho}$.
    As discussed in Lemma \ref{mult1}.
    For $i\ge 0$, $L(\mathfrak{m})$ is a ladder representation. 
    As recalled before the lemma, from \cite[Proposition 4.15 (3)]{LapidMinguez2} we have a criterion: 
    \begin{center}
        $\rho|\cdot|^{B+s+i}\times L(\mathfrak{m})$ is irreducible if and only if $LC(\rho|\cdot|^{B+s+i},\mathfrak{m})$ and $RC(\rho|\cdot|^{B+s+i},\mathfrak{m})$.
    \end{center}
    Let us first check the condition $LC(\rho|\cdot|^{B+s+i},\mathfrak{m})$.
    If $X_{\rho|\cdot|^{B+s+i};\mathfrak{m}}=\emptyset$, then the criterion holds trivially.
    Otherwise, since the minimal exponents in the segments $\Delta_j$ increase by one, it is a singleton $\{j_0\}$. 
    For the same reason, if $Y_{\rho|\cdot|^{B+s+i};\mathfrak{m}}$ is non-empty, it is equal to $\{j_0+1\}$.
    It would be empty only if $j_0=b$, which implies $\min(\Delta_b)=-A+s=B+s+i+1$.
    This is equivalent to $i=-a$, which is a contradiction since $i\ge 0$ and $-a<0$. 

    Let us now check the condition $RC(\rho|\cdot|^{B+s+i},\mathfrak{m})$.
    Since the maximal exponents of segments in $\mathfrak{m}$ increase and $\max(\Delta_{b-i+1})=B+s+i-2<B+s+i=\max(\Delta_{b-i})$, we get that $\tilde{X}_{\rho|\cdot|^{B+s+i};\mathfrak{m}}=\emptyset$, which implies that the matching function exists trivially. Similar considerations hold if $B+s+i=0,1$ where we have that $\sigma^{ij}$ is $\rho|\cdot|^1$ and $Z_\rho[0,1]$ respectively: 

    If $B+s+i=0$, we have $(M^{\max}_{\rho |\cdot|^{-1}}\circ\ldots\circ M^{\max}_{\rho |\cdot|^{B+s}})(\uab |\cdot|^s) = L(\mathfrak{m}), $
    for $\mathfrak{m}=\Delta_b+\ldots+\Delta_{b-i+1}+\Delta_{b-i}+\ldots+\Delta_1$.
    Then $\Delta_b= [B+s-1,-A+s]_{\rho},\ldots,\Delta_{b-i+1}=[-2,-A+s+i-1]_{\rho}$ and the unaltered segments $\Delta_{b-i}= [0,-A+s+i]_{\rho},\ldots,\Delta_1= [A+s,-B+s]_{\rho}$. Again $i\geq 0$ implies $LC(\rho|\cdot|^1,\mathfrak{m})$ and $RC(\rho|\cdot|^1,\mathfrak{m})$ holds because we have an injective best matching for $\tilde{X}_{\rho|\cdot|^{1};\mathfrak{m}}=\{ b-i \}$ and $\tilde{Y}_{\rho|\cdot|^{1};\mathfrak{m}}=\{ b-i-1 \}$.

    Now if $B+s+i=1$ (in the case where $\sigma^{ij}=Z_\rho[0,1]$), we have 
    $$
    (M^{\max}_{\rho |\cdot|^{1}}\circ\ldots\circ M^{\max}_{\rho |\cdot|^{B+s}})(\uab |\cdot|^s) = L(\mathfrak{m}), 
    $$
    for the same multi-segment $\mathfrak{m}$ as above. Now we use \cite[Corollary 4.14]{LapidMinguez2} for $Z_\rho[0,1]$ and $Z(\mathfrak{m}^t)$, where $\mathfrak{m}^t$ is given by the MW-algortihm \cite{zelev}). In our case, $\mathfrak{m}^t= \Delta_a'+\ldots +\Delta_1'$ with 
    \begin{align*}
        \Delta_1'&=[A+s,0]_\rho \\
        \Delta_i'&=[ A+s+1-i, B+s+1-i ]_\rho\ for\ i=2,\ldots, a.
    \end{align*}
     Now $LC([1,0]_\rho,\mathfrak{m})$ holds trivially and $RC([1,0]_\rho,\mathfrak{m})$ holds because if $\tilde{X}_{[1,0]_\rho;\mathfrak{m}}=\{ i \}$ then $\tilde{Y}_{[1,0]_\rho;\mathfrak{m}}=\{ i-1 \}$.
    \item[(2)]
    First, note that we have the following embeddings:
    \begin{gather*}
        \pi'\hookrightarrow\displaystyle u_{\rho}(a,b)|\cdot|^{s}\rtimes\left(\pi\right)\overset{(\ast)}{\hookrightarrow} \\
        u_{\rho}(a,b)|\cdot|^{s}\times (\sigma^{11})^{d^{11}(\pi)}\rtimes D^{\max}_{\sigma^{11}}(\pi) \overset{(\ast \ast)}{\cong}\\
        (\sigma^{11})^{d^{11}(\pi)}\times u_{\rho}(a,b)|\cdot|^{s} \rtimes D^{\max}_{\sigma^{11}}(\pi)\hookrightarrow\\
        (\sigma^{11})^{d^{11}(\pi)+1}\times L^{\max}_{\sigma^{11}}(u_{\rho}(a,b)|\cdot|^{s}) \rtimes D^{\max}_{\sigma^{11}}(\pi) = \\
        (\sigma^{11})^{d^{11}(\pi)+1}\times M^{\max}_{\sigma^{11}}(u_{\rho}(a,b)|\cdot|^{s}) \rtimes D^{\max}_{\sigma^{11}}(\pi).
    \end{gather*}
    The embedding $(\ast)$ follows by the theory of derivatives in Section \ref{derivatives}. Here we use the assumption that the induced representation $\pi$ is irreducible. 
    Furthermore, the isomorphism $(\ast\ast)$ holds because of the part (1) of this lemma. We have also used the fact that $M$ is here the same as the maximal left derivative by Lemma \ref{b0}. This shows 
    \begin{align*}
        d^{11}(\pi')=d^{11}(\pi)+1.
    \end{align*}
    Assume we know this for all indices up to $i,j$.
    Repeating these arguments for all the derivatives in $D_{(a,b,s)}^{\max}$, we get the embeddings:
    \begin{gather} \label{zf2}
        \pi'\hookrightarrow (\sigma^{11})^{d^{11}(\pi)+1} \times \cdots\times \\
        (\sigma^{ij})^{d^{ij}(\pi)+1_{ij}} \times M^{\max}_{(a,b,s),\leq i,j}(\uab|\cdot|^s)\rtimes D^{\max}_{(a,b,s),\leq i,j}(\pi), \nonumber
    \end{gather}
    which implies that 
\begin{gather*} 
        \pi'\hookrightarrow (\sigma^{11})^{d^{11}(\pi)+1} \times \cdots\times \\
        (\sigma^{ij})^{d^{ij}(\pi)+1_{ij}} \times \pi_{ij}, \nonumber
    \end{gather*}
    for some $\pi_{ij} \leq M^{\max}_{(a,b,s),\leq i,j}(\uab|\cdot|^s)\rtimes D^{\max}_{(a,b,s),\leq i,j}(\pi).$ By Frobenius reciprocity, we have that (for the appropriate parabolic subgroup $P$)
    \begin{align*}
        \mathrm{Jac}_P(\pi') \geq &(\sigma^{11})^{d^{11}(\pi)+1} \otimes \cdots\otimes \\
        &(\sigma^{ij})^{d^{ij}(\pi)+1_{ij}} \otimes \pi_{ij}.
    \end{align*}
    This shows $d^{ij}(\pi') \geq d^{ij}(\pi)+1_{ij}$.
    We claim that 
    \begin{align*}
        D_{(a,b,s),\leq i,j}^{max}(\pi') \leq M^{\max}_{(a,b,s),\leq i,j}(\uab|\cdot|^s)\rtimes D^{\max}_{(a,b,s),\leq i,j}(\pi). 
    \end{align*}
    For the indices $11$ this is clear, and assuming it for all indices up to $ij$ we see that 
    \begin{align*}
        D_{(a,b,s),< i,j}^{max}(\pi') \leq (\sigma^{ij})^{d^{ij}(\pi)+1_{ij}} \times M^{\max}_{(a,b,s),\leq i,j}(\uab|\cdot|^s)\rtimes D^{\max}_{(a,b,s),\leq i,j}(\pi), 
    \end{align*}
    which implies $d^{ij}(\pi') \leq d^{ij}(\pi)+1_{ij}$ and hence $d^{ij}(\pi') =d^{ij}(\pi)+1_{ij}$. 
    Moreover, we see that 
    \begin{align*}
        D_{(a,b,s),\leq i,j}^{max}(\pi') = \pi_{ij} \leq M^{\max}_{(a,b,s),\leq i,j}(\uab|\cdot|^s)\rtimes D^{\max}_{(a,b,s),\leq i,j}(\pi).
    \end{align*}
\end{itemize}
\end{proof}
\begin{lemma} \label{jmlema}
    For $i=1,2$, let $u_{\rho_i}(a_i,b_i)|\cdot|^{s_i}$ denote essentially Speh representations of small type.
    Assume that $B_1+s_1\le B_2+s_2$. Then 
    \begin{gather} \label{e1}
          m_2^{ij}(u_{\rho_1}(a_1,b_1)|\cdot|^{s_1})= 
           m_2^{ij}(u_{\rho_1}(2s_1,b_1)|\cdot|^{\frac{a_1}{2}})+ m_2^{ij}(u_{\rho_1}(a_1-2s_1,b_1)), 
    \end{gather}
    for each pair of indices $i,j$ in the composition of the derivatives $M^{\max}_{(a_2,b_2,s_2)}$.
\end{lemma}
\begin{proof}
Let us recall that $M^{\max}_{(a_2,b_2,s_2)}$ equals    $$M^{\max}_{\rho_2|\cdot|^{B_2-s_2+1},\ldots,\rho_2|\cdot|^{A_2-s_2+1}} \circ \cdots \circ M^{\max}_{\rho_2|\cdot|^{B_2+s_2},\ldots,\rho_2|\cdot|^{A_2+s_2}}.$$
Clearly, in case $\rho_1 |\cdot|^{B_1+s_1}$ is not in the line $\Z_{\rho_2}$, the equality (\ref{e1}) holds with zeros on both sides of the equality.  
Hence, in the sequel, we denote $\rho\ceq \rho_1 \cong \rho_2$.
Moreover, we denote blocks of derivatives with
$$D_i\ceq M^{\max}_{\rho|\cdot|^{A_2+s_2-i}}\circ\cdots\circ M^{\max}_{\rho|\cdot|^{B_2+s_2-i}}.$$
In the statement of the lemma, we assume $B_1+s_1\le B_2+s_2$.
If $\pi$ denotes a representation $u_{\rho}(a_1,b_1)|\cdot|^{s_1}, u_{\rho}(2s_1,b_1)|\cdot|^{\frac{a_1}{2}}$ or $u_{\rho}(a_1-2s_1,b_1)$, then $D_i(\pi)=\pi$, for all indices $i$ such that the minimal exponent in $D_i$ is greater than $B_1+s_1$.
Hence, all the multiplicities in (\ref{e1}) are equal to zero and the equality holds. 
In this way, without loss of generality, we can assume $B_1+s_1=B_2+s_2$. Note that there are two possibilities: $A_1+s_1\le A_2+s_2$ or $A_1+s_1>A_2+s_2$.
Vaguely speaking in the first case, the block of derivatives $D_i$ changes all the segments in the Langlands data of the representations, while in the second case, there is a block of segments which stay the same as in the original representation.

In the first case, we can apply Lemma \ref{mult1} for each pair of indices $i,j$ in the blocks 
$$D_0,\ldots,D_{2s_1-1}.$$ 
Then
\begin{gather} \label{first1}
    m_2^{ij}(u_{\rho}(a_1,b_1)|\cdot|^{s_1})= m_2^{ij}(u_{\rho}(2s_1,b_1)|\cdot|^{\frac{a_1}{2}})= 1_{ij} \text{ and } m_2^{ij}(u_{\rho}(a_1-2s_1,b_1))=0.
\end{gather}
Also, we have 
\begin{gather*}
    (D_{2s_1-1}\circ\cdots\circ D_0)(u_{\rho}(a_1,b_1)|\cdot|^{s_1})=u_{\rho}(a_1-2s_1,b_1), \\
    (D_{2s_1-1}\circ\cdots\circ D_0)(u_{\rho}(2s_1,b_1)|\cdot|^{\frac{a_1}{2}})=\textit{1}, \\
    (D_{2s_1-1}\circ\cdots\circ D_0)(u_{\rho}(a_1-2s_1,b_1))=u_{\rho}(a_1-2s_1,b_1).
\end{gather*}
Since we have two equal representations and one trivial representation, it trivially follows that the remaining multiplicities in the equality (\ref{e1}) are the same on both sides.

In the second case, we will determine the desired multiplicities and define the Langlands data of the representations appearing in the analysis.
The argument follows the lines of the proof of Lemma \ref{mult1} and we leave the details to the reader. 
We divide the composition of blocks of the derivatives in $M^{\max}_{(a_2,b_2,s_2)}$ into three parts. 
With $M_i$, $i=1,2,3$, we denote the composition of the blocks with the minimal exponents equal to:
\begin{enumerate}
    \item $B_2+s_2=B_1+s_1,\ldots,B_1-s_1+1$ in $M_1$,
    \item $B_1-s_1,\ldots,-A_1+s_1$ in $M_2$,
    \item $-A_1+s_1-1,\ldots,B_2-s_2+1$ in $M_3$.
\end{enumerate}
Throughout the analysis, we use the inequalities $ B_1+s_1>B_1-s_1>-A_1+s_1-1$, which follow from the assumption of the small case.

In the appliance of $M_1$, we have $m_2^{ij}(u_\rho(a_1,b_1)|\cdot|^{s_1})=m_2^{ij}(u_\rho(2s_1,b_1)|\cdot|^{\frac{a_1}{2}})=1_{ij}$ and $m_2^{ij}(u_\rho(a_1-2s_1,b_1))=0$. 
The fact that $M_1(u_\rho(a_1-2s_1,b_1))=u_\rho(a_1-2s_1,b_1)$ follows from the definition of the derivatives in $M_1$ and the representation itself.

Let us denote with $\mathfrak{m}_1, \mathfrak{m}_2$ (resp. $\mathfrak{n}_1, \mathfrak{n}_2$) the multisets of segments from the Langlands data of $u_\rho(a_1,b_1)|\cdot|^{s_1}$ (resp. $u_\rho(2s_1,b_1)|\cdot|^{\frac{a_1}{2}}$), such that in $\mathfrak{m}_1$ (resp. $\mathfrak{n}_1$), the maximal exponents are $B_2+s_2,\ldots,A_2+s_2$ and in $\mathfrak{m}_2$ (resp. $\mathfrak{n}_2$) are $A_2+s_2+1,\ldots,A_1+s_1$.
Then $$ u_\rho(a_1,b_1)|\cdot|^{s_1} = L(\mathfrak{m}_1+\mathfrak{m}_2) \text{ and } u_\rho(2s_1,b_1)|\cdot|^{\frac{a_1}{2}}= L(\mathfrak{n}_1+\mathfrak{n}_2). $$
Note that $M_1(u_\rho(a_1,b_1)|\cdot|^{s_1})$ equals $L(\mathfrak{m}_1'
\mathfrak{m}_2)$, where $\mathfrak{m}_1'$ is the multiset of segments from $\mathfrak{m}_1$ with maximal exponents respectively changed to $B_1-s_1,\ldots, A_2+s_2-2s_1$.
Also, we get $M_1(u_\rho(2s_1,b_1)|\cdot|^{\frac{a_1}{2}})=L(\mathfrak{n}_2)$.
More precisely, we have 
\begin{gather*}
    \mathfrak{m}_1'= [B_1-s_1,-A_1+s_1]_{\rho}+\ldots+[A_2+s_2-2s_1,-A_1+s_1+b_2-1]_{\rho},\\
    \mathfrak{m}_2= [A_2+s_2+1,-A_1+s_1+b_2]_{\rho}+\ldots+[A_1+s_1,-B_1+s_1]_{\rho},\\
    \mathfrak{n}_2= [A_2+s_2+1,B_1-s_1+b_2+1]_{\rho}+\ldots+[A_1+s_1,A_1-s_1+1]_{\rho}.
\end{gather*}
In the appliance of $M_2$, we have 
$$
m_2^{ij}(u_\rho(a_1,b_1)|\cdot|^{s_1})=m_2^{ij}(u_\rho(a_1-2s_1,b_1))
$$
and 
$$
m_2^{ij}(u_\rho(2s_1,b_1)|\cdot|^{\frac{a_1}{2}})=0.
$$
The fact that 
$$M_2(L(\mathfrak{n}_2))=M_2(M_1(u_\rho(2s_1,b_1)|\cdot|^{\frac{a_1}{2}}))=M_1(u_\rho(2s_1,b_1)|\cdot|^{\frac{a_1}{2}})=L(\mathfrak{n}_2)$$ follows since the exponent $x$ in the constituent $\rho|\cdot|^x\otimes\pi$ of the Jacquet module of $L(\mathfrak{n}_2)$ is either smaller than the minimal exponent (equal to $-A_1+s_1$) in $M_2$ or bigger that the maximal one (equal to $B_1-s_1+b_2-1$).
Namely, $x$ could be $A_2+s_2+1$ or $-A_1+s_1-1$. The first one is bigger than $B_1-s_1+b_2-1$ since we have $B_2+s_2>B_1-s_1$ and the other one follows directly.

Let us show that $m_2^{ij}(u_\rho(a_1,b_1)|\cdot|^{s_1})=m_2^{ij}(u_\rho(a_1-2s_1,b_1))$ for each pair of indices $i,j$ in $M_2$.
We note that the sequences of the minimal exponents in the segments from the Langlands data of $M_1(u_\rho(a_1,b_1)|\cdot|^{s_1})=L(\mathfrak{m}_1'+\mathfrak{m}_2)$ and $M_1(u_\rho(2s_1,b_1)|\cdot|^{\frac{a_1}{2}})=L(\mathfrak{n}_2)$ are the same and equal to $-A_1+s_1,\ldots,-B_1+s_1$. 
Moreover, the smallest $b_2$ maximal exponents of the segments in the Langlands data of the same representations also form the same sequences $B_1-s_1,\ldots,A_2+s_2-2s_1$.
Since each block of derivatives in $M_2$ has $b_2$ derivatives, we conclude the wanted equality of the multiplicities.

If $2b_2\ge b_1$, then $M_2(u_{\rho}(a_1-2s_1,b_1))=1$.
Namely, there are $b_1$ segments in the Langlands data of $u_{\rho}(a_1-2s_1,b_1)$, and hence in case $2b_2\ge b_1$, the derivatives from $M_2$ cancel out every segment.
Furthermore, it follows easily that 
$$M_2(M_1(u_{\rho}(a_1,b_1)|\cdot|^{s_1}))=M_2(L(\mathfrak{m}_1'+\mathfrak{m}_2))=L(\mathfrak{n}_2)=M_2(M_1(u_{\rho}(2s_1,b_1)|\cdot|^{\frac{a_1}{2}})).$$
If $2b_2<b_1$, we have that $M_2(M_1(u_{\rho}(a_1-2s_1,b_1)))$ equals
$$ L(\Delta_{\rho}[B_1-s_1+b_2, -A_1+s_1+b_2],\ldots,\Delta_{\rho}[A_1-s_1-b_2, -B_1+s_1-b_2]). $$

If the block $M_3$ is non-empty, we have
\begin{align*}
    A_2-s_2+1 &\geq -B_2 +s_2-1 \geq A_1-s_1 \geq -B_1+s_1\\
    \Rightarrow b_2 &\geq  b_1 \\
    \Rightarrow A_2+s_2 & \geq A_1+s_1,
\end{align*}
which means we are actually in the first case considered at the beginning of this proof.
\end{proof}

\subsection{The small type case}\label{smallcase}

We will often use the following result by Atobe, which is also the main tool in the induction step in Section \ref{bigcase}.
\begin{prop} (\cite[Proposition 3.4.]{A2}) \label{3.4} \\
    For $\tau\in\mathrm{Irr}(G)$, let $\uab|\cdot|^s$ be an essentially Speh representation of non-half-integral type or of big type.
    Then the socle $\mathrm{soc}(\uab|\cdot|^s\rtimes\tau)$ is irreducible. Moreover, it appears in the composition series of $\uab|\cdot|^s\rtimes\tau$ with multiplicity one.
\end{prop}
\begin{proof}
    This follows from \cite[Proposition 3.4.]{A2}, except for the case when $\rho \not \cong \rho^\vee$. The proof works exactly the same as the one of \cite[Proposition 3.4.]{A2} for $s\not \in \frac{1}{2}\Z$ in this case. 
\end{proof}

Simple arguments from the representation theory of the general linear group show that for essentially Speh representations of small type $u_{\rho_i}(a_i,b_i)|\cdot|^{s_i}\rtimes\pi \hookrightarrow u_{\rho_i}(2s_i,b_i)|\cdot|^{a_i/2}\times u_{\rho_i}(a_i-2s_i,b_i)\rtimes\pi$. 
From Atobe's work, we have the following
\begin{prop} (\cite[Proposition 3.6]{A2})\label{3.6} \\
    Let $\pi' \ceq u_{\rho_i}(a_i,b_i)|\cdot|^{s_i}\rtimes\pi$ be an irreducible representation for $s_i$ of small type and $\pi$ of Arthur type. Then there exists a unique irreducible summand $\pi_{\mu_i}$ of $u_{\rho_i}(a_i-2s_i,b_i)\rtimes\pi$ such that $\pi'\cong \mathrm{soc}(u_{\rho_i}(2s_i,b_i)|\cdot|^{a_i/2}\rtimes \pi_{\mu_i})$.
    Moreover, $\pi'$ appears in the composition series of $u_{\rho_i}(a_i,b_i)|\cdot|^{s_i}\rtimes\pi$ with multiplicity one.
\end{prop}

Note that $\mathrm{soc}(u_{\rho_i}(2s_i,b_i)|\cdot|^{a_i/2}\rtimes \pi_{\mu_i})$ is irreducible by Proposition \ref{3.4} (it is of big type). The characterization of the representation $\pi_{\mu_i}$ is given in Proposition 3.7 of \cite{A2}.
We formulate this characterization adjusted to our setting in the following proposition.
\begin{prop}
    \label{char}
    Let $u_{\rho_i}(a_i,b_i)|\cdot|^{s_i}\rtimes\pi$ be a representation for $s_i$ of small type and $\pi$ of Arthur type. 
    For each pair of indices $i_0,j_0$ in the composition of derivatives $D_{(a_i,b_i,s_i)}$, let $k_{i_0,j_0}\ceq d_i^{i_0j_0}(\pi)$.  
    Moreover, let $\mu_i'$ be one of the parameters from the summation indices in Definition \ref{pi-ein-mu}. Then the representation $\mathrm{soc}(u_{\rho_i}(2s_i,b_i)|\cdot|^{a_i/2}\rtimes \pi_{\mu_i'})$ is a subrepresentation of $u_{\rho_i}(a_i,b_i)|\cdot|^{s_i}\rtimes\pi$  if and only if $D^{(k_{i_0,j_0})}_{(a_i,b_i,s_i)}(\pi_{\mu_i'})\neq 0$ for all indices $i_0,j_0$.
\end{prop}
\begin{rmk} \label{eqmult}
   We will show that if $\mathrm{soc}(u_{\rho_1}(2s_1,b_1)|\cdot|^{a_1/2}\rtimes \pi_{\mu_1})$ is a subrepresentation of the induced representation $u_{\rho_1}(a_1,b_1)|\cdot|^{s_1}\rtimes\pi$, then 
   \begin{gather} \label{mult-same}
       d_1^{ij}(\pi_{\mu_1})=d_1^{ij}(\pi), 
   \end{gather}
   for each pair of indices $i,j$.
   Using the notation from the Proposition \ref{char}, from the same proposition we get $D^{(k_{i,j})}_{(a_1,b_1,s_1)}(\pi_{\mu_1})\neq 0$, for each pair of indices $i,j$. 
   Note that $m_1^{ij}(u_{\rho_1}(a_1-2s_1,b_1))=0$, for each pair of indices $i,j$. 
   This is equivalent to $M^{\max}_{(a_1,b_1,s_1)}(u_{\rho_1}(a_1-2s_1,b_1))=u_{\rho_1}(a_1-2s_1,b_1)$.
   Since $\pi_{\mu_1}$ is a subquotient of $u_{\rho_1}(a_1-2s_1,b_1)\rtimes\pi$ and $D^{(k_{1,1})}_{\rho_1|\cdot|^{B_1+s_1}}(\pi_{\mu_1})\neq 0$, from the exactness of the Jacquet functor we conclude
   $ D^{(k_{1,1})}_{\rho_1|\cdot|^{B_1+s_1}}(\pi_{\mu_1}) \le u_{\rho_1}(a_1-2s_1,b_1)\rtimes D^{(k_{1,1})}_{\rho_1|\cdot|^{B_1+s_1}}(\pi) $ and 
   $$ \text{mult}_{ij}(D_{\rho_1|\cdot|^{B_1+s_1}},\pi_{\mu_1})=k_{1,1}=\text{mult}_{ij}(D_{\rho_1|\cdot|^{B_1+s_1}},\pi). $$
   Now we repeat these arguments for each $i,j$-derivative in $D^{\max}_{(a_1,b_1,s_1)}$.

   Note that in the Proposition \ref{3.6}, for the irreducible $u_{\rho_1}(a_1,b_1)|\cdot|^{s_1}\rtimes\pi$, the equality (\ref{mult-same}) characterizes the parameter $\mu_1$. Namely, if there were two parameters $\mu_1'$ and $\mu_1''$ such that 
   $$ d_1^{ij}(\pi_{\mu_1'})=d_1^{ij}(\pi)=d_1^{ij}(\pi_{\mu_1''}) $$
   for each pair of indices $i,j$, Proposition \ref{char} implies that both $\mathrm{soc}(u_{\rho_i}(2s_i,b_i)|\cdot|^{a_i/2}\rtimes \pi_{\mu_i'})$ and $\mathrm{soc}(u_{\rho_i}(2s_i,b_i)|\cdot|^{a_i/2}\rtimes \pi_{\mu_i''})$ are subrepresentations of $u_{\rho_1}(a_1,b_1)|\cdot|^{s_1}\rtimes\pi$.
   The irreducibility of $u_{\rho_1}(a_1,b_1)|\cdot|^{s_1}\rtimes\pi$ now implies that these are isomorphic representations. 
   Considering their Jacquet modules, one can easily show that then $ \pi_{\mu_i'} \cong \pi_{\mu_i''}$, a contradiction.
\end{rmk}

 Next, we will show how we can reduce the proof irreducibility of $\Pi_{n-b}$ to the case where there are no unitary Spehs $u_{s+1},\ldots u_n$. Namely, the idea is to consider
\begin{align*}
     \Pi_{n-s} &= u_{s+1} \times \ldots \times u_n \rtimes \pi
 \\
 &= \left(\bigtimes_{i=s+1}^n u_{\rho_i}(a_i,b_i)\right)\rtimes\pi,
 \end{align*}
 which is irreducible and of Arthur type, as we just proved.
 If we can then show that $u_i \rtimes \Pi_{n-s}$ is irreducible for $1 \leq i \leq s$, the representations $u_1, \ldots u_s, \Pi_{n-s}$ satisfy the assumptions of (\ref{irr}) and are in the small type case, which we will prove in the following sections. 

\begin{thm} \label{still-irr}
    Let the essentially Speh representation $u_{\rho}(a,b)|\cdot|^{s}$ of small type, the Speh representations $u_{\rho_i}(a_i,b_i)$ for $i=1,\ldots,n$ and the Arthur type representation $\pi$ satisfy \eqref{irr}. Then the induced representation $\displaystyle u_{\rho}(a,b)|\cdot|^{s}\times\left(\bigtimes_{i=1}^n u_{\rho_i}(a_i,b_i)\right)\rtimes\pi$ is irreducible.
\end{thm}
\begin{proof}
    To show that the induced representation $$\displaystyle u_{\rho}(a,b)|\cdot|^{s}\times\left(\bigtimes_{i=1}^n u_{\rho_i}(a_i,b_i)\right)\rtimes\pi$$
    is irreducible, by Lemma \ref{lemica} it suffices to show that its socle is irreducible. 
    To describe its socle, we use Proposition \ref{char}.
    Let $\pi_{\mu'}$ be an irreducible summand of $ u_{\rho}(a-2s,b)\rtimes\pi $. We claim that 
    $$ \left(\bigtimes_{i=1}^n u_{\rho_i}(a_i,b_i)\right)\rtimes\pi_{\mu'} $$
    is irreducible. Due Corollary \ref{bad-parity-unitary}, it suffices to show that 
    \begin{align*}
        u_{\rho_i}(a_i,b_i) \rtimes \pi_{\mu'}
    \end{align*}
    is irreducible for each $i$. We may decompose $\pi \cong \tau_{\psi_1} \times \pi_0$ (as in Proposition \ref{bad-parityy}). A summand $\pi_{\mu'}$ of $u_\rho(a-2s,b) \rtimes \pi$ is of the form $\tau_{\psi_1} \rtimes (\pi_0)_{\mu'}$. Now
    \begin{align*}
        u_{\rho_i}(a_i,b_i) \rtimes \pi_{\mu'} &\cong  u_{\rho_i}(a_i,b_i) \times  \tau_{\psi_1} \rtimes (\pi_0)_{\mu'} \\
        &\cong  \tau_{\psi_1} \times u_{\rho_i}(a_i,b_i)\rtimes (\pi_0)_{\mu'} \\
        &\cong  \tau_{\psi_1} \times \left(  \bigoplus_{j} (\pi_0)_{\mu',\mu_{i,j}} \right).
    \end{align*}
    But the fact that $u_{\rho_i}(a_i,b_i) \rtimes \pi \cong u_{\rho_i}(a_i,b_i) \times \tau_{\psi_1} \rtimes \pi_0 \cong  \tau_{\psi_1} \times u_{\rho_i}(a_i,b_i) \rtimes \pi_0$ is irreducible implies that $u_{\rho_i}(a_i,b_i) \rtimes \pi_0 \cong (\pi_0)_{\mu_{i}}$ irreducible. Hence the parameter $\mu_{i,j}=\mu_i$ is unique and $u_{\rho_i}(a_i,b_i) \rtimes \pi_{\mu'} $ irreducible. 
    
    As in Proposition \ref{char}, we have the following characterization: 
    $$
    \mathrm{soc}(u_{\rho}(2s,b)|\cdot|^{a/2}\rtimes\left(\bigtimes_{i=1}^n u_{\rho_i}(a_i,b_i)\rtimes\pi_{\mu'}\right))
    $$
    is a subrepresentation of 
    $$\displaystyle u_{\rho}(a,b)|\cdot|^{s}\times\left(\bigtimes_{i=1}^n u_{\rho_i}(a_i,b_i)\right)\rtimes\pi
    $$ if and only if for the non-negative integers $l_{i_0,j_0}$ such that 
      \begin{gather} \label{v3}
      D^{\max}_{(a,b,s)}\left(\bigtimes_{i=1}^n u_{\rho_i}(a_i,b_i)\rtimes\pi\right)= 
       D^{(l_{i_0,j_0})}_{(a,b,s)}\left(\bigtimes_{i=1}^n u_{\rho_i}(a_i,b_i)\rtimes\pi\right)
    \end{gather}
    we have $\displaystyle D^{(l_{i_0,j_0})}_{(a,b,s)}\left(\bigtimes_{i=1}^n u_{\rho_i}(a_i,b_i)\rtimes\pi_{\mu'}\right)\neq 0.$
    We want to show that this is the case only for $\mu'=\mu$, where $\mu$ is the unique parameter from Proposition \ref{3.6} such that
    $$ u_{\rho}(a,b)|\cdot|^{s}\rtimes\pi\cong \mathrm{soc}(u_{\rho}(2s,b)|\cdot|^{a/2}\rtimes\pi_{\mu}). $$
    
    Now we get the irreducibility of the socle of $\displaystyle u_{\rho}(a,b)|\cdot|^s\times\left(\bigtimes_{i=1}^n u_{\rho_i}(a_i,b_i)\right)\rtimes\pi$ in the following way.
    For $\mu'\neq\mu$, we take the indices $i_0,j_0$ to be the first such that 
    \begin{align*}
        (D^{\max}_{\rho|\cdot|^{\alpha_{i_0,j_0}}}\circ\cdots\circ D^{\max}_{\rho|\cdot|^{B+s}})(\pi_{\mu'})=0,
    \end{align*}
    for the corresponding exponent $\alpha_{i_0,j_0}$.
    Let us denote 
    $
    k_{i_0,j_0}=d^{i_0 j_0}(\pi)
    $
    and note that it is strictly greater than $d^{i_0 j_0}(\pi_{\mu'})$.
    For the number $l_{i_0,j_0}$ defined in (\ref{v3}), we conclude
    $$ l_{i_0,j_0}=k_{i_0,j_0}+\text{mult}\left(M_{\rho|\cdot|^{\alpha_{i_0,j_0}}},  M^{\max}_{(a,b,s),<i_0,j_0}\left(\bigtimes_{i=1}^n u_{\rho_i}(a_i,b_i)\right)\right). $$
    Then we have
    \begin{gather*}
        d^{i_0 j_0}\left(\bigtimes_{i=1}^n u_{\rho_i}(a_i,b_i)\rtimes \pi_{\mu'}\right)= \\
        m^{i_0 j_0}\left(\bigtimes_{i=1}^n u_{\rho_i}(a_i,b_i)\right)+ d^{i_0 j_0}\left(\pi_{\mu'}\right)< \\
        m^{i_0 j_0}\left(\bigtimes_{i=1}^n u_{\rho_i}(a_i,b_i)\right) + k_{i_0,j_0}=l_{i_0,j_0}.
    \end{gather*}
   Hence, $\displaystyle D^{(l_{i_0,j_0})}_{(a,b,s)}\left(\bigtimes_{i=1}^n u_{\rho_i}(a_i,b_i)\rtimes\pi_{\mu'}\right)= 0.$
   According to the characterization, we get that $\displaystyle\mathrm{soc}(u_{\rho}(2s,b)|\cdot|^{\frac{a}{2}}\rtimes\left(\bigtimes_{i=1}^n u_{\rho_i}(a_i,b_i)\rtimes\pi_{\mu'}\right))$ is not a subrepresentation of the induced representation $\displaystyle u_{\rho}(a,b)|\cdot|^s\times\left(\bigtimes_{i=1}^n u_{\rho_i}(a_i,b_i)\right)\rtimes\pi$.
   This proves the claim.

\end{proof}

To deal with several essentially Speh representations of small type, we will need the following:
\begin{de} \label{pi-ein-mu}
Let $\pi$ be of Arthur type and let $u_{\rho_i}(a_i,b_i)|\cdot|^{s_i}$ be an essentially Speh representation of small type. We denote by
\begin{align*}
    u_{\rho_i}(a_i-2s_i,b_i)\rtimes\pi = \bigoplus_{j} \pi_{\mu_{i,j}}
\end{align*} 
the decomposition of $ u_{\rho_i}(a_i-2s_i,b_i)\rtimes\pi$ into irreducible subrepresentations of Arthur type, where $\pi_{\mu_{i,j}}$ is defined as follows: We may decompose $\pi \cong \tau_{\psi_1} \times \pi_0$ (as in Proposition \ref{bad-parityy}). If $u_{\rho_i}(a_i-2s_i,b_i)$ is not of good parity with respect to $\pi_0$, we have that $ u_{\rho_i}(a_i-2s_i,b_i)\rtimes\pi$ is irreducible and we set $\pi_{\mu_{i,j}} \ceq u_{\rho_i}(a_i-2s_i,b_i)\rtimes\pi$ (the index $j$ is unique). Otherwise, we have $\pi_0=\pi(\mathcal{S})$ and we set $\pi_{\mu_{i,j}} \ceq \tau_{\psi_1} \rtimes \pi(\mathcal{S}_{([A_i-s_i,B_i-s_i]_{\rho_i},\mu_{i,j})})$. 
\end{de}

\begin{de} \label{pi_mu}
    Let $\rho_1, \ldots \rho_n \in \mathrm{Cusp}_{unit}(GL)$, let $\pi$ be of Arthur type and let $u_{\rho_i}(a_i,b_i)|\cdot|^{s_i}$ be essentially Speh representations of small type. We define the irreducible subrepresentation $\pi_{\mu_{1,j_i},\ldots,\mu_{n,j_n}}$ of $\bigtimes_{i=1}^n u_{\rho_i}(a_i-2s_i,b_i)\rtimes\pi$ 
    by
    $$ \pi_{\mu_{1,j_1},\ldots,\mu_{n,j_n}}\hookrightarrow \left(\bigtimes_{\substack{i=1 \\ i\neq r}}^{n} u_{\rho_i}(a_i-2s_i,b_i)\right)\rtimes \pi_{\mu_{r,j_r}}, $$
    for every $r=1,\ldots,n$, where $\pi_{\mu_{r,j_r}}$ is defined in Definition \ref{pi-ein-mu}. 
\end{de}
\begin{prop}Let $\rho_1, \ldots \rho_n \in \mathrm{Cusp}_{unit}(GL)$, let $\pi$ be of Arthur type and let $u_{\rho_i}(a_i,b_i)|\cdot|^{s_i}$ for $i=1,\ldots, n$ be essentially Speh representations of small type. Then the representation  $\pi_{\mu_{1,j_1},\ldots,\mu_{n,j_n}}$ exists and is unique (it also has multiplicity one in $\bigtimes_{i=1}^n u_{\rho_i}(a_i-2s_i,b_i)\rtimes\pi$). Every constituent of $\bigtimes_{i=1}^n u_{\rho_i}(a_i-2s_i,b_i)\rtimes\pi$ is of this form.
\end{prop}
\begin{proof}
If all $u_{\rho_i}(a_i-2s_i,b_i)$ are of good parity with respect to $\pi$, then $\pi_{\mu_{1,j_1},\ldots,\mu_{n,j_n}}$ corresponds to the element $(\mu_{1,j_1},...,\mu_{n,j_n})\in N(\mathcal{S},[A_1-s_1,B_1-s_1]_{\rho_1},...,[A_n-s_n,B_n-s_n]_{\rho_n})$ in the decomposition of Corollary \ref{Nindex}. 
Otherwise, we may decompose $\pi \cong \tau_{\psi_1} \times \pi_0$ (as in Proposition \ref{bad-parityy}) and w.l.o.g. we may assume that $u_{\rho_i}(a_i-2s_i,b_i)$ is not of good parity with respect to $\pi_0$ for $i=1,\ldots,k$ and of good parity for $i=k+1,\ldots, n$. Note that 
\begin{align*}
    u_{\rho_i}(a_i-2s_i,b_i) \rtimes \pi &\cong u_{\rho_i}(a_i-2s_i,b_i) \times \tau_{\psi_1} \rtimes \pi_0 \\
    & \cong \tau_{\psi_1} \times u_{\rho_i}(a_i-2s_i,b_i)  \rtimes \pi_0 \\
    &\cong \bigoplus_{j} \left( \tau_{\psi_1}\rtimes (\pi_0)_{\mu_{i,j}'}\right) \\
    &\cong \bigoplus_{j} \pi_{\mu_{i,j}'},
\end{align*}
for $i=k+1,\ldots,n$. Also, $u_{\rho_i}(a_i-2s_i,b_i) \rtimes \pi $ is irreducible for $i=1,\ldots,k$. We define 
\begin{align*}
    \pi_{\mu_{{k+1},j_{k+1}},\ldots,\mu_{n,j_n}} \ceq \tau_{\psi_1} \rtimes (\pi_0)_{\mu_{{k+1},j_{k+1}},\ldots,\mu_{n,j_n}}. 
\end{align*}
Then
\begin{align*}
  \pi_{\mu_{1,j_1},\ldots, \mu_{n,j_n}} \ceq  \left(\bigtimes_{i=1}^{k} u_{\rho_i}(a_i-2s_i,b_i)\right)  \times \pi_{\mu_{{k+1},j_{k+1}},\ldots,\mu_{n,j_n}}
\end{align*}
is irreducible (Corollary \ref{bad-parity-is-easy} and Corollary \ref{bad-parity-unitary}) and satisfies the inclusion condition of the Definition. To show that it is unique, we note that every constituent of $\bigtimes_{i=1}^n u_{\rho_i}(a_i-2s_i,b_i)\rtimes\pi$ is of the form 
\begin{align*}
    \left(\bigtimes_{i=1}^{k} u_{\rho_i}(a_i-2s_i,b_i)\right)  \times \tau_{\psi_1} \rtimes (\pi_0)_{\mu_{k+1,j_{k+1}}',\ldots,\mu_{n,j_n}'}. 
\end{align*}
If this representation also satisfies the inclusion relations, then  the fact that it is contained in the corresponding Arthur packet, which is multiplicity-free (due to Mœglin \cite{Mœglin2}), shows that $\mu_{k+1,j_{k+1}}'=\mu_{k+1,j_{k+1}},\ldots ,\mu_{n,j_n}'=\mu_{n,j_n}$.
\end{proof}
We will now proceed to prove Theorem \ref{thm:IRR} for the case of small type representations.
From the assumptions in (\ref{irr}) it follows that we can order the essentially Speh representations such that real numbers $B_i+s_i$ form a non-decreasing sequence, i.e. $B_i+s_i\le B_{i+1}+s_{i+1}$, for $i=1,\ldots,n-1$. The main ingredient of the proof is the following Lemma, which allows us to identify if a subquotient is a subrepresentation in the given case:
\begin{lemma} \label{pom11}
    Let $\delta \in \mathrm{Irr}(G)$ and let $u_{\rho_k}(a_k,b_k)|\cdot|^{s_k}$ be an essentially Speh representation of small type. We consider the essentially Speh representation $u_{\rho_k}(2s_k,b_k)|\cdot|^{a_k/2}$ of big type. There is a unique irreducible subquotient $\sigma$ of the induced representation $u_{\rho_k}(2s_k,b_k)|\cdot|^{a_k/2} \rtimes \delta$ such that $d_k^{ij}(\sigma)=1_{ij}+d_k^{ij}(\delta)$, for each pair of indices $i,j$.
    It is isomorphic to the unique irreducible subrepresentation of the induced representation $u_{\rho_k}(2s_k,b_k)|\cdot|^{a_k/2} \rtimes \delta$.
\end{lemma}
\begin{proof}
    Assume that $\sigma$ is an irreducible subquotient of $u_{\rho_k}(2s_k,b_k)|\cdot|^{a_k/2} \rtimes \delta$ such that $d_k^{ij}(\sigma)=1_{ij}+d_k^{ij}(\delta)$, for each pair of indices $i,j$.
    By definition of the derivatives, we have that the representation
    \begin{gather} \label{jacconstsigma}
         \tau_\sigma \ceq (\rho_k|\cdot|^{B_k+s_k})^{(k_{1,1}+1)}\otimes\cdots\otimes (\rho_k|\cdot|^{A_k-s_k+1})^{(k_{2s_k,b_k}+1)}\otimes D^{\max}_{(a_k,b_k,s_k)}(\sigma)
    \end{gather} 
    is a constituent of the Jacquet module (with respect to the appropriate parabolic subgroup $P$) of the representation $\sigma$, where we denote $k_{i,j}=d_k^{ij}(\delta)$. We claim that $D^{\max}_{(a_k,b_k,s_k)}(\sigma)=D^{\max}_{(a_k,b_k,s_k)}(\delta)$.
    To see this, note that for
    \begin{gather} \label{jacconstpi}
         \tau_\delta\ceq (\rho_k|\cdot|^{B_k+s_k})^{(k_{1,1}+1)}\otimes\cdots\otimes (\rho_k|\cdot|^{A_k-s_k+1})^{(k_{2s_k,b_k}+1)}\otimes D^{\max}_{(a_k,b_k,s_k)}(\delta)
    \end{gather} 
    we have
     $\tau_\delta\leq  \mathrm{Jac}_P(u_{\rho_k}(2s_k,b_k)|\cdot|^{a_k/2} \rtimes \delta) $
    since
    \begin{align*}
       D^{\max}_{(a_k,b_k,s_k)} (u_{\rho_k}(2s_k,b_k)|\cdot|^{a_k/2} \rtimes \delta) = D^{\max}_{(a_k,b_k,s_k)}(\delta)
    \end{align*}
     due to Remark \ref{tad} and Lemma \ref{glirred} (1).
    Since we showed $\tau_\sigma \leq  \mathrm{Jac}_P(\sigma),$
    the exactness of the Jacquet functor implies 
    \begin{align*}
        \tau_\sigma \leq  \mathrm{Jac}_P(\sigma)\leq \mathrm{Jac}_P(u_{\rho_k}(2s_k,b_k)|\cdot|^{a_k/2} \rtimes \delta) .
    \end{align*}
    Note that since we have taken the maximal derivatives on the right side, we must have that $\tau_\sigma=\tau_\delta$, which implies 
    $D^{\max}_{(a_k,b_k,s_k)}(\sigma)=D^{\max}_{(a_k,b_k,s_k)}(\delta)$.
    
    On the other hand, the representation $\tau_\delta$ from (\ref{jacconstpi}) is a constituent of the Jacquet module of a representation $\sigma_0:=\mathrm{soc}(u_{\rho_k}(2s_k,b_k)|\cdot|^{a_k/2} \rtimes \delta)$, which is irreducible, of multiplicity one, by Proposition \ref{3.4}.
    This follows directly from Lemma \ref{glirred} (1) and the Frobenius reciprocity applied on the embedding:
    \begin{gather*}
        \sigma_0\hookrightarrow u_{\rho_k}(2s_k,b_k)|\cdot|^{a_k/2} \rtimes \delta\hookrightarrow \\
        (\rho_k|\cdot|^{B_k+s_k})^{(k_{1,1}+1)}\times\cdots\times (\rho_k|\cdot|^{A_k-s_k+1})^{(k_{2s_k,b_k}+1)}\rtimes D^{\max}_{(a_k,b_k,s_k)}(\delta).
    \end{gather*}
    If $\sigma\not\cong\sigma_0$, then the representation $\tau_\sigma$ appears in the Jacquet module of the induced representation $u_{\rho_k}(2s_k,b_k)|\cdot|^{a_k/2} \rtimes \delta$ with multiplicity greater than or equal to two. 
    This is not possible since each derivative from the composition of derivatives $D^{\max}_{(a_k,b_k,s_k)}$ is taken the maximal amount of times and $D^{\max}_{(a_k,b_k,s_k)}(\delta)$ is irreducible.
\end{proof}

\begin{lemma} \label{nested-soc}

Let $\pi'$ denote an irreducible subquotient of the induced representation 
$$
\displaystyle\bigtimes_{k=1}^n u_{\rho_k}(a_k,b_k)|\cdot|^{s_k} \rtimes\pi,
$$
where the inducing factors satisfy the conditions of (\ref{irr}), the essentially Speh representations are of small type and $B_k+s_k \leq B_{k+1}+s_{k+1}$ for $k=1,...,n-1$. If $\pi'$ satisfies 
        \begin{gather*} 
        d^{ij}_{k}(\pi')= 1_{ij}+\sum_{\substack{l=1 \\ l \neq k}}^{n} m^{ij}_{k}(u_{\rho_l}(a_l,b_l)|\cdot|^{s_l})+d^{ij}_{k}(\pi),
\end{gather*}
for all indices $i,j$ in the block of derivatives $D_{(a_k,b_k,s_k)}^{\max}$ and all $k=1,\ldots n$, then the representation $\pi'$ 
    is isomorphic to 
    \begin{gather*}
         \mathrm{soc}( u_{\rho_1}(2s_1,b_1)|\cdot|^{a_1/2} \rtimes \mathrm{soc}(u_{\rho_{2}}(2s_{2},b_{2})|\cdot|^{a_{2}/2}\rtimes\mathrm{soc}(\ldots \\
         \rtimes \mathrm{soc}(u_{\rho_{n}}(2s_{n},b_{n})|\cdot|^{a_{n}/2}\rtimes\pi_{\mu_1,\ldots,\mu_n}))\ldots),
    \end{gather*}
    where $\pi_{\mu_1,\ldots,\mu_n}$ is the representation defined in Definition \ref{pi_mu} and $\mu_i$ are the parameters defined by Proposition \ref{3.6}.
If 
    \begin{align*} 
    \pi' & \leq  u_{\rho_1}(2s_1,b_1)|\cdot|^{a_1/2} \times \ldots \times u_{\rho_k}(2s_k,b_k)|\cdot|^{a_k/2} \\ \nonumber
    &\times u_{\rho_{k+1}}(a_{k+1},b_{k+1})|\cdot|^{s_{k+1}} \ldots  u_{\rho_{n}}(a_{n},b_{n})|\cdot|^{s_{n}} \rtimes \sigma
    \end{align*}
    for some $\sigma \leq u_{\rho_k}(a_k-2s_k,b_k) \rtimes \pi_{\mu_1,\ldots ,\mu_{k-1}}$, then 
    \begin{align} \label{pi'1}
        \sigma = \pi_{\mu_1,\ldots ,\mu_k}.
    \end{align}
Moreover, if 
    \begin{align*}
    \pi' & \leq  u_{\rho_1}(2s_1,b_1)|\cdot|^{a_1/2} \times \ldots \times u_{\rho_k}(2s_k,b_k)|\cdot|^{a_k/2} \rtimes \sigma
    \end{align*}
for some
\begin{align*} 
    \sigma \leq  u_{\rho_{k+1}}(2s_{k+1},b_{k+1})|\cdot|^{a_{k+1}/2}\rtimes\mathrm{soc}(\ldots
         \rtimes \mathrm{soc}(u_{\rho_{n}}(2s_{n},b_{n})|\cdot|^{a_{n}/2}\rtimes\pi_{\mu_{k+1},\ldots,\mu_n}))\ldots),
\end{align*}
then 
\begin{align}\label{pi'2}
    \sigma =\mathrm{soc}( u_{\rho_{k+1}}(2s_{k+1},b_{k+1})|\cdot|^{a_{k+1}/2}\rtimes\mathrm{soc}(\ldots
         \rtimes \mathrm{soc}(u_{\rho_{n}}(2s_{n},b_{n})|\cdot|^{a_{n}/2}\rtimes\pi_{\mu_{k+1},\ldots,\mu_n}))\ldots).
\end{align}
\end{lemma}
\begin{proof}
    For this proof, to shorten notation, we write 
\begin{align*}
    u_i &\ceq u_{\rho_i}(a_i,b_i)|\cdot|^{s_i}, \\
    u_i' &\ceq u_{\rho_i}(a_i-2s_i,b_i), \\
    u_i'' &\ceq u_{\rho_i}(2s_i,b_i)|\cdot|^{a_i/2}, 
\end{align*}
for $i=1,...,n$. 
Now
    \begin{align} \label{embed}
    \pi' &\leq \displaystyle\bigtimes_{k=1}^n u_k \rtimes\pi \nonumber \\
    &\cong u_1 \times \displaystyle\bigtimes_{k=2}^{n} u_k \rtimes\pi \nonumber \\
     \Rightarrow  \pi' &\leq  u_1'' \times \left( \displaystyle\bigtimes_{k=2}^{n} u_k \right) \times u_1' \rtimes \pi, 
\end{align}
where we have used that changing the order of factors in the parabolic induction gives the same semisimplification.
There is a parameter $\mu_1'$ (remember that $\pi$ is of Arthur type) such that
\begin{align} \label{p1}
    \pi' \leq u_1'' \times \left( \displaystyle\bigtimes_{k=2}^{n} u_k \right) \rtimes \pi_{\mu_1'}. 
\end{align}
Now for all indices $i,j$ in the block of derivatives $D_{(a_1,b_1,s_1)}^{\max}$, we have successively (by the assumption on $\pi'$ and the exactness of the Jacquet functor applied on (\ref{p1}))
\begin{gather*} \label{zgt}
        d^{ij}_1(\pi')= 1_{ij}+\sum_{k=2}^{n} m^{ij}_1(u_k)+d^{ij}_1(\pi) \\
        d^{ij}_1(\pi') \leq  1_{ij}+\sum_{k=2}^{n} m^{ij}_1(u_k)+d^{ij}_1(\pi_{\mu_1'}), 
\end{gather*}
which implies that $d^{ij}_1(\pi_{\mu_1'}) = d^{ij}_1(\pi)$. Since we get this for all indices $i,j$, we have $\mu_1'= \mu_1$ (see Remark \ref{eqmult}).
Now we have that there is a subquotient $\sigma_1$ of 
\begin{align*}
 \left( \displaystyle\bigtimes_{k=2}^{n} u_k \right) \rtimes \pi_{\mu_1}. 
\end{align*}
such that $\pi' \leq u_1'' \rtimes \sigma_1$. Set $\sigma_0 \ceq \pi'$. Recursively, we will choose irreducible representations $\sigma_k$ with the following properties:

\begin{enumerate}
    \item \begin{align*}
    \sigma_k \leq  \left( \displaystyle\bigtimes_{l=k+1}^{n} u_l \right) \rtimes \pi_{\mu_1,...,\mu_k},
\end{align*}
    \item $$ \sigma_{k} \leq u_{k+1}'' \rtimes \sigma_{k+1}, $$
\end{enumerate}
for $k=0,1,\ldots,n$.
We just showed that $\sigma_1$ satisfies these properties. 
Assume now that they are true for $\sigma_k$. We will show that $\sigma_{k+1}$ exists (and later that it is in fact uniquely determined): 
\begin{align} \label{sigmak}
    \sigma_k &\leq  u_{k+1}'' \times u_{k+1}'  \times \left( \displaystyle\bigtimes_{l=k+2}^{n} u_l \right) \rtimes \pi_{\mu_1,...,\mu_k} \nonumber  \\
   \Rightarrow  \sigma_k &\leq  u_{k+1}'' \times \left( \displaystyle\bigtimes_{l=k+2}^{n} u_l \right) \times u_{k+1}' \rtimes \pi_{\mu_1,...,\mu_k} \nonumber  \\
    \Rightarrow \sigma_k &\leq  u_{k+1}'' \times \left( \displaystyle\bigtimes_{l=k+2}^{n} u_l \right) \rtimes \pi_{\mu_1,...,\mu_k,\mu_{k+1}'}.
\end{align}
One can in fact show that $\mu_{k+1}'=\mu_{k+1}$. Namely, for the indices $i,j$ in the block of derivatives $D_{(a_{k+1},b_{k+1},s_{k+1})}^{\max}$, we have successively
\begin{gather*} 
        d^{ij}_{k+1}(\pi')= 1_{ij}+\sum_{\substack{l=1 \\ l \neq k+1}}^{n} m^{ij}_{k+1}(u_l)+d^{ij}_{k+1}(\pi), 
\end{gather*}
by the assumption. 
Using properties 1. and 2., we obtain
$$ \pi' \le u_1''\times\ldots\times u_{k+1}''\times u_{k+2}\times\ldots\times u_n \rtimes \pi_{\mu_1,...,\mu_k,\mu_{k+1}'}. $$
Now from the exactness of the Jacquet functor we get:
\begin{gather*}
        d^{ij}_{k+1}(\pi') \leq  1_{ij}+\sum_{l=1}^{k} m^{ij}_{k+1}(u_l'')+\sum_{l=k+2}^{n} m^{ij}_{k+1}(u_l)+d^{ij}_{k+1}(\pi_{\mu_1,...,\mu_k,\mu_{k+1}'}).
\end{gather*}
Note that we have $B_{k+1}+s_{k+1} \geq B_l +s_l$ for $l=1,\ldots,k$. Hence, by Lemma \ref{jmlema} we get
\begin{gather}
        d^{ij}_{k+1}(\pi_{\mu_1,...,\mu_k,\mu_{k+1}'}) \geq d^{ij}_{k+1}(\pi)+  \sum_{l=1}^{k} m^{ij}_{k+1}(u_l')
\end{gather}
and since 
\begin{align*}
    \pi_{\mu_1,...,\mu_k,\mu_{k+1}'} \hookrightarrow u_1' \times ... \times u_k' \rtimes \pi_{\mu_{k+1}'} \hookrightarrow u_1' \times ... \times u_k' \times u_{k+1}' \rtimes \pi,
\end{align*}
we have 
\begin{align*}
    d^{ij}_{k+1}(\pi_{\mu_1,...,\mu_k,\mu_{k+1}'})  &\leq  \sum_{l=1}^{k} m^{ij}_{k+1}(u_l')+  d^{ij}_{k+1}(\pi_{\mu_{k+1}'}) \\
    &\leq   \sum_{l=1}^{k} m^{ij}_{k+1}(u_l')+  d^{ij}_{k+1}(\pi) \leq d^{ij}_{k+1}(\pi_{\mu_1,...,\mu_k,\mu_{k+1}'}).  
\end{align*}
Therefore we have equality everywhere and this means that $d^{ij}_{k+1}(\pi_{\mu_{k+1}'})= d^{ij}_{k+1}(\pi)$ (and hence for all indices $i,j$ step-by-step), which shows that $\mu_{k+1}'=\mu_{k+1}$ by Remark \ref{eqmult}. Since the parabolic induction $u_{k+1}'\rtimes \pi$ is multiplicity free, we have proved the second assertion (\ref{pi'1}) of the Lemma.

From (\ref{sigmak}), we can choose any (due to transitivity of parabolic induction) $\sigma_{k+1}$ with
\begin{align*}
    \sigma_{k+1} \leq  \left( \displaystyle\bigtimes_{l=k+2}^{n} u_l \right)  \rtimes \pi_{\mu_1,...,\mu_{k+1}}
\end{align*}
such that 
\begin{align*}
    \sigma_{k} \leq u_{k+1}'' \rtimes \sigma_{k+1}.
\end{align*}
We see that 
\begin{align} \label{p11}
\pi' & \leq u_1'' \times ... \times u_{k+1}'' \rtimes \sigma_{k+1}.
\end{align}
In the end, we will have chosen $\sigma_{n-1} \leq u_n'' \rtimes \pi_{\mu_1,...,\mu_n}$. Set $\sigma_n \ceq \pi_{\mu_1,...,\mu_n}$.
We claim that in fact $\sigma_{k-1} = \mathrm{soc}( u_{k}'' \rtimes \sigma_{k})$, and we will prove it using Lemma \ref{pom11}. As above, we have
\begin{align} \label{p111}
    d^{ij}_{k}(\pi') = 1_{ij}+  \sum_{\substack{l=1 \\ l \neq k}}^n m^{ij}_{k}(u_l) +  d^{ij}_{k}(\pi),
\end{align}
for all indices $i,j$ in the block of derivatives $D_{(a_{k},b_{k},s_{k})}^{\max}$.
From the analogues of a relation (\ref{p11}) for indices $k-1$ and $k$, we get successively 
\begin{align*}
    d^{ij}_{k}(\pi') \leq  \sum_{l=1}^{k-1} m^{ij}_{k}(u_l'') +  d^{ij}_{k}(\sigma_{k-1})
\end{align*}
and
\begin{align*}
    d^{ij}_{k}(\pi') \leq 1_{ij}+ \sum_{l=1}^{k-1} m^{ij}_{k}(u_l'') +  d^{ij}_{k}(\sigma_{k}).
\end{align*}
Together with (\ref{p111}) and Lemma \ref{jmlema}, this implies
\begin{align} \label{p2}
    d^{ij}_{k}(\sigma_{k-1}) \geq 1_{ij} + \sum_{l=1}^{k-1} m^{ij}_{k}(u_l') +  \sum_{l=k+1}^{n} m^{ij}_{k}(u_l) + d^{ij}_{k}(\pi)
\end{align}
and
\begin{align} \label{p3}
    d^{ij}_{k}(\sigma_{k}) \geq  \sum_{l=1}^{k-1} m^{ij}_{k}(u_l')+ \sum_{l=k+1}^{n} m^{ij}_{k}(u_l) + d^{ij}_{k}(\pi).
\end{align}
On the other hand, from property 1. we get
\begin{align*}
    \sigma_{k-1} \leq \left(  \bigtimes_{l=k}^n u_l \right) \times \left(  \bigtimes_{l=1}^{k-1} u_l' \right) \rtimes \pi 
\end{align*}
and
\begin{align*}
    \sigma_{k} \leq \left(  \bigtimes_{l=k+1}^n u_l \right) \times \left(  \bigtimes_{l=1}^{k} u_l' \right) \rtimes \pi.
\end{align*}
Hence, from this we get exactly the opposite inequalities than the ones in (\ref{p2}) and (\ref{p3}).
Consequently, we have the following equalities
\begin{align*}
    d^{ij}_{k}(\sigma_{k-1}) =1_{ij} + \sum_{l=1}^{k-1} m^{ij}_{k}(u_l') +  \sum_{l=k+1}^{n} m^{ij}_{k}(u_l) + d^{ij}_{k}(\pi)
\end{align*}
and
\begin{align*}
    d^{ij}_{k}(\sigma_{k})= \sum_{l=1}^{k-1} m^{ij}_{k}(u_l')+ \sum_{l=k+1}^{n} m^{ij}_{k}(u_l) + d^{ij}_{k}(\pi).
\end{align*}
This implies $d^{ij}_{k}(\sigma_{k})+1_{ij} =d^{ij}_{k}(\sigma_{k-1})$ and hence by Lemma \ref{pom11} we get $\sigma_{k-1} \cong \mathrm{soc}( u_{k}'' \rtimes \sigma_{k})$.
Note that this implies that $\sigma_{k-1}$ is unique because $u_{k}'' \rtimes \sigma_{k}$ is socle irreducible due to Proposition \ref{3.4}. 
This also implies that $\sigma_{k-1}$ appears with multiplicity one in $[u_{k}'' \rtimes \sigma_{k}]$.
But this shows that all the $\sigma_{k-1}$ are uniquely determined and hence $\pi'$ is isomorphic to 
    \begin{gather*}
         \mathrm{soc}( u_1'' \rtimes \mathrm{soc}(u_2''\rtimes\mathrm{soc}(\ldots
         \rtimes \mathrm{soc}(u_n''\rtimes\pi_{\mu_1,\ldots,\mu_n}))\ldots).
    \end{gather*}

To prove the last assertion (\ref{pi'2}), let 
\begin{align*}
    \pi' \leq u_1'' \times \ldots u_k'' \rtimes \sigma
\end{align*}
for some $\sigma \leq u_{k+1}'' \rtimes \sigma_{k+1}$. We want to show that $\sigma\cong\sigma_k$. But now 
\begin{align*}
    d^{ij}_{k+1}(\pi') &\leq  \sum_{l=1}^{k} m^{ij}_{k+1}(u_l'') +  d^{ij}_{k+1}(\sigma) \\
    &\leq \sum_{l=1}^{k} m^{ij}_{k+1}(u_l'') +  1_{ij}+ d^{ij}_{k+1}(\sigma_{k+1}) \\
    & = \sum_{l=1}^{k} m^{ij}_{k+1}(u_l'')+1_{ij} + \sum_{l=1}^{k} m^{ij}_{k+1}(u_l')+ \sum_{l=k+2}^{n} m^{ij}_{k+1}(u_l) + d^{ij}_{k+1}(\pi) \\
    &=  \sum_{l=1}^{k} m^{ij}_{k+1}(u_l)+1_{ij} +   \sum_{l=k+2}^{n} m^{ij}_{k+1}(u_l) + d^{ij}_{k+1}(\pi) \\
    &= d^{ij}_{k+1}(\pi') 
\end{align*}
for all indices $i,j$ in the block of derivatives $D_{(a_{k+1},b_{k+1},s_{k+1})}^{\max}$. But this implies $\sigma\cong\mathrm{soc}(u_{k+1}'' \rtimes \sigma_{k+1})\cong\sigma_k$ by Lemma \ref{pom11}.
\end{proof}

\begin{prop} \label{subrep}
    Consider 
    $$\Pi \ceq \displaystyle\bigtimes_{k=1}^n u_{\rho_k}(a_k,b_k)|\cdot|^{s_k} \rtimes\pi,
    $$
    for essentially Speh representations of small type $u_{\rho_k}(a_k,b_k)|\cdot|^{s_k} $ and an Arthur type representation $\pi$, 
    where the inducing factors satisfy the conditions (\ref{irr}). Then the representation $\Pi$ is irreducible.
\end{prop}
\begin{proof}
We can assume $B_k+s_k \leq B_{k+1}+s_{k+1}$ for $k=1,...,n-1$, because the essentially Speh representations isomorphically commute.
We prove the claim by induction over the number $n$ of essentially Speh representations. For $n=1$, this is clear. Now assume the claim holds for $n-1$. Especially, this implies that the representation 
$$
\displaystyle\bigtimes_{\substack{l=1 \\
l \neq k}}^{n} u_{\rho_l}(a_l,b_l)|\cdot|^{s_l} \rtimes\pi
$$
is irreducible for all $l$. Let $\pi'$ be an irreducible subrepresentation of $\Pi$. 
Then Lemma \ref{glirred} implies
        \begin{gather*} 
        d^{ij}_{k}(\pi')= 1_{ij}+d^{ij}_{k}\left(\displaystyle\bigtimes_{\substack{l=1 \\
l \neq k}}^{n} u_{\rho_l}(a_l,b_l)|\cdot|^{s_l} \rtimes\pi
\right)
\end{gather*}
for all $k$ and all indices $i,j$. Note that from the multiplicativity of the map $M^*$ (Proposition \ref{tad}) we get
    \begin{gather}
        d_k^{ij}(\pi')= 1_{ij}+ \sum_{\substack{l=1 \\ l\neq k}}^n m^{ij}_k(u_{\rho_l}(a_l,b_l)|\cdot|^{s_l}) + d_l^{ij}(\pi).
    \end{gather}Now by Lemma \ref{nested-soc}, this implies that $\pi'$ is unique and hence $\mathrm{soc}(\Pi)\cong \pi'\oplus\overset{m \text{ times}}{\ldots}\oplus \pi'$ for some $m\ge0$. For the rest of the proof of Proposition, we prove $m=1$ and that $\pi'$ appears in $[\Pi]$ with multiplicity one.

If $\pi'$ appeared in $[\Pi]$ with multiplicity greater than one,
we would have (with the notation of the proof of Lemma \ref{nested-soc}):
\begin{align*}
   \pi' \oplus \pi' \leq  u_1'' \times u_2 \times \ldots \times u_n \times u_1' \rtimes \pi.
\end{align*}
Since $\pi_{\mu_1}$ has multiplicity one in $u_1' \rtimes \pi$ and $\pi'$ can only appear in 
\begin{align*}
    u_1'' \times u_2 \times \ldots \times u_n \rtimes \pi_{\mu_1}
\end{align*}
(by Lemma \ref{nested-soc}, (\ref{pi'1})), we also have 
\begin{align*}
   \pi' \oplus \pi' \leq   u_1'' \times u_2 \times \ldots \times u_n \rtimes \pi_{\mu_1}.
\end{align*}
Successively, this shows 
\begin{align*}
   \pi' \oplus \pi' \leq   u_1'' \times \ldots \times u_k'' \times u_{k+1} \times \ldots \times u_n \rtimes \pi_{\mu_1,\ldots,\mu_k}
\end{align*}
and thus
\begin{align*}
   \pi' \oplus \pi' \leq   u_1'' \times \ldots \times u_n'' \rtimes \pi_{\mu_1,\ldots,\mu_n}.
\end{align*}
Now conversely, since $\sigma_{n-1}$ appears in $ u_n'' \rtimes \pi_{\mu_1,\ldots,\mu_n}$ with multiplicity one (it is the socle) and since $\pi'$ only appears in 
\begin{align*}
u_1'' \times \ldots\times u_{n-1}'' \rtimes \sigma_{n-1}
\end{align*}
(by Lemma \ref{nested-soc}, (\ref{pi'1})), we have 
\begin{align*}
\pi' \oplus \pi' \leq   u_1'' \times \ldots\times u_{n-1}'' \rtimes \sigma_{n-1}.
\end{align*}
Successively we get to 
\begin{align*}
\pi' \oplus \pi' \leq   u_1'' \rtimes \sigma_{1},
\end{align*}
which is a contradiction.
This shows that $\Pi$ satisfies the conditions of Lemma \ref{lemica} and hence that it is irreducible.
\end{proof}

\subsection{The big type and non-half-integral type case}\label{bigcase}

We will now carry out the final step in the proof of Theorem \ref{thm:IRR}. Let $u_i= u_{\rho_i}(a_i,b_i)|\cdot|^{s_i}$ for $i=1, \ldots, n$ be essentially Speh representations such that $u_i$ is of big or non-half-integral type for $1\leq i\leq b$ and of small or unitary type else. Let $\pi \in \mathrm{Irr}(G)$, such that the assumptions (\ref{irr}) hold for $u_1,\ldots u_n$ and $\pi$. In the previous section, we showed that then 
\begin{align*}
   \Pi_{n-b} \ceq u_{b+1} \times \ldots u_n \rtimes \pi
 \end{align*}
 is irreducible. 
 We will now prove that 
 \begin{align*}
      \Pi_n &= u_1 \times \ldots \times u_n \rtimes \pi\\
      &= u_1 \times \ldots \times u_b \rtimes \Pi_{n-b}
 \end{align*}
 is irreducible by induction over the number $b$. If $b=0$, there is nothing to show. Now assume the statement holds true for $b-1$. Then we have that $\Pi_{n-1} \ceq u_2 \times \ldots \times u_b \rtimes \Pi_{n-b}$ is irreducible. By Proposition \ref{3.4}, we see that 
 $\mathrm{soc}(u_1 \rtimes \Pi_{n-1})$ is irreducible and appears in $[\Pi_n]$ with multiplicity one. Also, from the assumptions of (\ref{irr}), as in the proof of Lemma \ref{lemica}, it follows that 
 \begin{align*}
     & \mathrm{soc}(u_1 \rtimes \Pi_{n-1}) \\
          =\ &\mathrm{soc}(u_1 \times \ldots \times u_n \rtimes \pi) \\ 
          \cong\ &\mathrm{soc}(u_2 \times \ldots \times u_n \times u_1 \rtimes \pi) \\ 
          \cong\ &\mathrm{soc}(u_2 \times \ldots\times u_n \times u_1^\vee \rtimes \pi) \\ 
          =\ &\mathrm{soc}(u_1^\vee \times \ldots\times u_n \rtimes \pi) \\ 
     \cong\  & \mathrm{soc}(u_1^\vee \rtimes \Pi_{n-1}).
 \end{align*}
 But now, due to Proposition \ref{ired}, we see that $\Pi_n$ is irreducible.

\nopagebreak[4]

\appendix
\section{Appendix}

\subsection{Proofs about extended multi-segments}\label{app}

We will now give proofs for some of the claims made throughout this article:

\begin{lemma} \label{necess}
Let $\mathcal{S}= \bigcup_{\rho \in C_\mathcal{S}} \{(( [A_i^{\rho},B_i^{\rho}]_\rho,\mu_i^{\rho}))_{i=1}^{n_\rho}\}$ be a non-negative extended multi-segment. Atobe's necessary condition (\cite[Proposition 4.1]{A1}) for $\mathcal{F}(\mathcal{S})$ is equivalent to the condition (\ref{nec}):
\begin{align*} 
     |A_i^\rho-A_{i-1}^\rho| + |B_i^\rho-B_{i-1}^\rho| \geq |\mu_i^\rho-\mu_{i-1}^\rho | \qquad \forall\ 1 <i \leq n_\rho\ \forall\rho\in C_\mathcal{S}.
\end{align*}  
\end{lemma}
\begin{proof}
We consider any $\rho$-part of $\mathcal{S}$ (we drop the index $\rho$). 
Let $S_{i-1}=([A_{i-1},B_{i-1}]_\rho,\mu_{i-1})$ and $S_i=([A_i,B_i]_\rho,\mu_i)$ be two consecutive elements of the sequences $\mathcal{S}$.
Denote
\begin{align*}
    w_i &\ceq |\mu_i| \\
    \varepsilon_i &\ceq \mathrm{sgn}(\mu_i)\\
    \varepsilon &\ceq\varepsilon_{i-1}\varepsilon_i
\end{align*}
    \begin{itemize}
        \item If $ A_i\geq  A_{i-1}$ and $ B_i\geq  B_{i-1}$, we are in case $(1)$ of Atobe. We have 
        \begin{align*}
           &A_i-A_{i-1}+B_i-B_{i-1} = \\& =a_i-a_{i-1} \geq |w_i - w_{i-1}\varepsilon| = |b_i-2l_i - (b_{i-1}-2l_{i-1})\eta_i\eta_{i-1}(-1)^{b_{i-1}-1}| \\
           & \Leftrightarrow A_i - l_i \geq A_{i-1}-l_{i-1} \ \& \ B_i + l_i \geq B_{i-1}+l_{i-1} \\
           & or\ B_i + l_i \geq A_{i-1}+1 -l_{i-1} \ ( \Leftrightarrow l_i + l_{i-1} > A_{i-1} - B_i  ), 
        \end{align*}
        (depending on the sign $\eta_i\eta_{i-1}(-1)^{b_{i-1}-1}$) where the equivalence in the last bracket is due to the fact that $A_{i-1}-B_i$ is an integer.
        \item If $ A_i\geq  A_{i-1}$ and $ B_i \leq  B_{i-1} $, we are in case $(2)$ of Atobe. We have 
        \begin{align*}
        &A_i-A_{i-1}-B_i+B_{i-1} = \\
           & b_i-b_{i-1} \geq |w_i - w_{i-1}\varepsilon| = |b_i-2l_i - (b_{i-1}-2l_{i-1})\eta_i\eta_{i-1}(-1)^{b_{i-1}-1}| \\
           & \Leftrightarrow l_i-l_{i-1} \geq 0\ \& \ b_i-b_{i-1} \geq l_i- l_{i-1} \\
           & or \  l_{i-1} + l_i \geq b_{i-1}.
        \end{align*}
        \item If $ A_i \leq  A_{i-1}$ and $ B_i \geq  B_{i-1} $, we are in case $(3)$ of Atobe. We have 
        \begin{align*}
        &-A_i+A_{i-1}+B_i-B_{i-1} = \\
           & b_{i-1}-b_{i} \geq |w_i - w_{i-1}\varepsilon| = |b_i-2l_i - (b_{i-1}-2l_{i-1})\eta_i\eta_{i-1}(-1)^{b_{i-1}-1}| \\
           & \Leftrightarrow l_{i-1}-l_i \geq 0\ \& \ b_{i-1}-b_i \geq  l_{i-1}-l_i \\
           & or \  l_{i-1} + l_i \geq b_{i}.
        \end{align*}
    \end{itemize}
\end{proof}

We want to give a proof for Theorem \ref{thm:nonvanishingstandard}. Before we can do that, we need a short lemma:
\begin{lemma} \label{permutations}
    Let $\mathcal{S} \sim \mathcal{S}'$ be two formal extended multi-sets, such that $\mathcal{S} \in \mathrm{Rep}$ and $A_i^\rho=A_i'^\rho$ and $B_i^\rho=B_i'^\rho$ (the extended segments end up in the same order), then $\mu_i^\rho=\mu_i'^\rho$ for all $1\leq i\leq n_\rho$. That is, the $\mu_i^\rho$ are well-defined for a given reordering of the extended multi-segment. 
\end{lemma}

\begin{proof}
We consider any $\rho$-part of $\mathcal{S}$ (we drop the index $\rho$). Since the reorders $R_i$ perform transpositions on the extended multi-segments (changing the $\mu_i$ in the process), we need to analyze if the different ways of writing a permutation as a sequence of transpositions give the same $\mu_i$'s. It is known that the symmetric group $S_n$ is generated by the transpositions $\tau_1,...,\tau_{n-1}$ and the relations
\begin{enumerate}
    \item $\tau_i^2=1$ for all $i$
    \item $\tau_i\tau_{i+1} \tau_i =\tau_{i+1}\tau_i \tau_{i+1}$ for all $i<n-1$
    \item $\tau_i\tau_j=\tau_j\tau_i$ for $|i-j|\geq 2$.
\end{enumerate}
It is easily checked that the $\mu_i$'s stay invariant under any sequence of reorders $R_i$ corresponding to these transpositions.

\end{proof}

\nonvanishingstandard*
\begin{proof}
Since we want to use an inductive argument for this proof, we have to consider the slightly larger set of standard extended multi-segments, that do not satisfy the sign condition (\ref{sign}). For these generalized extended multi-segments, we  show that the inequalities (\ref{ine}) imply the necessary condition (\ref{nec}) for all reorderings. 
We consider any $\rho$-part of $\mathcal{S}$ (we drop the index $\rho$). If $\mathcal{S}'$ is a reordering of $\mathcal{S}$ and $S_i$ is an extended segment appearing in $\mathcal{S}$, we denote by $S_i'$ the same extended segment in its new position (with the new $\mu$-parameter) and \textbf{not} the $i$-th element in the corresponding sequence of $\mathcal{S}'$. Let us first assume the inequalities (\ref{ine}) hold for $\mathcal{S}_{\rho}=(([A_i,B_i],\mu_i))_{i=1}^n$. We will prove that $\mathcal{S}\in \mathrm{Rep}$ by induction over the number $n$. For $n=1$ there is nothing to show. For $n=2$, the inequality (\ref{ine}) is the same as (\ref{nec}), even after possibly reordering the two extended segments. 
Now let $\mathcal{S}'$ be any reordering of $\mathcal{S}$. Let $1\leq i<j\leq n$, such that $S_i'$ and $S_j'$ are consecutive in $\mathcal{S}'$. First, assume that for every $1\leq k <i$ the extended segment $S_k'$ comes before $S_i'$ in $\mathcal{S}_\rho'$ and that for every $j<k\leq n$ the extended segment $S_k'$ comes after $S_j'$ in $\mathcal{S}'_\rho$. We distinguish two cases:
\begin{enumerate}
    \item If there exists an index $k$ with $i<k<j$ such that $S_k'$ comes after $S_j'$ in $\mathcal{S}'_{\rho}$, let $k$ be maximal with this property. This implies that for $k<m \leq j$, the extended segment $S_m'$ comes before $S_i'$ and $S_k'$ in $\mathcal{S}'_{\rho}$, which means that 
\begin{align*}
    A_k\geq A_m.
\end{align*}
    Define the extended multi-segment 
    \begin{align*}
        \mathcal{S}'' \ceq R_{j-1}^\rho \circ \ldots R_k^\rho(\mathcal{S}),
    \end{align*}
    in which $S_k''$ appears immediately after $S_j''$. We claim that now the inequality (\ref{ine}) holds for any two connected extended segments in $\mathcal{S}''$ except $S_k''$. If both extended segments come before $S_k$ in $\mathcal{S}_\rho$, this is clear. If both come after $S_k$, all $\mu_l$ for $l$ from $k+1$ and $j$ are the same after being reordered with $S_k$, hence (\ref{ine}) holds for them as well. Now let $i'<k<j'$ such that $S_{i'}''$ and $S_{j'}''$ are connected in $\mathcal{S}''_\rho$. If $A_{i'}<A_{k}$, the corresponding segments are also connected in $\mathcal{S}$ and we have $\mu_m''=\mu_m'$ for $k+1 \leq m \leq j'$. Now 
    \begin{align*}
       & \Delta_{\mathcal{S}''}(\mu_{i'}'',\mu_{j'}'') =\\
       & = (-1)^{ | \{ k<k'<j'|\ A_{i'} \geq A_{k'} \} | }  \left(\sum_{m=i'}^{k-2} (-1)^{ | \{ m<k'<k|\ A_{i'} \geq A_{k'} \} | } (\mu_{m+1}-\mu_m)
        + (\mu_{k+1}-\mu_{k-1}) \right)\\
       & + \sum_{m=k+1}^{j'-1} (-1)^{ | \{ m<k'<j'|\ A_{i'} \geq A_{k'} \} | } (\mu_{m+1}-\mu_m) \\
        &= \Delta_{\mathcal{S}}(\mu_{i'},\mu_{j'}).
    \end{align*}
    If $A_{i'}\geq A_k \geq A_{j'}$, the extended segments $S_{i'}$ and $S_{j'}$ are not connected in $\mathcal{S}$. However, $S_{i'}\sim S_k$ and $S_k \sim S_{j'}$ in $\mathcal{S}$. Now we note 
    \begin{align*}
        |\Delta_{\mathcal{S}''}(\mu_{i'}'',\mu_{j'}'')| &\leq |\Delta_{\mathcal{S}}(\mu_{i'},\mu_{k})| + |\Delta_{\mathcal{S}}(\mu_{k},\mu_{j'})| \\
        &\leq (A_{i'}-A_k) + (B_k-B_{i'}) + (A_k - A_{j'}) + (B_{j'}-B_k) \\
        &=
(A_{i'}-A_{j'}) + (B_{j'}-B_{i'}).        
    \end{align*}
    By induction hypothesis, any reordering of $\mathcal{S}''$ that does not reorder $S_k''$ satisfies the necessary condition (\ref{nec}) for all consecutive pairs of extended segments that do not involve $S_k''$ (this follows by considering the sequence obtained by omitting the last element of $\mathcal{S}''$, namely $S_k''$, noting that this is standard and performing the necessary reorders). Through this, we can reorder $\mathcal{S}''$ (to $\mathcal{S}'''$) in such a way, that  $S_i'''$ and $S_j'''$ are consecutive and that all extended segments in $\mathcal{S}$ that are in between $S_i$ and $S_j$ are now on the same side of these two extended segments as in $\mathcal{S}'$. The extended segments $S_i'''$ and $S_j'''$ satisfy (\ref{nec}), by induction hypothesis. Now we can reorder from $\mathcal{S}'''$ to $\mathcal{S}'$ without performing any reorders involving $S_i'''$ and $S_j'''$. Since they satisfy the necessary condition (\ref{nec}) in $\mathcal{S}'''$, they do also in $\mathcal{S}'$, which was to be shown.  
    \item Now if no $k$ as above exists, we have that $S_k'$ comes before $S_i'$ in $\mathcal{S}_\rho$, for every $i<k<j$. If $j=i+1$, we are done. Otherwise, we denote 
    \begin{align*}
        \mathcal{S}'' \ceq R_i^\rho(\mathcal{S}),
    \end{align*}
    which moves $S_{i+1}''$ before $S_i''$. 
We claim again that the inequality (\ref{ine}) holds for any two connected extended segments in $\mathcal{S}''$ except $S_{i+1}''$. Again this is clear if both come after $S_{i+1}$, so we have to consider $S_i''$ and $S_k''$ that are connected. This implies $A_{i} \geq A_{i+1}$. If $ A_{i+1}<A_k$, the corresponding extended segments are also connected in $\mathcal{S}$ and we have $\mu_i''=2\mu_{i+1}-\mu_i$. Then
    \begin{align*}
        \Delta_{\mathcal{S}''}(\mu_{i}'',\mu_{k}'')=(-1)^{ | \{ i+1<k'<k|\ A_i \geq A_{k'} \} | } (\mu_{i+2}-2\mu_{i+1}+\mu_i) \\
        +\sum_{m=i+2}^{k-1} (-1)^{ | \{ m<k'<k|\ A_i \geq A_{k'} \} | } (\mu_{m+1}-\mu_m)=
        \Delta_{\mathcal{S}}(\mu_{i},\mu_{k}).
    \end{align*}
    If $A_i \geq A_{i+1}\geq A_k $, the extended segments $S_{i}$ and $S_{k}$ are not connected in $\mathcal{S}$. However, $S_{i}\sim S_{i+1}$ and $S_{i+1} \sim S_{k}$ in $\mathcal{S}$. Now we note 
    \begin{align*}
        |\Delta_{\mathcal{S}''}(\mu_{i}'',\mu_{k}'')| &\leq |\Delta_{\mathcal{S}}(\mu_{i},\mu_{i+1})| + |\Delta_{\mathcal{S}}(\mu_{i+1},\mu_{k})| \\
        &\leq (A_{i}-A_{i+1}) + (B_{i+1}-B_{i}) + (A_{i+1} - A_{k}) + (B_{k}-B_{i+1}) \\
        &=
(A_{i}-A_{k}) + (B_{k}-B_{i}).        
    \end{align*} 
    The remaining argument is the same as in the previous case. 
\end{enumerate}
We still have to treat the case, where the extended segments not between $S_i$ and $S_j$ are not all on the same side of $S_i$ and $S_j$ in $\mathcal{S}$ as of $S_i'$ and $S_j'$ in $\mathcal{S}'$. However, as described above, we can find a reordering $\mathcal{S}''$ where $S_i''$ and $S_j''$ are consecutive, satisfy (\ref{nec}) and the extended segments between $S_i$ and $S_j$ are on the correct side. To reorder from $\mathcal{S}''$ to $\mathcal{S}'$, we have to reorder some extended segments through both $S_i''$ and $S_j''$. Under any such reorder (and by possibly exchanging $S_i''$ and $S_j''$), the difference
    \begin{align*}
        |\mu_i''-\mu_j''|
    \end{align*}
    is invariant (to see this, note that we assumed $\mathcal{S}$ to be standard, so if $S_k''$ comes after $S_i''$ and $S_j''$ in $\mathcal{S}''$ and $S_k$ is not between $S_i$ and $S_j$ in $\mathcal{S}$, the only case in which we can reorder it to come before $S_i''$ and $S_j''$ is if $[A_k,B_k]\subseteq [A_i,B_i]$ and $[A_k,B_k]\subseteq [A_j,B_j]$, which implies that both reorders are of the same "type": they are replaced by $2\mu_k''-\mu_i''$ and $2\mu_k''-\mu_j''$. The other case follows similarly), so the necessary condition (\ref{nec}) follows for $\mathcal{S}'$. 

    For the opposite direction, the above procedure shows that if $S_i^\rho$ and $S_j^\rho$ are connected, there is a reordering $\mathcal{S}'$ of $\mathcal{S}$ where the corresponding extended multisegments are consecutive and 
    \begin{align*}
        \mu_j'-\mu_i'= \Delta_{\mathcal{S}}(\mu_{i},\mu_{j}).
    \end{align*}
\end{proof}

\subsection{The Aubert dual and deformation}\label{equivalences}

Atobe describes further constructions in terms of extended multi-segments. We translate them into our slightly altered definition. In \cite{aubert}, Aubert defines an involution on $\mathrm{Irr}(G_n)$, known as the \textit{Aubert dual}. As outlined in \cite[Section 6]{A1} the Aubert dual can be defined as follows:
\begin{de}
    Let $\pi$ be an irreducible representation of $G_n$. There exists a sign $\varepsilon \in \{ \pm 1\}$ such that the Aubert dual 
    \begin{align*}
        \Hat{\pi} = \varepsilon \sum_P (-1)^{\mathrm{dim}A_M} [ \mathrm{Ind}_P^{G_n} (\mathrm{Jac}_P(\pi))]
    \end{align*}
    is again an irreducible representation of $G_n$, where $P=MN$ runs over all standard parabolic subgroups of $G_n$ and $A_M$ is the maximal split torus of the center of $M$. 
\end{de}
Given a representation $\pi(\mathcal{E})$ of Arthur type, Atobe gives a formula for an extended multi-segment $\Hat{\mathcal{E}}$ such that $\pi(\Hat{\mathcal{E}})$ is the Aubert dual of $\pi(\mathcal{E})$. We will recall this formula in our notation (i.e. we describe $\Hat{\mathcal{S}} = \mathcal{F}^{-1}(\Hat{\mathcal{E}}) $): 
\begin{de}
    Let $\mathcal{S}= \bigcup_{\rho \in C_\mathcal{S}} \{(( [A_i^\rho,B_i^\rho]_\rho,\mu_i^\rho ))_{i=1}^{n_\rho}\} \in \mathrm{AdExMult}$. Define
    \begin{align*}
        \Hat{\mathcal{S}} \ceq \bigcup_{\rho \in C_\mathcal{S}} \{ (([A_{n_\rho-i+1}^\rho,-B_{n_\rho-i+1}^\rho]_\rho,\Hat{\mu}_{n_\rho-i+1}^\rho ))_{i=1}^{n_\rho}\},
    \end{align*}
    where 
    \begin{align*}
        \Hat{\mu_i}^\rho = \left \{ 
        \begin{matrix}
            \mu_i \qquad &if\ B_i^\rho \in \Z, \\
            \mu_i -1 \qquad &if\ B_i^\rho \not\in \Z. 
        \end{matrix}
        \right.
    \end{align*}
\end{de}

\begin{thm} (\cite[Theorem 6.2]{A1}) \\
    If $\mathcal{S}\in \mathrm{Rep}$, its Aubert dual is given by
    \begin{align*}
        \widehat{\pi(\mathcal{S})} \cong \pi(\Hat{\mathcal{S}}).
    \end{align*}
\end{thm}

Atobe gives another equivalence relation of extended multisegments, that leaves the emerging representation invariant. It is called \textit{deformation} and described in \cite[Section 5.2]{A1}. We translate it into our definition: \\
Let $\mathcal{S} \in \mathrm{Rep}$ be an extended multi-segment. Let $S_{k-1}=([A_{k-1},B_{k-1}]_\rho,\mu_{k-1})$ and $S_k=([A_k,B_k]_\rho,\mu_k)$ be consecutive extended segments of $\mathcal{S}$ where $A_k \geq A_{k-1}$ and $B_k \geq B_{k-1}$ such that 
\begin{align*}
    |A_k- A_{k-1}| + |B_k- B_{k-1}| = |\mu_k- \mu_{k-1}|. 
\end{align*}
Replace them by 
\begin{align*}
    S_{k-1}' \ceq ([A_k,B_{k-1}]_\rho,\mu_{k-1} -\mathrm{sgn}(\mu_k-\mu_{k-1}) (A_k- A_{k-1}))
\end{align*}
and 
\begin{align*}
    S_k' \ceq ([A_{k-1},B_k]_\rho,\mu_{k} -\mathrm{sgn}(\mu_k-\mu_{k-1})  (A_k- A_{k-1}))
\end{align*}
to define $\mathcal{S}'$ (with the convention $\mathrm{sgn}(0)=1$, although it is irrelevant). By \cite[Theorem 3.5]{A3} it follows that for $\mathcal{S}\in \mathrm{Rep}$, we have $\pi(\mathcal{S})\cong \pi(\mathcal{S}')$.

\nopagebreak[4]

\end{document}